\documentclass[10pt,a4paper]{article}
\usepackage[english]{babel}
\usepackage[utf8]{inputenc}
\usepackage{amsmath,amsthm}
\usepackage{amsfonts}
\usepackage{amssymb}
\usepackage[T1]{fontenc}
\usepackage{graphicx}
\usepackage{fancyhdr}
\usepackage[left=2cm,right=2cm,top=5cm,bottom=5cm]{geometry}
\usepackage[usenames,dvipsnames]{pstricks}
\usepackage{epsfig}
\usepackage{pst-grad}
\usepackage{pst-plot}
\usepackage[tight,TABTOPCAP]{subfigure}

\newtheorem{mainth1}{Theorem}[section]
\newtheorem{mainth}[mainth1]{Theorem}

\newtheorem{propcycle}{Property}[section]
\newtheorem{defreleq}[propcycle]{Definition}
\newtheorem{defreleqtop}[propcycle]{Definition}

\newtheorem{defrmap}[propcycle]{Definition}

\newtheorem{thenmap}{Theorem}[section]
\newtheorem{deflrmap}[thenmap]{Definition}
\newtheorem{thenmapnr}[thenmap]{Theorem}
\newtheorem{cortree}[thenmap]{Corollary}
\newtheorem{lemvalbr}[thenmap]{Lemma}
\newtheorem{thlbinv}[thenmap]{Theorem}
\newtheorem{thenpoly}[thenmap]{Theorem}
\newtheorem{cortreeg}[thenmap]{Corollary}

\begin{document}

\renewcommand{\proofname}{\textbf{Proof}}

\begin{center}
\textbf{\large{Topological enumeration of complex polynomial vector fields}}
\end{center}

\begin{center}
\textbf{J. Tomasini, \today}
\end{center}

\begin{flushleft}
\textbf{Abstract.} The enumeration of combinatorial classes of the complex polynomial vector fields in $\mathbb{C}$ presented in \cite{Dias2} is extended here to a closed form enumeration of combinatorial classes for degree $d$ polynomial vector fields up to rotations of $2(d-1)^{st}$ roots of unity. The main tool in the proof of this result is based on a general method of enumeration developed by V.A.Liskovets \cite{Lisk}.
\end{flushleft}

\section{Introduction}

In this work, our study is based on the equation
\begin{equation}
\label{poldiffeq}
\dot{z} := \frac{dz}{dt} = Q(z), \quad t\in \mathbb{R}, \quad z \in \mathbb{C},
\end{equation}
where $Q$ is a complex polynomial of degree $d \geq 2$. The set of trajectories of this equation partitions the complex plane following a particular pattern describing by a simple combinatorial model. The main goal of this article is to give a complete classification and enumeration of the topological structure of a vector field defined by a complex polynomial of degree $d$. \newline
In the following, $\gamma(t,z)$ (or simply $\gamma(t)$ if the initial condition does not matter) will denote the solution of equation (\ref{poldiffeq}) passing through the point $z$, and we will define for all complex polynomial $P$
\begin{displaymath}
\xi_P(z) = P(z) . \frac{d}{dz},
\end{displaymath}
the polynomial vector field defined by $P$. \newline
At first, as we are only interested in the topological structure of polynomial vector fields, we can consider that $\xi_P$ and $\xi_{-P}$ are topologically equivalent. In another way, it's easy to see that we can restrict our study to the set of monic, centered complex polynomial vector fields (i.e. the vector fields defined by a polynomial of the form $X^d + a_{d-2} X^{d-2} + \ldots + a_0$), in the sense that there exists an affine change of variable (uniquely determined by the multiplication of a $(d-1)^{st}$ root of unity) that changes the vector field defined by a complex polynomial $Q$ of degree $d$ to the vector field defined by a monic, centered complex polynomial $P$ of degree $d$. Moreover, this change of variable does not change the topological structure of $\xi_Q$, i.e. $\xi_Q$ and $\xi_P$ are topologically equivalent. From now we denote by $\mathcal{P}_d \simeq \mathbb{C}^{d-1}$ the set of monic, centered complex polynomial of degree $d$.

A recent study for the global classification of complex polynomial vector fields in $\mathbb{C}$, made first by A.Douady, F.Estrada and P.Sentenac \cite{DES} in the structurally stable (or generic) case, i.e. for polynomial vector fields such that there are neither homoclinic separatrices nor multiple equilibrium points, and then completed by B.Branner and K.Dias \cite{Dias1} for the general case, proved that a complex polynomial vector field in $\mathbb{C}$ is entirely determined by a combinatorial model (describing the topology of the vector field) and an analytic data set (describing the geometry of the vector field). More exactly, these two works show that for two given invariants (combinatorial and analytic), there exists a unique polynomial vector field defined by a monic, centered polynomial of fixed degree. In parallel to the work of A.Douady, F.Estrada and P.Sentenac, K.M.Pilgrim gives in his article \cite{Pilg} a method to construct explicitly a generic, monic, centered polynomial $P$ from any given combinatorial and a specific analytic data set.

We are often compelled in this article to consider the two situations, generic and non-generic, of polynomial vector field separately in order to simplify the presentation of the combinatorial models.

\vspace{2mm}

A natural question, appearing in the following of these works, is to count the different combinatorial models coming from polynomial vector fields of fixed degree, which is equivalent to classify our vector fields up to their topological structures. A part of this modeling and enumeration work is already done by A.Douady, F.Estrada and P.Sentenac in the generic case, and then enhanced by K.Dias in \cite{Dias2} for the non-generic case, but these two works do not exactly answer to the original problem in the sense that their modellisation (and so a fortiori their enumerations) show a distinction between two polynomial vector fields which are conformally conjugate by a rotation. This distinction is due to the fact that the modelisations proposed in \cite{DES} and \cite{Dias2} impose an "order" (symbolized by a marking in the combinatorial model) which is not specific to the vector fields.

The main contribution of this work is to find appropriate combinatorial models to enable the counting of the equivalent classes of polynomial vector fields without markings.

\vspace{2mm}

The contents of this paper are as follow. In Section 2, we introduce main notions and basic facts about polynomial vector fields and we define two equivalence relations whose equivalence classes are the ones we want to count. In Section 3, we present four types of combinatorial models describing the topological structure of a polynomial vector field. The first three are combinatorial objects that already exist in the article \cite{DES} for the generic case, and \cite{Dias2} for non-generic case, but are necessary to introduce the last combinatorial model that fits our problem. Finally, in Section 4, we first give a general method, developed by V.A.Liskovets (see for example \cite{Lisk}), in order to enumerate a set of unrooted planar maps. Thanks to this method, we will be able to prove the main results of this article:

\begin{mainth1}
\label{mainth1}
The number $\sigma_n$ of topological structures of generic complex polynomial vector fields of degree $d=n+1$ is
\begin{displaymath}
\sigma_n = \frac{1}{2n} \left[ \frac{1}{n+1} \left( \! \begin{array}{c} 2n \\ n \end{array} \! \right) + \sum_{l \geq 2 \atop l|n}{\varphi(l) \left( \! \begin{array}{c} 2n/l \\ n/l \end{array} \! \right)} + \left\{ \begin{array}{cc}
\left( \! \begin{array}{c} n \\ \frac{n-1}{2} \end{array} \! \right) & \text{if } n \text{ is odd} \\
& \\
0 & \text{if } n \text{ is even}
\end{array} \right. \right] ,
\end{displaymath}
where $\phi$ is the Euler totient function. Moreover, \begin{displaymath}
\lim_{n \rightarrow +\infty}{(\sigma_n)^{1/n}} = 4.
\end{displaymath}
\end{mainth1}

\begin{mainth}
\label{mainth}
The number $p_n^{+}$ of topological structure of complex polynomial vector fields of degree $d=n+1$ is

\begin{displaymath}
p^+_n = \frac{1}{2n} \left[ p_n + \sum_{l \geq 0 \atop l|2n} \varphi(l). \left\{ \begin{array}{cc}
\sum\limits_{k \geq 0}{p_{n/l,k}.(k+1)} & \text{ if } 2n/l \text{ even} \\
c_{2n/l} & \text{ if } 2n/l \text{ odd}
\end{array} \right. + \left\{ \begin{array}{cc}
0 & \text{ if } n \text{ even} \\
2n. p_{(n-1)/2} & \text{ if } n \text{ odd}
\end{array} \right. \right], \quad n \geq 1,
\end{displaymath}
with 
\begin{displaymath}
p_{n,k} = \frac{(-2n)_k (-n)_k}{(2)_k} \frac{2^k}{k!}, \qquad p_n = \sum_{k \geq 0}{p_{n,k}},
\end{displaymath}
and
\begin{eqnarray*}
c_{2m+1} & = & -\sum_{k=0}^{m} a_k \left( \sum_{i=0}^{m-k} \alpha^{i} \beta^{m-k-i} \right), \\
c_{2m} & = & c_{2m+1} - 4 \sum_{j=0}^{m-1} c_{2j+1} p_{m-j-1},
\end{eqnarray*}
where $\alpha$ and $\beta$ are the roots of the polynomial $X^2 +11X -1$, and
\begin{displaymath}
a_0 = -1 \qquad a_1 = 4,
\end{displaymath}
\begin{displaymath}
a_n = 6 \sum_{k=0}^{n-2} p_k p_{n-2-k} - 2 \sum_{k=0}^{n-2} p_k p_{n-1-k} + p_n, \quad \text{for } n \geq 2.
\end{displaymath}
Moreover, \begin{displaymath}
\lim_{n \rightarrow +\infty}{(p^+_n)^{1/n}} = \frac{2}{5\sqrt{5} - 11}.
\end{displaymath}
\end{mainth}

The proof of this second theorem will be in three parts. In subsection $4.3$, we explain and give an explicit formula for the number $p_n$, then in subsection $4.4$, we prove Theorem \ref{thenpoly}, which is similar to Theorem \ref{mainth} but with a recursive formula of the numbers $c_n$. Finally, in the last subsection, we give a closed formula for $c_n$.

\vspace{3mm}

\textbf{Acknowledgements.} I am grateful to Tan Lei for introducing the subject, useful conversations and a careful reading of this work during its preparation.

\section{Basic facts about polynomial vector fields.}

Now we give a quick description of the main notions and definitions necessary to understand and to describe the combinatorial data set used in this article. For details about these results, see \cite{DES} and \cite{Dias1}.

A fundamental property of polynomial vector fields, that make their study simpler, is that they don't have a limit cycle. It's a direct consequence of the following property:

\begin{propcycle}
Any accumulation point of a non-periodic orbit of $\xi_P$ is a root of the polynomial $P$.
\end{propcycle}

In other words, for all non-periodic maximal solution $\gamma(t,z)$ of $\xi_P$ defined on the interval $]t_{min} , t_{max}[$, 
\begin{itemize}
\item if $t_{min} = -\infty$, $\gamma(t,z)$ tends to an equilibrium point of $\xi_P$ as $t$ tends to $-\infty$.
\item if $t_{max} = +\infty$, $\gamma(t,z)$ tends to an equilibrium point of $\xi_P$ as $t$ tends to $+\infty$.
\item if $t_{min}$ (resp. $t_{max}$) is a finite number, $\gamma(t,z)$ tends to infinity as $t$ tends to $t_{min}$ (resp. $t_{max}$).
\end{itemize}

\vspace{2mm}

Given $P \in \mathcal{P}_d$, let $\zeta$ be a root of the polynomial $P$. Then the vector field $\xi_P$ associated to $P$ admits an equilibrium point (or singularity) at the point $\zeta$, and this singularity can be of four different types:

\begin{itemize}
\item $\zeta$ is a source if $Re(P'(\zeta)) > 0$.
\item $\zeta$ is a sink if $Re(P'(\zeta)) < 0$.
\item $\zeta$ is a center if $Re(P'(\zeta)) = 0$ and $Im(P'(\zeta)) \neq 0$.
\item $\zeta$ is a multiple equilibrium point of multiplicity $m \geq 2$ if $P'(\zeta) = \ldots = P^{(m-1)}(\zeta) = 0$ and $P^{(m)}(\zeta) \neq 0$.
\end{itemize}

Each type of singularities influences the global structure of the vector field in that a given singularity $\zeta$ determines the behavior of the solutions passing through a neighborhood of it. This zone of influence is called \textbf{basin}, denoted by $\mathcal{B}(\zeta)$, and is defined as follows:

\begin{itemize}
\item If $\zeta$ is a source, $\mathcal{B}(\zeta) = \{ z \in \mathbb{C} \mid  \gamma(t,z) \rightarrow \zeta \text{ for } t \rightarrow -\infty \}$.
\item If $\zeta$ is a sink, $\mathcal{B}(\zeta) = \{ z \in \mathbb{C} \mid  \gamma(t,z) \rightarrow \zeta \text{ for } t \rightarrow +\infty \}$.
\item If $\zeta$ is a center, $\mathcal{B}(\zeta) = \{ \zeta \} \cup \{ z \in \mathbb{C} \mid  \gamma(.,z) \text{ is periodic and } \zeta \text{ is in the bounded component of } \mathbb{C} \setminus \gamma(\mathbb{R},z) \}$.
\item If $\zeta$ is a multiple equilibrium point, $\mathcal{B}(\zeta) = \mathcal{B}_{\alpha}(\zeta) \cup \mathcal{B}_{\omega}(\zeta) \cup \{ \zeta \}$, where
\begin{eqnarray*}
\mathcal{B}_{\alpha}(\zeta) & = &  \{ z \neq \zeta \mid \gamma(t,z) \rightarrow \zeta \text{ for } t \rightarrow -\infty \} \quad \text{is the repelling basin} \\
\mathcal{B}_{\omega}(\zeta) & = &  \{ z \neq \zeta \mid \gamma(t,z) \rightarrow \zeta \text{ for } t \rightarrow +\infty \} \quad \text{is the attracting basin}
\end{eqnarray*}
\end{itemize}

The connected components of $\mathcal{B}_{\alpha}(\zeta)$ and $\mathcal{B}_{\omega}(\zeta)$ are called repelling petals and attracting petals respectively. In all four cases, $\mathcal{B}(\zeta)$ is an open, simply-connected domain containing $\zeta$. See Figure \ref{figexpoly} for an example showing the four different types of singularity.

\begin{figure}[ht]
\begin{center}
\scalebox{0.5} 
{
\begin{pspicture}(0,-7.078052)(21.02,7.078052)
\psdots[dotsize=0.2](18.88,-1.2780521)
\psdots[dotsize=0.2](16.42,0.6219479)
\psdots[dotsize=0.2](17.98,2.241948)
\psdots[dotsize=0.2](12.3,-0.15805209)
\psdots[dotsize=0.2](8.58,-3.7380521)
\psdots[dotsize=0.2](4.64,-2.1980522)
\psdots[dotsize=0.2](3.06,-4.038052)
\psdots[dotsize=0.2](6.54,3.7819479)
\psbezier[linewidth=0.04,linestyle=dashed,dash=0.16cm 0.16cm](1.26,5.261948)(1.8,4.801948)(6.6041026,1.0837426)(7.92,2.361948)(9.235898,3.6401532)(7.42,5.741948)(6.96,6.521948)
\psbezier[linewidth=0.04,linestyle=dashed,dash=0.16cm 0.16cm](4.66,-2.2180521)(5.38,-2.6980522)(5.74795,-3.400651)(5.74,-4.158052)(5.73205,-4.915453)(5.44,-5.698052)(4.8,-6.958052)
\psbezier[linewidth=0.04,linestyle=dashed,dash=0.16cm 0.16cm](4.6,-2.1780522)(3.8,-1.6780521)(3.2983806,-1.4635454)(2.32,-1.1180521)(1.3416193,-0.77255875)(0.86,-0.8980521)(0.0,-0.8580521)
\psbezier[linewidth=0.04,linestyle=dashed,dash=0.16cm 0.16cm](3.02,-4.058052)(2.26,-4.878052)(2.4,-4.778052)(1.24,-5.8580523)
\psbezier[linewidth=0.04,linestyle=dashed,dash=0.16cm 0.16cm](8.58,-3.6780522)(8.88,-4.058052)(9.363518,-4.788908)(9.5,-5.118052)(9.636482,-5.447196)(9.72,-6.3180523)(9.64,-7.058052)
\psbezier[linewidth=0.04,linestyle=dashed,dash=0.16cm 0.16cm](0.12,1.9819479)(0.8,1.7419479)(4.824745,0.99181074)(6.14,0.4419479)(7.455255,-0.107914925)(10.2,-1.8780521)(10.9,-2.7180521)(11.6,-3.558052)(13.34,-6.198052)(13.54,-7.058052)
\psbezier[linewidth=0.04,linestyle=dashed,dash=0.16cm 0.16cm](12.32,-0.15805209)(12.76,-0.7980521)(14.403013,-2.5841496)(15.06,-3.338052)(15.716987,-4.0919547)(17.22,-6.078052)(18.18,-6.878052)
\psbezier[linewidth=0.04,linestyle=dashed,dash=0.16cm 0.16cm](12.28,-0.15805209)(11.6,0.9819479)(10.8,2.201948)(10.56,3.181948)(10.32,4.1619477)(9.96,5.541948)(9.9,6.521948)
\psbezier[linewidth=0.04,linestyle=dashed,dash=0.16cm 0.16cm](16.42,0.6219479)(16.66,-0.09805209)(16.848816,-1.1985806)(17.38,-2.078052)(17.911184,-2.9575236)(19.34,-4.098052)(20.24,-4.918052)
\psbezier[linewidth=0.04,linestyle=dashed,dash=0.16cm 0.16cm](16.42,0.6419479)(16.34,1.681948)(16.090956,2.7089837)(16.5,3.7819479)(16.909044,4.8549123)(17.68,5.7019477)(18.5,6.321948)
\psbezier[linewidth=0.04,linestyle=dashed,dash=0.16cm 0.16cm](18.0,2.241948)(19.06,3.5019479)(19.24,3.5419478)(20.48,4.081948)
\psbezier[linewidth=0.04,linestyle=dashed,dash=0.16cm 0.16cm](18.86,-1.2580521)(19.78,-1.2780521)(20.2,-1.2780521)(21.0,-1.2980521)
\psbezier[linewidth=0.04,linestyle=dashed,dash=0.16cm 0.16cm](16.4,0.6219479)(17.26,0.5819479)(19.72,0.8619479)(20.8,1.301948)
\psbezier[linewidth=0.04,linestyle=dashed,dash=0.16cm 0.16cm](12.3,-0.09805209)(12.84,0.4819479)(13.976947,1.7463572)(14.46,2.621948)(14.943053,3.4975386)(15.66,5.321948)(15.82,6.501948)
\psbezier[linewidth=0.04](12.3,-0.13805209)(13.36,0.3219479)(14.62,2.461948)(15.1,2.461948)(15.58,2.461948)(16.04,1.6019479)(16.4,0.6819479)
\psbezier[linewidth=0.04](12.3,-0.16326536)(13.208571,-0.7911326)(15.389143,-3.538052)(15.954476,-3.3418436)(16.51981,-3.1456351)(16.54,-0.7911326)(16.418858,0.6019479)
\psbezier[linewidth=0.04](12.26,-0.13805209)(13.46,-0.1980521)(15.68,0.121947914)(16.42,0.6419479)
\psbezier[linewidth=0.04](16.42,0.6619479)(17.54,0.82194793)(19.78,0.7819479)(19.84,1.9019479)(19.9,3.0219479)(18.92,2.901948)(17.98,2.2619479)
\psbezier[linewidth=0.04](16.4,0.6419479)(16.52,1.4819479)(16.48,4.341948)(17.64,4.341948)(18.8,4.341948)(18.64,3.3219478)(17.98,2.2819479)
\psbezier[linewidth=0.04](16.42,0.6419479)(17.46,1.2419479)(17.56,1.4419479)(17.96,2.2619479)
\psbezier[linewidth=0.04](16.42,0.6019479)(16.72,-0.09805209)(18.54,-4.138052)(19.68,-3.098052)(20.82,-2.058052)(20.06,-1.318052)(18.88,-1.2780521)
\psbezier[linewidth=0.04](16.46,0.6419479)(17.44,0.3819479)(19.613354,0.9793113)(20.06,0.021947911)(20.506647,-0.93541545)(19.96,-1.1580521)(18.94,-1.2380521)
\psbezier[linewidth=0.04](16.4,0.6219479)(16.82,-0.13805209)(18.1,-1.1580521)(18.84,-1.2780521)
\psbezier[linewidth=0.04](12.28,-0.11805209)(13.14,1.061948)(15.776803,5.305844)(15.1,6.1819477)(14.423197,7.058052)(11.7,6.921948)(10.7,6.221948)(9.7,5.521948)(11.6,1.1419479)(12.26,-0.15805209)
\psbezier[linewidth=0.04](12.28,-0.15805209)(12.58,0.76194793)(14.115043,3.8294706)(13.6,4.341948)(13.084958,4.8544254)(12.06,4.941948)(11.42,4.381948)(10.78,3.8219478)(11.92,0.9219479)(12.3,-0.13805209)
\psbezier[linewidth=0.04](12.26,-0.11805209)(11.28,0.9219479)(10.019933,3.9135287)(9.28,3.901948)(8.540067,3.8903673)(9.2,2.641948)(8.62,2.0819478)(8.04,1.5219479)(6.2308245,1.7933534)(6.14,1.2219479)(6.0491757,0.6505424)(9.880507,-1.1448877)(10.62,-1.818052)(11.359493,-2.4912164)(13.140292,-5.562238)(13.74,-5.598052)(14.339707,-5.6338663)(15.05788,-4.39805)(15.04,-3.858052)(15.0221195,-3.3180544)(13.12,-1.5580521)(12.3,-0.15805209)
\psbezier[linewidth=0.04](12.26,-0.13805209)(11.66,0.1619479)(10.46,1.801948)(9.74,1.661948)(9.02,1.5219479)(9.18784,0.45317638)(9.26,-0.07805209)(9.332159,-0.6092805)(11.56,-2.278052)(12.04,-2.578052)(12.52,-2.878052)(13.869439,-3.6968293)(13.9,-3.058052)(13.930561,-2.419275)(12.62,-0.9780521)(12.28,-0.1980521)
\psbezier[linewidth=0.04](5.58,4.601948)(6.262626,5.332716)(7.4715185,4.9154963)(7.7,3.941948)(7.9284816,2.9683998)(7.31611,2.4246726)(6.34,2.641948)(5.3638897,2.8592234)(4.897374,3.87118)(5.58,4.601948)
\psbezier[linewidth=0.04](4.84,6.1819477)(5.8076634,6.4141936)(7.407258,5.1094813)(7.86,4.2019477)(8.312742,3.2944143)(7.48,2.2819479)(6.54,2.3419478)(5.6,2.401948)(3.7730982,4.053737)(3.72,4.801948)(3.6669018,5.550159)(3.8723369,5.9497023)(4.84,6.1819477)
\psbezier[linewidth=0.04](8.58,-3.6780522)(9.34,-4.118052)(10.12,-6.018052)(10.82,-5.878052)(11.52,-5.738052)(11.6012535,-4.7825227)(11.24,-3.838052)(10.878746,-2.8935816)(8.484328,-1.1449184)(7.58,-0.6580521)(6.675672,-0.17118572)(4.0502315,1.0041002)(3.36,0.30194792)(2.6697688,-0.40020442)(3.58,-1.3980521)(4.6,-2.1780522)
\psbezier[linewidth=0.04](8.56,-3.7380521)(8.92,-4.518052)(9.48,-6.218052)(8.42,-6.438052)(7.36,-6.658052)(6.54,-5.598052)(6.28,-4.398052)(6.02,-3.1980522)(5.48,-2.518052)(4.62,-2.2180521)
\psbezier[linewidth=0.04](8.54,-3.7180521)(8.3,-3.1580522)(7.7793794,-2.451424)(7.08,-2.138052)(6.3806205,-1.8246802)(5.7,-1.6780521)(4.62,-2.1780522)
\psbezier[linewidth=0.04](4.64,-2.1780522)(5.26,-2.878052)(5.5526476,-4.587823)(5.06,-5.398052)(4.5673523,-6.208281)(3.3657815,-6.901153)(2.6,-6.238052)(1.8342185,-5.574951)(2.61938,-4.9557457)(3.06,-4.058052)
\psbezier[linewidth=0.04](4.6,-2.1780522)(3.68,-1.8980521)(1.8,-0.8580521)(1.12,-1.3780521)(0.44,-1.8980521)(0.2925134,-4.1417203)(1.04,-4.718052)(1.7874866,-5.294384)(2.22,-4.638052)(3.06,-4.018052)
\psbezier[linewidth=0.04](3.06,-4.018052)(3.96,-3.318052)(4.24,-2.878052)(4.64,-2.1780522)
\psline[linewidth=0.04cm](20.38,-1.078052)(20.58,-1.2780521)
\psline[linewidth=0.04cm](20.58,-1.2780521)(20.38,-1.438052)
\psline[linewidth=0.04cm](19.74,1.2019479)(19.62,0.9619479)
\psline[linewidth=0.04cm](19.62,0.9619479)(19.8,0.8019479)
\psline[linewidth=0.04cm](20.02,-0.3980521)(20.14,-0.15805209)
\psline[linewidth=0.04cm](20.14,-0.15805209)(20.36,-0.3380521)
\psline[linewidth=0.04cm](19.84,-2.598052)(19.92,-2.838052)
\psline[linewidth=0.04cm](19.94,-2.838052)(20.2,-2.798052)
\psline[linewidth=0.04cm](19.6,2.2619479)(19.84,2.101948)
\psline[linewidth=0.04cm](19.84,2.101948)(20.04,2.2819479)
\psline[linewidth=0.04cm](19.28,3.741948)(19.56,3.661948)
\psline[linewidth=0.04cm](19.56,3.661948)(19.44,3.421948)
\psline[linewidth=0.04cm](18.06,4.461948)(17.78,4.321948)
\psline[linewidth=0.04cm](17.78,4.321948)(18.02,4.101948)
\psline[linewidth=0.04cm](16.9,5.001948)(16.98,4.741948)
\psline[linewidth=0.04cm](16.98,4.741948)(17.26,4.781948)
\psline[linewidth=0.04cm](14.74,3.5819478)(14.96,3.681948)
\psline[linewidth=0.04cm](14.96,3.681948)(15.04,3.461948)
\psline[linewidth=0.04cm](13.1,7.021948)(12.86,6.781948)
\psline[linewidth=0.04cm](12.86,6.781948)(13.08,6.581948)
\psline[linewidth=0.04cm](12.7,4.981948)(12.48,4.761948)
\psline[linewidth=0.04cm](12.48,4.761948)(12.7,4.541948)
\psline[linewidth=0.04cm](10.18,3.741948)(10.44,3.621948)
\psline[linewidth=0.04cm](10.44,3.621948)(10.58,3.8419478)
\psline[linewidth=0.04cm](14.04,1.7419479)(14.3,1.7819479)
\psline[linewidth=0.04cm](14.3,1.7819479)(14.3,1.541948)
\psline[linewidth=0.04cm](14.54,0.2419479)(14.78,0.08194791)
\psline[linewidth=0.04cm](14.78,0.08194791)(14.58,-0.11805209)
\psline[linewidth=0.04cm](14.78,-2.358052)(14.82,-2.618052)
\psline[linewidth=0.04cm](14.82,-2.618052)(14.56,-2.638052)
\psline[linewidth=0.04cm](15.9,-4.058052)(15.88,-4.3380523)
\psline[linewidth=0.04cm](15.88,-4.3380523)(15.6,-4.3180523)
\psline[linewidth=0.04cm](8.98,0.4819479)(9.18,0.7219479)
\psline[linewidth=0.04cm](9.18,0.7219479)(9.4,0.50194794)
\psline[linewidth=0.04cm](8.3,2.161948)(8.62,2.0619478)
\psline[linewidth=0.04cm](8.62,2.0619478)(8.6,1.7619479)
\psline[linewidth=0.04cm](14.42,-4.778052)(14.52,-5.078052)
\psline[linewidth=0.04cm](14.52,-5.078052)(14.82,-5.038052)
\psline[linewidth=0.04cm](3.92,3.641948)(3.96,3.361948)
\psline[linewidth=0.04cm](3.96,3.361948)(3.72,3.2819479)
\psline[linewidth=0.04cm](7.76,5.001948)(7.8,5.241948)
\psline[linewidth=0.04cm](7.8,5.241948)(8.04,5.1419477)
\psline[linewidth=0.04cm](6.1,6.061948)(5.86,6.021948)
\psline[linewidth=0.04cm](5.86,6.021948)(5.9,5.761948)
\psline[linewidth=0.04cm](5.44,4.721948)(5.46,4.461948)
\psline[linewidth=0.04cm](5.46,4.461948)(5.74,4.481948)
\psline[linewidth=0.04cm](4.72,1.1219479)(4.52,0.94194794)
\psline[linewidth=0.04cm](4.52,0.94194794)(4.66,0.6819479)
\psline[linewidth=0.04cm](12.6,-5.538052)(12.66,-5.278052)
\psline[linewidth=0.04cm](12.66,-5.278052)(12.94,-5.3580523)
\psline[linewidth=0.04cm](8.26,-1.078052)(8.56,-1.0380521)
\psline[linewidth=0.04cm](8.26,-1.078052)(8.34,-1.3780521)
\psline[linewidth=0.04cm](7.34,-2.518052)(7.26,-2.2180521)
\psline[linewidth=0.04cm](7.26,-2.2180521)(7.6,-2.1780522)
\psline[linewidth=0.04cm](9.48,-6.278052)(9.68,-6.458052)
\psline[linewidth=0.04cm](9.68,-6.458052)(9.88,-6.2980523)
\psline[linewidth=0.04cm](8.2,-6.278052)(8.06,-6.458052)
\psline[linewidth=0.04cm](8.06,-6.458052)(8.22,-6.638052)
\psline[linewidth=0.04cm](5.18,-5.698052)(5.42,-5.638052)
\psline[linewidth=0.04cm](5.42,-5.638052)(5.5,-5.8580523)
\psline[linewidth=0.04cm](3.38,-6.238052)(3.56,-6.478052)
\psline[linewidth=0.04cm](3.56,-6.478052)(3.4,-6.698052)
\psline[linewidth=0.04cm](0.3,-3.538052)(0.54,-3.298052)
\psline[linewidth=0.04cm](0.54,-3.298052)(0.78,-3.538052)
\psline[linewidth=0.04cm](1.78,-5.118052)(1.78,-5.3580523)
\psline[linewidth=0.04cm](1.78,-5.3580523)(2.04,-5.3380523)
\psline[linewidth=0.04cm](1.48,-0.6380521)(1.62,-0.9180521)
\psline[linewidth=0.04cm](1.62,-0.9180521)(1.36,-1.1180521)
\end{pspicture} 
}
\end{center}
\caption{Example of a polynomial vector field of degree $9$.}
\label{figexpoly}
\end{figure}
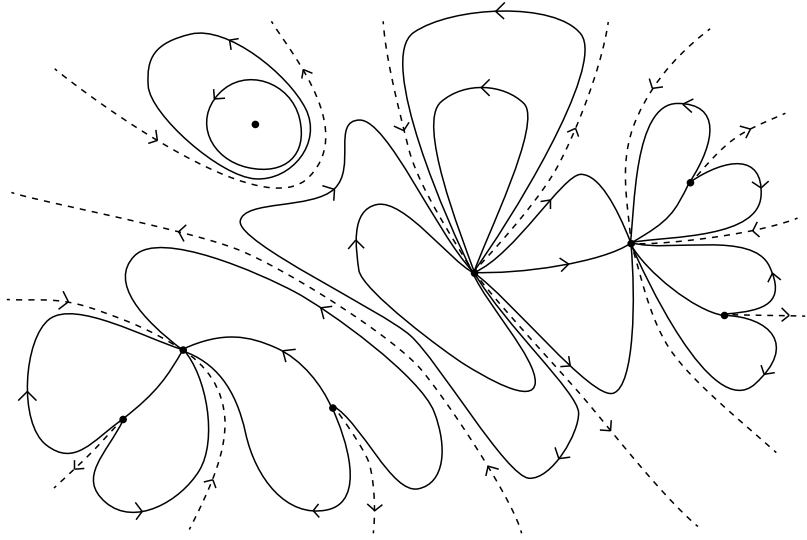

\vspace{2mm}

These tools already allow us to better understand the global structure of polynomial vector fields. In order to define this structure, a study of the vector field at infinity turns out to be of the most interesting. In fact, for every polynomial $P \in \mathcal{P}_d$, there exists a unique $\mathbb{C}$-analytic isomorphism $\Phi$, tangent to the identity at infinity, such that for each solution $\gamma_P(t,z)$ of $\xi_P$ defined in a neighborhood of infinity, there exists a solution $\gamma$ of the vector field $\xi_d$, defined by the polynomial $X^d$, such that
\begin{displaymath}
\Phi(\gamma_P(t,z)) = \gamma(t,\Phi(z)).
\end{displaymath}
This relationship is also abbreviated as follows : $\Phi_{\ast}(\xi_P) = \xi_d$.

\vspace{2mm}

In other words, the behavior of the vector field $\xi_P$ in a neighborhood of infinity is uniquely determined by the degree of the polynomial $P$. The study of the vector field $\xi_d$ allows us to deduce that there exist $2d-2$ solutions $\gamma_l$, with $l \in \{0,\ldots, 2d-3\}$, of the polynomial differential equation $\dot{z} = P(z)$ defined in a neighborhood of infinity and asymptotic to the ray $t.\delta_l$ for $t$ large enough, where 
\begin{equation}
\label{pointdelta}
\delta_l = exp\left( \pi i \frac{l}{d-1} \right), \quad l \in \{0, \ldots , 2d-3\},
\end{equation}
is the consecutive $2(d-1)$-th roots of unity.

\vspace{2mm}

Moreover, if $l$ is odd, there exists $\alpha_l \in ]-\infty , 0[$ such that $\gamma_l$ is defined on $]\alpha_l , 0]$ and $|\gamma_l(t)| \rightarrow \infty$ as $t \rightarrow \alpha_l$. Similarly, if $l$ is even, there exists $\beta_l \in ]0, +\infty[$ such that $\gamma_l$ is defined on $]0, \beta_l[$ and $|\gamma_l(t)| \rightarrow \infty$ as $t \rightarrow \beta_l$.
Then we call separatrices of the vector field $\xi_P$, noted $s_l$, the maximal trajectories of $\xi_P$ which coincide with the particular solutions $\gamma_l$. We distinguish three types of separatrices:

\begin{itemize}
\item $s_l$ is an outgoing separatrix from the point at infinity if $\zeta := \overline{s}_l \setminus s_l$ is the limit point of $s_l$ as $t$ tends to $+\infty$ (this implies that $l$ is odd).
\item $s_l$ is an incoming separatrix to the point at infinity if $\zeta := \overline{s}_l \setminus s_l$ is the limit point of $s_l$ as $t$ tends to $-\infty$ (this implies that $l$ is even).
\item $s_l$ is a homoclinic separatrix of infinity if $\overline{s}_l \setminus s_l = \emptyset$ in $\mathbb{C}$. In this case, the separatrix $s_l$ is both outgoing from and incoming to infinity, so this maximal trajectory is defined by two particular solutions $\gamma_k$ and $\gamma_j$, with $k$ odd and $j$ even. In this case, $l=k$ or $j$, that's why we note also $s_l = s_{k,j}$.
\end{itemize}

These trajectories are the only ones which tend to infinity. Note that the way we present the separatrices here imposes a label, and so an order between these solutions. This is the main issue of the enumeration presented in this article.

\vspace{3mm}

The \textbf{separatrix graph} (or configuration) $\Gamma_P$ of $\xi_P$ is defined as follows:
\begin{displaymath}
\Gamma_P = \bigcup_{l=0}^{2d-3}{\widehat{s}_l},
\end{displaymath}
where $\widehat{s}_l$ is the closure of $s_l$ in the Riemann sphere $\widehat{\mathbb{C}}$.

The separatrix graph is an important tool for our study. In fact, the global phase portrait of a polynomial equation (\ref{poldiffeq}) is univocally determined by its separatrix graph, in the sense that the separatrix graph is perfect to identify the topological structure of polynomial vector fields. According to that result, to give a topological enumeration of complex polynomial vector fields, we need to know how many separatrix graph we can construct (up to equivalence).

\vspace{2mm}

Now, let $\Gamma_P$ be the separatrix graph on the compactification of $\mathbb{C}$ associated to the polynomial $P$, and embed it into the closed unit disk $\overline{\mathbb{D}}$ so that the point at infinity be sent on the unit circle, i.e. it performs a blowing up of the sphere at the point $\infty$. Interest of such an embedding is to highlight the common property of polynomial vector fields in a neighborhood of infinity, i.e. the directions of the separatrices to infinity. Consider $\mathbb{S}^1$ as the disjoint union of half-closed intervals $\tilde{E}_j = \{ e^{2\pi it} , t \in ]\theta_{j-1},\theta_j] \}$, where $\theta_j = \frac{j}{2(d-1)}$, and $E_j$ the interior of $\tilde{E}_j$. The half-closed intervals $\tilde{E}_j$ are called the \textbf{boards} of the unit disk.

\vspace{2mm}

Let $Z$ be a connected component of $\overline{\mathbb{D}} \setminus \Gamma_P$ (where the separatrix graph $\Gamma_P$ is embedded in $\overline{\mathbb{D}}$). Such a component is called \textbf{zone} and can be of three different types:

\begin{enumerate}
\item[1.] There is no equilibrium point on the boundary $\partial Z$. In this case, $Z$ is called a center zone that contains a center, and all trajectories in $Z$ are periodic (with the same period). Moreover, the boundary of $Z$ consists of one or several homoclinic separatrices, and if a center zone contains $k$ homoclinic separatrices on the boundary $\partial Z$, then the center zone intersects $\overline{\mathbb{D}}$ at $k$ open arcs $E_{i_1}, \ldots , E_{i_k}$. 

\item[2.] There is exactly one equilibrium point on the boundary $\partial Z$. Then this equilibrium point is necessary a multiple equilibrium point $\zeta$. In this case, $Z = \mathcal{B}_{\alpha}(\zeta) \cap \mathcal{B}_{\omega}(\zeta)$, is called a sepal zone. There are exactly $2m-2$ sepals corresponding to a multiple equilibrium point of multiplicity $m$. Moreover the boundary of $Z$ contains exactly one incoming and one outgoing separatrix, and possibly one or several homoclinic separatrices. If a sepal zone contains $k$ homoclinic separatrices on its boundary, then it intersects $\overline{\mathbb{D}}$ at $k+1$ open arcs.

\item[3.] There are exactly two equilibrium points $\zeta_{\alpha}$ and $\zeta_{\beta}$ on the boundary $\partial Z$. In this case, $Z$ contains no equilibrium point inside. It is called an $\alpha \omega$-zone and is of four subtypes:
\begin{itemize}
\item $Z = \mathcal{B}(\zeta_{\alpha}) \cap \mathcal{B}(\zeta_{\omega})$, where $\zeta_{\alpha}$ and $\zeta_{\omega}$ are a source and a sink respectively.
\item $Z = \mathcal{B}(\zeta_{\alpha}) \cap \mathcal{B}_{\omega}(\zeta_{\omega})$, where $\zeta_{\alpha}$ and $\zeta_{\omega}$ are a source and a multiple equilibrium point respectively. In this case $Z$ is called an attracting interpetal for $\zeta_{\omega}$.
\item $Z = \mathcal{B}_{\alpha}(\zeta_{\alpha}) \cap \mathcal{B}(\zeta_{\omega})$, where $\zeta_{\alpha}$ and $\zeta_{\omega}$ are a multiple equilibrium point and a sink respectively. In this case $Z$ is called an repelling interpetal for $\zeta_{\alpha}$.
\item $Z = \mathcal{B}_{\alpha}(\zeta_{\alpha}) \cap \mathcal{B}_{\omega}(\zeta_{\omega})$, where $\zeta_{\alpha}$ and $\zeta_{\omega}$ are both multiple equilibrium points. In this case $Z$ is a repelling interpetal for $\zeta_{\alpha}$ and an attracting interpetal for $\zeta_{\omega}$.
\end{itemize}
The boundary of $Z$ contains one or two incoming separatrices and one or two outgoing separatrices, and possibly one or several homoclinic separatrices. Moreover if $Z$ is both on the left of $k$ (oriented) homoclinic separatrices and on the right of $l$ (oriented) homoclinic separatrices on its boundary, then $Z$ intersects $\overline{\mathbb{D}}$ at $k+1$ odd open arcs and $l+1$ even open arcs.
\end{enumerate}

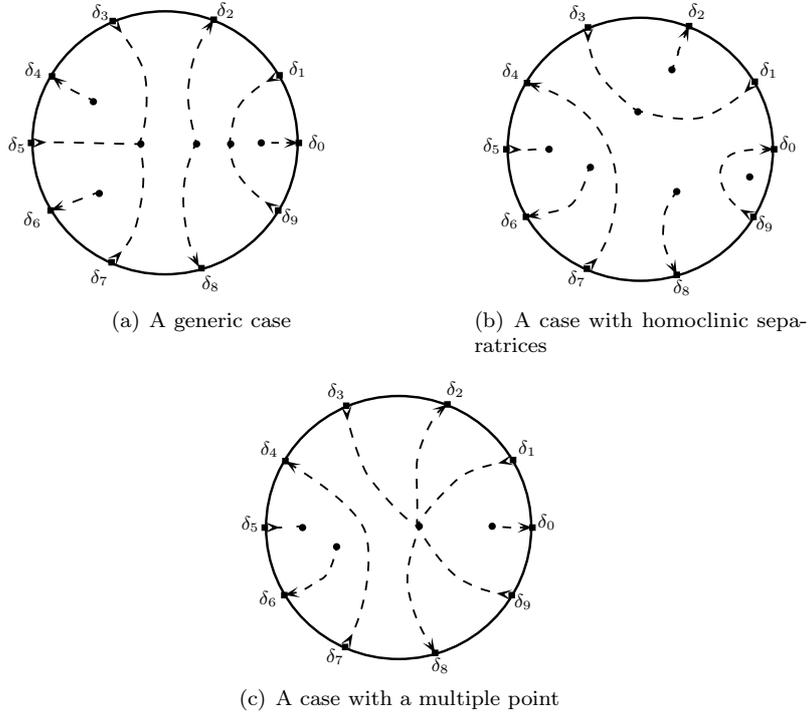
\begin{figure}[ht]
\centering
\subfigure[hang,small][A generic case]{
\scalebox{0.8}{
\begin{pspicture}(0,-2.43875)(6.395,2.43875)
\pscircle[linewidth=0.04,dimen=outer](2.6021874,0.0946875){2.2}
\psdots[dotsize=0.12,dotstyle=square*](4.8021874,0.0946875)
\psdots[dotsize=0.12,dotstyle=square*](0.4021875,0.0946875)
\psdots[dotsize=0.12,dotstyle=square*](3.4021876,2.1346874)
\psdots[dotsize=0.12,dotstyle=square*](1.7421875,2.1146874)
\psdots[dotsize=0.12,dotstyle=square*](4.4821873,1.2146875)
\psdots[dotsize=0.12,dotstyle=square*](0.7421875,1.1946875)
\psdots[dotsize=0.12,dotstyle=square*](0.7221875,-1.0253125)
\psdots[dotsize=0.12,dotstyle=square*](1.7221875,-1.8853126)
\psdots[dotsize=0.12,dotstyle=square*](3.2021875,-1.9853125)
\psdots[dotsize=0.12,dotstyle=square*](4.4621873,-1.0253125)
\psdots[dotsize=0.12](1.4221874,0.7746875)
\psdots[dotsize=0.12](1.5221875,-0.7453125)
\psdots[dotsize=0.12](2.2021875,0.0746875)
\psdots[dotsize=0.12](3.1221876,0.0746875)
\psdots[dotsize=0.12](4.1821876,0.0946875)
\psdots[dotsize=0.12](3.6821876,0.0746875)
\psbezier[linewidth=0.03,linestyle=dashed,dash=0.16cm 0.16cm,arrowsize=0.12cm 2.0,arrowlength=1.4,arrowinset=0.5]{->}(4.1821876,0.0946875)(4.7421875,0.0946875)(4.3421874,0.0946875)(4.8021874,0.0946875)
\psbezier[linewidth=0.03,linestyle=dashed,dash=0.16cm 0.16cm,arrowsize=0.12cm 2.0,arrowlength=1.4,arrowinset=0.5]{-<}(3.6621876,0.0746875)(3.6821876,-0.2653125)(3.7821875,-0.4253125)(3.9221876,-0.5853125)(4.0621877,-0.7453125)(4.1421876,-0.8053125)(4.4421873,-1.0253125)
\psbezier[linewidth=0.03,linestyle=dashed,dash=0.16cm 0.16cm,arrowsize=0.12cm 2.0,arrowlength=1.4,arrowinset=0.5]{-<}(3.6621876,0.0946875)(3.7221875,0.4346875)(3.7221875,0.6546875)(3.9421875,0.8546875)(4.1621876,1.0546875)(4.2421875,1.0946875)(4.4821873,1.2146875)
\psbezier[linewidth=0.03,linestyle=dashed,dash=0.16cm 0.16cm,arrowsize=0.12cm 2.0,arrowlength=1.4,arrowinset=0.5]{->}(3.1021874,0.0746875)(2.9621875,-0.2853125)(2.9021876,-0.5053125)(2.9021876,-0.8453125)(2.9021876,-1.1853125)(3.0021875,-1.6653125)(3.1621876,-1.9653125)
\psbezier[linewidth=0.03,linestyle=dashed,dash=0.16cm 0.16cm,arrowsize=0.12cm 2.0,arrowlength=1.4,arrowinset=0.5]{->}(3.1221876,0.0546875)(2.9821875,0.6346875)(2.955107,0.825563)(3.0021875,1.1546875)(3.049268,1.483812)(3.1421876,1.7346874)(3.4021876,2.1346874)
\psbezier[linewidth=0.03,linestyle=dashed,dash=0.16cm 0.16cm,arrowsize=0.12cm 2.0,arrowlength=1.4,arrowinset=0.5]{-<}(2.1821876,0.0946875)(2.3021874,-0.4853125)(2.2277858,-0.7910921)(2.1621876,-1.1253124)(2.096589,-1.4595329)(1.9821875,-1.6053125)(1.7221875,-1.8853126)
\psbezier[linewidth=0.03,linestyle=dashed,dash=0.16cm 0.16cm,arrowsize=0.12cm 2.0,arrowlength=1.4,arrowinset=0.5]{-<}(2.1821876,0.0946875)(2.3221874,0.4146875)(2.3221874,1.0146875)(2.2621875,1.2546875)(2.2021875,1.4946876)(1.9221874,1.8746876)(1.7421875,2.1146874)
\psbezier[linewidth=0.03,linestyle=dashed,dash=0.16cm 0.16cm,arrowsize=0.12cm 2.0,arrowlength=1.4,arrowinset=0.5]{->}(1.4221874,0.7946875)(1.2221875,0.8946875)(1.0221875,0.9746875)(0.7421875,1.1946875)
\psbezier[linewidth=0.03,linestyle=dashed,dash=0.16cm 0.16cm,arrowsize=0.12cm 2.0,arrowlength=1.4,arrowinset=0.5]{->}(1.5021875,-0.7253125)(1.2021875,-0.8453125)(1.0421875,-0.9053125)(0.7221875,-1.0253125)
\psbezier[linewidth=0.03,linestyle=dashed,dash=0.16cm 0.16cm,arrowsize=0.12cm 2.0,arrowlength=1.4,arrowinset=0.5]{-<}(2.1821876,0.0746875)(1.6821876,0.0746875)(0.9621875,0.1146875)(0.4021875,0.0946875)
\usefont{T1}{ptm}{m}{n}
\rput(5.1021874,0.0946875){\small $\delta_0$}
\usefont{T1}{ptm}{m}{n}
\rput(4.7821873,1.3146875){\small $\delta_1$}
\usefont{T1}{ptm}{m}{n}
\rput(3.6021876,2.2846874){\small $\delta_2$}
\usefont{T1}{ptm}{m}{n}
\rput(1.5421875,2.2646874){\small $\delta_3$}
\usefont{T1}{ptm}{m}{n}
\rput(0.4421875,1.2946875){\small $\delta_4$}
\usefont{T1}{ptm}{m}{n}
\rput(0.16859375,0.064375){\small $\delta_5$}
\usefont{T1}{ptm}{m}{n}
\rput(0.4221875,-1.1753125){\small $\delta_6$}
\usefont{T1}{ptm}{m}{n}
\rput(1.5221875,-2.1353126){\small $\delta_7$}
\usefont{T1}{ptm}{m}{n}
\rput(3.3521875,-2.2353125){\small $\delta_8$}
\usefont{T1}{ptm}{m}{n}
\rput(4.6621873,-1.1253125){\small $\delta_9$}
\end{pspicture} 
}
} \qquad
\subfigure[hang,small][A case with homoclinic separatrices]{
\scalebox{0.8}{
\begin{pspicture}(0,-2.4203124)(4.9596877,2.4203124)
\pscircle[linewidth=0.04,dimen=outer](2.46875,0.004375){2.2}
\psdots[dotsize=0.12,dotstyle=square*](4.66875,0.004375)
\psdots[dotsize=0.12,dotstyle=square*](0.26875,0.004375)
\psdots[dotsize=0.12,dotstyle=square*](3.26875,2.044375)
\psdots[dotsize=0.12,dotstyle=square*](1.60875,2.024375)
\psdots[dotsize=0.12,dotstyle=square*](4.34875,1.124375)
\psdots[dotsize=0.12,dotstyle=square*](0.60875,1.104375)
\psdots[dotsize=0.12,dotstyle=square*](0.58875,-1.115625)
\psdots[dotsize=0.12,dotstyle=square*](1.58875,-1.975625)
\psdots[dotsize=0.12,dotstyle=square*](3.06875,-2.075625)
\psdots[dotsize=0.12,dotstyle=square*](4.32875,-1.115625)
\usefont{T1}{ptm}{m}{n}
\rput(4.9021874,0.0946875){\small $\delta_0$}
\usefont{T1}{ptm}{m}{n}
\rput(4.5821873,1.3146875){\small $\delta_1$}
\usefont{T1}{ptm}{m}{n}
\rput(3.4021876,2.2846874){\small $\delta_2$}
\usefont{T1}{ptm}{m}{n}
\rput(1.4421875,2.2646874){\small $\delta_3$}
\usefont{T1}{ptm}{m}{n}
\rput(0.3421875,1.2946875){\small $\delta_4$}
\usefont{T1}{ptm}{m}{n}
\rput(0.02859375,0.064375){\small $\delta_5$}
\usefont{T1}{ptm}{m}{n}
\rput(0.3221875,-1.1753125){\small $\delta_6$}
\usefont{T1}{ptm}{m}{n}
\rput(1.4221875,-2.1353126){\small $\delta_7$}
\usefont{T1}{ptm}{m}{n}
\rput(3.1521875,-2.2853125){\small $\delta_8$}
\usefont{T1}{ptm}{m}{n}
\rput(4.5121873,-1.2253125){\small $\delta_9$}
\psdots[dotsize=0.12](2.98875,1.324375)
\psdots[dotsize=0.12](2.42875,0.624375)
\psdots[dotsize=0.12](0.96875,0.004375)
\psdots[dotsize=0.12](1.64875,-0.295625)
\psdots[dotsize=0.12](3.06875,-0.695625)
\psdots[dotsize=0.12](4.26875,-0.455625)
\psbezier[linewidth=0.03,linestyle=dashed,dash=0.16cm 0.16cm,arrowsize=0.12cm 2.0,arrowlength=1.4,arrowinset=0.5]{->}(2.98875,1.324375)(3.22875,1.924375)(2.94875,1.304375)(3.26875,2.044375)
\psbezier[linewidth=0.03,linestyle=dashed,dash=0.16cm 0.16cm,arrowsize=0.12cm 2.0,arrowlength=1.4,arrowinset=0.5]{-<}(2.40875,0.644375)(2.20875,0.784375)(1.94875,1.044375)(1.82875,1.204375)(1.70875,1.364375)(1.60875,1.684375)(1.60875,2.024375)
\psbezier[linewidth=0.03,linestyle=dashed,dash=0.16cm 0.16cm,arrowsize=0.12cm 2.0,arrowlength=1.4,arrowinset=0.5]{-<}(0.96875,0.004375)(0.44875,0.004375)(0.86875,-0.015625)(0.26875,0.004375)
\psbezier[linewidth=0.03,linestyle=dashed,dash=0.16cm 0.16cm,arrowsize=0.12cm 2.0,arrowlength=1.4,arrowinset=0.5]{->}(1.62875,-0.275625)(1.56875,-0.575625)(1.46875,-0.755625)(1.36875,-0.875625)(1.26875,-0.995625)(1.12875,-1.015625)(0.58875,-1.115625)
\psbezier[linewidth=0.03,linestyle=dashed,dash=0.16cm 0.16cm,arrowsize=0.12cm 2.0,arrowlength=1.4,arrowinset=0.5]{<-<}(0.60875,1.104375)(1.32875,0.864375)(1.76875,0.544375)(1.98875,-0.075625)(2.20875,-0.695625)(1.98875,-1.575625)(1.58875,-1.975625)
\psbezier[linewidth=0.03,linestyle=dashed,dash=0.16cm 0.16cm,arrowsize=0.12cm 2.0,arrowlength=1.4,arrowinset=0.5]{-<}(2.42875,0.644375)(2.88875,0.504375)(3.248423,0.48699465)(3.44875,0.524375)(3.649077,0.56175536)(4.06875,0.824375)(4.34875,1.124375)
\psbezier[linewidth=0.03,linestyle=dashed,dash=0.16cm 0.16cm,arrowsize=0.12cm 2.0,arrowlength=1.4,arrowinset=0.5]{<-<}(4.66875,0.004375)(3.90875,0.004375)(3.8783169,-0.13259)(3.80875,-0.275625)(3.739183,-0.41866)(3.78875,-0.815625)(4.32875,-1.115625)
\psbezier[linewidth=0.03,linestyle=dashed,dash=0.16cm 0.16cm,arrowsize=0.12cm 2.0,arrowlength=1.4,arrowinset=0.5]{->}(3.04875,-0.715625)(2.82875,-1.055625)(2.8627622,-1.0783677)(2.84875,-1.275625)(2.8347378,-1.4728824)(2.82875,-1.575625)(3.06875,-2.075625)
\end{pspicture} 
}
}\\
\subfigure[hang,small][A case with a multiple point]{
\scalebox{0.8}{
\begin{pspicture}(0,-2.4203124)(4.9596877,2.4203124)
\pscircle[linewidth=0.04,dimen=outer](2.46875,0.004375){2.2}
\psdots[dotsize=0.12,dotstyle=square*](4.66875,0.004375)
\psdots[dotsize=0.12,dotstyle=square*](0.26875,0.004375)
\psdots[dotsize=0.12,dotstyle=square*](3.26875,2.044375)
\psdots[dotsize=0.12,dotstyle=square*](1.60875,2.024375)
\psdots[dotsize=0.12,dotstyle=square*](4.34875,1.124375)
\psdots[dotsize=0.12,dotstyle=square*](0.60875,1.104375)
\psdots[dotsize=0.12,dotstyle=square*](0.58875,-1.115625)
\psdots[dotsize=0.12,dotstyle=square*](1.58875,-1.975625)
\psdots[dotsize=0.12,dotstyle=square*](3.06875,-2.075625)
\psdots[dotsize=0.12,dotstyle=square*](4.32875,-1.115625)
\usefont{T1}{ptm}{m}{n}
\rput(4.9021874,0.0946875){\small $\delta_0$}
\usefont{T1}{ptm}{m}{n}
\rput(4.5821873,1.3146875){\small $\delta_1$}
\usefont{T1}{ptm}{m}{n}
\rput(3.4021876,2.2846874){\small $\delta_2$}
\usefont{T1}{ptm}{m}{n}
\rput(1.4421875,2.2646874){\small $\delta_3$}
\usefont{T1}{ptm}{m}{n}
\rput(0.3421875,1.2946875){\small $\delta_4$}
\usefont{T1}{ptm}{m}{n}
\rput(0.02859375,0.064375){\small $\delta_5$}
\usefont{T1}{ptm}{m}{n}
\rput(0.3221875,-1.1753125){\small $\delta_6$}
\usefont{T1}{ptm}{m}{n}
\rput(1.4221875,-2.1353126){\small $\delta_7$}
\usefont{T1}{ptm}{m}{n}
\rput(3.1521875,-2.2853125){\small $\delta_8$}
\usefont{T1}{ptm}{m}{n}
\rput(4.5121873,-1.2253125){\small $\delta_9$}
\psdots[dotsize=0.12](4.00875,0.024375)
\psdots[dotsize=0.12](2.80875,0.024375)
\psdots[dotsize=0.12](0.88875,0.004375)
\psdots[dotsize=0.12](1.44875,-0.315625)
\psbezier[linewidth=0.03,linestyle=dashed,dash=0.16cm 0.16cm,arrowsize=0.12cm 2.0,arrowlength=1.4,arrowinset=0.5]{->}(1.44875,-0.315625)(1.40875,-0.555625)(1.30875,-0.815625)(1.18875,-0.855625)(1.06875,-0.895625)(0.98875,-1.035625)(0.58875,-1.115625)
\psbezier[linewidth=0.03,linestyle=dashed,dash=0.16cm 0.16cm,arrowsize=0.12cm 2.0,arrowlength=1.4,arrowinset=0.5]{>-}(0.26875,0.004375)(0.72875,0.004375)(0.46875,0.004375)(0.86875,0.024375)
\psbezier[linewidth=0.03,linestyle=dashed,dash=0.16cm 0.16cm,arrowsize=0.12cm 2.0,arrowlength=1.4,arrowinset=0.5]{<-<}(0.60875,1.104375)(1.04875,0.824375)(1.809771,0.33213678)(1.92875,-0.095625)(2.047729,-0.5233868)(1.92875,-1.375625)(1.58875,-1.975625)
\psbezier[linewidth=0.03,linestyle=dashed,dash=0.16cm 0.16cm,arrowsize=0.12cm 2.0,arrowlength=1.4,arrowinset=0.5]{>-}(1.60875,2.024375)(1.62875,1.724375)(1.84875,1.004375)(1.98875,0.824375)(2.12875,0.644375)(2.50875,0.164375)(2.80875,0.004375)
\psbezier[linewidth=0.03,linestyle=dashed,dash=0.16cm 0.16cm,arrowsize=0.12cm 2.0,arrowlength=1.4,arrowinset=0.5]{->}(2.80875,0.024375)(2.70875,-0.315625)(2.559331,-0.684035)(2.62875,-1.035625)(2.698169,-1.387215)(2.78875,-1.615625)(3.06875,-2.075625)
\psbezier[linewidth=0.03,linestyle=dashed,dash=0.16cm 0.16cm,arrowsize=0.12cm 2.0,arrowlength=1.4,arrowinset=0.5]{>-}(4.32875,-1.115625)(3.96875,-1.075625)(3.72875,-0.995625)(3.48875,-0.855625)(3.24875,-0.715625)(2.96875,-0.295625)(2.76875,0.024375)
\psbezier[linewidth=0.03,linestyle=dashed,dash=0.16cm 0.16cm,arrowsize=0.12cm 2.0,arrowlength=1.4,arrowinset=0.5]{<-}(3.26875,2.044375)(3.00875,1.844375)(2.78875,1.324375)(2.76875,1.084375)(2.74875,0.844375)(2.78875,0.504375)(2.76875,0.064375)
\psbezier[linewidth=0.03,linestyle=dashed,dash=0.16cm 0.16cm,arrowsize=0.12cm 2.0,arrowlength=1.4,arrowinset=0.5]{-<}(2.76875,0.044375)(2.96875,0.344375)(3.14875,0.524375)(3.34875,0.724375)(3.54875,0.924375)(3.70875,0.984375)(4.34875,1.124375)
\psbezier[linewidth=0.03,linestyle=dashed,dash=0.16cm 0.16cm,arrowsize=0.12cm 2.0,arrowlength=1.4,arrowinset=0.5]{->}(3.96875,0.024375)(4.6275,-0.015625)(4.24875,0.004375)(4.66875,0.004375)
\end{pspicture}
}
}
\caption{Example of separatrix graphs}
\label{exsepgr}
\end{figure}

We can finally give simple definitions of the two equivalence relations necessary for our enumeration.

\begin{defreleq}
Let $P$, $Q$ be two monic, centered polynomials of degree $d$, and $\Gamma_P$, $\Gamma_Q$ be their respective separatrix graphs. We say that the two polynomial vector fields $\xi_P$ and $\xi_Q$ are equivalent, denoted by $P \sim Q$, if there exists an isotopy $h: \overline{\mathbb{D}} \times [0,1] \rightarrow \overline{\mathbb{D}}$ that sends the separatrices of $\Gamma_P$ to the separatrices of $\Gamma_Q$ and such that $h_{|S^1 \times [0,1]} = id$.
\end{defreleq}

\begin{defreleqtop}
Let $P$, $Q$ be two monic, centered polynomials of degree $d$, and $\Gamma_P$, $\Gamma_Q$ be their respective separatrix graphs. We say that the two polynomial vector fields $\xi_P$ and $\xi_Q$ are topologically equivalent, denoted by $P \sim_{top} Q$, if there exists an isotopy $h: \overline{\mathbb{D}} \times [0,1] \rightarrow \overline{\mathbb{D}}$ that sends the separatrices of $\Gamma_P$ to the separatrices of $\Gamma_Q$ (cf Figure \ref{extopbrarep} for an example).
\end{defreleqtop}

\section{Combinatorial models}

In this section, we introduce various combinatorial models which will enable us to enumerate the equivalence classes of the complex polynomial vector fields. First models studied in this section are similar to those already present in the articles \cite{DES} and \cite{Dias2}. The interested reader can refer to these articles for more details.

\subsection{From separatrix graph $\Gamma_P$ to transversal graph $\Sigma_P$.}

Let $P$ be a monic centered polynomial of degree $d$, and $\Gamma_P$ be its separatrix graph. Then using the different known informations about the connected components of $\overline{\mathbb{D}} \setminus \Gamma_P$ given in the previous section, we can construct a new combinatory data set. For this, we distinguish two important objects:

\begin{enumerate}
\item[1.] The homoclinic separatrices. These particular solutions of our polynomial differential equation link two points $\delta_i$ and $\delta_j$ at infinity (such a solution is usually denoted by $s_{i,j}$), and so link boards $\tilde{E}_i$ and $\tilde{E}_j$.
\item[2.] The $\alpha \omega$-zone. Such a zone is equivalent to a band, having exactly one even boards $\tilde{E}_i$ and one odd boards $\tilde{E}_j$. So we can trace a curve, contained in the zone, between $E_i$ and $E_j$ such that it does not cross any separatrices and crosses the trajectories of $\xi_P$ at a constant, non-zero angle. This curve is called transversal and denoted by $T_{i,j}$. Notice that the transversal $T_{i,j}$ pairs the boards $\tilde{E}_i$ and $\tilde{E}_j$.
\end{enumerate}

\vspace{2mm}

So starting from a separatrix graph $\Gamma_P$, we construct a new graph on the unit disk, called the transversal graph and denoted by $\Sigma_P$, defined by the union of the transversals and the homoclinic separatrices coming from $\Gamma_P$. Note that this graph is constructed by pairing even board $\tilde{E}_i$ with odd board $\tilde{E}_j$. Nevertheless, some boards are not paired with others; these boards are those derived from sepal zones. See Figure \ref{exsepgraph} for some examples.

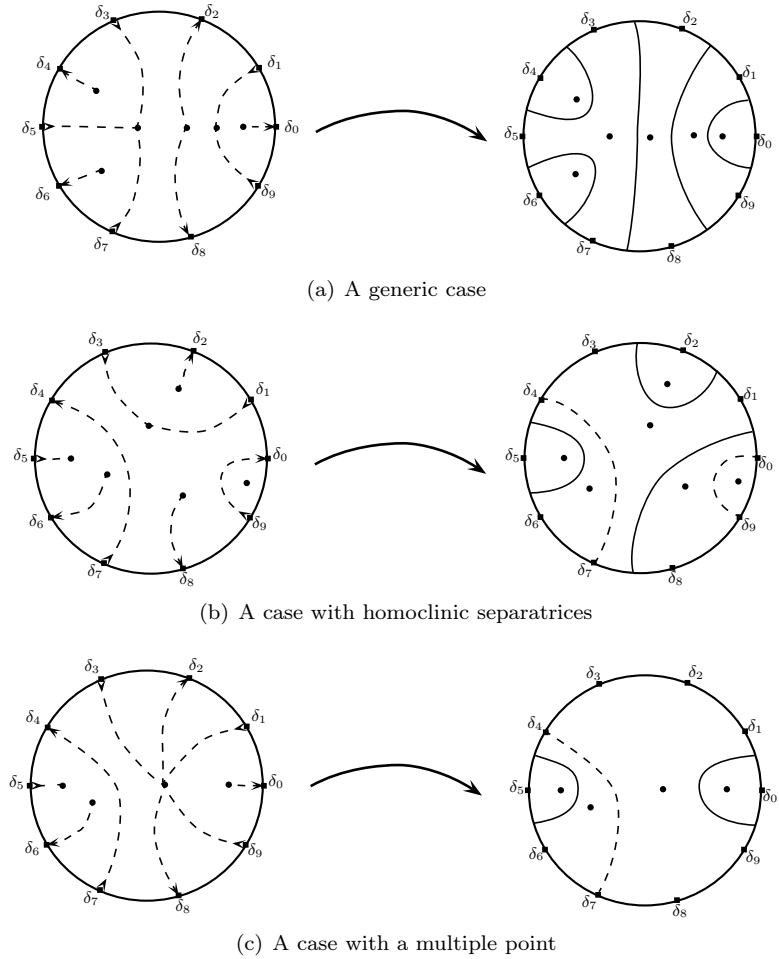
\begin{figure}[ht]
\centering
\subfigure[hang,small][A generic case]{
\scalebox{0.7}{
\begin{pspicture}(0,-2.5103126)(14.119687,2.5103126)
\pscircle[linewidth=0.04,dimen=outer](2.6021874,0.0946875){2.2}
\psdots[dotsize=0.12,dotstyle=square*](4.8021874,0.0946875)
\psdots[dotsize=0.12,dotstyle=square*](0.4021875,0.0946875)
\psdots[dotsize=0.12,dotstyle=square*](3.4021876,2.1346874)
\psdots[dotsize=0.12,dotstyle=square*](1.7421875,2.1146874)
\psdots[dotsize=0.12,dotstyle=square*](4.4821873,1.2146875)
\psdots[dotsize=0.12,dotstyle=square*](0.7421875,1.1946875)
\psdots[dotsize=0.12,dotstyle=square*](0.7221875,-1.0253125)
\psdots[dotsize=0.12,dotstyle=square*](1.7221875,-1.8853126)
\psdots[dotsize=0.12,dotstyle=square*](3.2021875,-1.9853125)
\psdots[dotsize=0.12,dotstyle=square*](4.4621873,-1.0253125)
\psdots[dotsize=0.12](1.4221874,0.7746875)
\psdots[dotsize=0.12](1.5221875,-0.7453125)
\psdots[dotsize=0.12](2.2021875,0.0746875)
\psdots[dotsize=0.12](3.1221876,0.0746875)
\psdots[dotsize=0.12](4.1821876,0.0946875)
\psdots[dotsize=0.12](3.6821876,0.0746875)
\psbezier[linewidth=0.03,linestyle=dashed,dash=0.16cm 0.16cm,arrowsize=0.12cm 2.0,arrowlength=1.4,arrowinset=0.5]{->}(4.1821876,0.0946875)(4.7421875,0.0946875)(4.3421874,0.0946875)(4.8021874,0.0946875)
\psbezier[linewidth=0.03,linestyle=dashed,dash=0.16cm 0.16cm,arrowsize=0.12cm 2.0,arrowlength=1.4,arrowinset=0.5]{-<}(3.6621876,0.0746875)(3.6821876,-0.2653125)(3.7821875,-0.4253125)(3.9221876,-0.5853125)(4.0621877,-0.7453125)(4.1421876,-0.8053125)(4.4421873,-1.0253125)
\psbezier[linewidth=0.03,linestyle=dashed,dash=0.16cm 0.16cm,arrowsize=0.12cm 2.0,arrowlength=1.4,arrowinset=0.5]{-<}(3.6621876,0.0946875)(3.7221875,0.4346875)(3.7221875,0.6546875)(3.9421875,0.8546875)(4.1621876,1.0546875)(4.2421875,1.0946875)(4.4821873,1.2146875)
\psbezier[linewidth=0.03,linestyle=dashed,dash=0.16cm 0.16cm,arrowsize=0.12cm 2.0,arrowlength=1.4,arrowinset=0.5]{->}(3.1021874,0.0746875)(2.9621875,-0.2853125)(2.9021876,-0.5053125)(2.9021876,-0.8453125)(2.9021876,-1.1853125)(3.0021875,-1.6653125)(3.1621876,-1.9653125)
\psbezier[linewidth=0.03,linestyle=dashed,dash=0.16cm 0.16cm,arrowsize=0.12cm 2.0,arrowlength=1.4,arrowinset=0.5]{->}(3.1221876,0.0546875)(2.9821875,0.6346875)(2.955107,0.825563)(3.0021875,1.1546875)(3.049268,1.483812)(3.1421876,1.7346874)(3.4021876,2.1346874)
\psbezier[linewidth=0.03,linestyle=dashed,dash=0.16cm 0.16cm,arrowsize=0.12cm 2.0,arrowlength=1.4,arrowinset=0.5]{-<}(2.1821876,0.0946875)(2.3021874,-0.4853125)(2.2277858,-0.7910921)(2.1621876,-1.1253124)(2.096589,-1.4595329)(1.9821875,-1.6053125)(1.7221875,-1.8853126)
\psbezier[linewidth=0.03,linestyle=dashed,dash=0.16cm 0.16cm,arrowsize=0.12cm 2.0,arrowlength=1.4,arrowinset=0.5]{-<}(2.1821876,0.0946875)(2.3221874,0.4146875)(2.3221874,1.0146875)(2.2621875,1.2546875)(2.2021875,1.4946876)(1.9221874,1.8746876)(1.7421875,2.1146874)
\psbezier[linewidth=0.03,linestyle=dashed,dash=0.16cm 0.16cm,arrowsize=0.12cm 2.0,arrowlength=1.4,arrowinset=0.5]{->}(1.4221874,0.7946875)(1.2221875,0.8946875)(1.0221875,0.9746875)(0.7421875,1.1946875)
\psbezier[linewidth=0.03,linestyle=dashed,dash=0.16cm 0.16cm,arrowsize=0.12cm 2.0,arrowlength=1.4,arrowinset=0.5]{->}(1.5021875,-0.7253125)(1.2021875,-0.8453125)(1.0421875,-0.9053125)(0.7221875,-1.0253125)
\psbezier[linewidth=0.03,linestyle=dashed,dash=0.16cm 0.16cm,arrowsize=0.12cm 2.0,arrowlength=1.4,arrowinset=0.5]{-<}(2.1821876,0.0746875)(1.6821876,0.0746875)(0.9621875,0.1146875)(0.4021875,0.0946875)
\usefont{T1}{ptm}{m}{n}
\rput(5.1021874,0.0946875){\small $\delta_0$}
\usefont{T1}{ptm}{m}{n}
\rput(4.7821873,1.3146875){\small $\delta_1$}
\usefont{T1}{ptm}{m}{n}
\rput(3.6021876,2.2846874){\small $\delta_2$}
\usefont{T1}{ptm}{m}{n}
\rput(1.5421875,2.2646874){\small $\delta_3$}
\usefont{T1}{ptm}{m}{n}
\rput(0.4421875,1.2946875){\small $\delta_4$}
\usefont{T1}{ptm}{m}{n}
\rput(0.16859375,0.064375){\small $\delta_5$}
\usefont{T1}{ptm}{m}{n}
\rput(0.4221875,-1.1753125){\small $\delta_6$}
\usefont{T1}{ptm}{m}{n}
\rput(1.5221875,-2.1353126){\small $\delta_7$}
\usefont{T1}{ptm}{m}{n}
\rput(3.3521875,-2.2353125){\small $\delta_8$}
\usefont{T1}{ptm}{m}{n}
\rput(4.6621873,-1.1253125){\small $\delta_9$}
\psbezier[linewidth=0.05,arrowsize=0.05291667cm 4.0,arrowlength=1.4,arrowinset=0.4]{->}(5.54875,-0.005625)(6.14875,0.334375)(6.84875,0.394375)(7.20875,0.394375)(7.56875,0.394375)(8.14875,0.234375)(8.76875,-0.185625)
\pscircle[linewidth=0.04,dimen=outer](11.62875,-0.085625){2.2}
\psdots[dotsize=0.12,dotstyle=square*](13.82875,-0.085625)
\psdots[dotsize=0.12,dotstyle=square*](9.42875,-0.085625)
\psdots[dotsize=0.12,dotstyle=square*](12.42875,1.954375)
\psdots[dotsize=0.12,dotstyle=square*](10.76875,1.934375)
\psdots[dotsize=0.12,dotstyle=square*](13.50875,1.034375)
\psdots[dotsize=0.12,dotstyle=square*](9.76875,1.014375)
\psdots[dotsize=0.12,dotstyle=square*](9.74875,-1.205625)
\psdots[dotsize=0.12,dotstyle=square*](10.74875,-2.065625)
\psdots[dotsize=0.12,dotstyle=square*](12.22875,-2.165625)
\psdots[dotsize=0.12,dotstyle=square*](13.48875,-1.205625)
\usefont{T1}{ptm}{m}{n}
\rput(13.988594,-0.185625){\small $\delta_0$}
\usefont{T1}{ptm}{m}{n}
\rput(13.708593,1.134375){\small $\delta_1$}
\usefont{T1}{ptm}{m}{n}
\rput(12.5885935,2.154375){\small $\delta_2$}
\usefont{T1}{ptm}{m}{n}
\rput(10.668593,2.134375){\small $\delta_3$}
\usefont{T1}{ptm}{m}{n}
\rput(9.568594,1.194375){\small $\delta_4$}
\usefont{T1}{ptm}{m}{n}
\rput(9.228594,-0.025625){\small $\delta_5$}
\usefont{T1}{ptm}{m}{n}
\rput(9.5885935,-1.285625){\small $\delta_6$}
\usefont{T1}{ptm}{m}{n}
\rput(10.5885935,-2.205625){\small $\delta_7$}
\usefont{T1}{ptm}{m}{n}
\rput(12.288593,-2.365625){\small $\delta_8$}
\usefont{T1}{ptm}{m}{n}
\rput(13.668593,-1.345625){\small $\delta_9$}
\psdots[dotsize=0.12](10.44875,0.614375)
\psdots[dotsize=0.12](11.06875,-0.085625)
\psdots[dotsize=0.12](10.42875,-0.805625)
\psdots[dotsize=0.12](11.82875,-0.105625)
\psdots[dotsize=0.12](12.64875,-0.065625)
\psdots[dotsize=0.12](13.18875,-0.085625)
\psbezier[linewidth=0.03](10.24875,1.614375)(10.54875,1.394375)(10.884085,0.7541802)(10.68875,0.414375)(10.493416,0.07456978)(9.84875,0.294375)(9.50875,0.414375)
\psbezier[linewidth=0.03](9.52875,-0.685625)(10.04875,-0.425625)(10.58875,-0.305625)(10.74875,-0.545625)(10.90875,-0.785625)(10.68875,-1.405625)(10.22875,-1.745625)
\psbezier[linewidth=0.03](11.38875,-2.265625)(11.54875,-1.545625)(11.596673,-0.3251104)(11.58875,-0.045625)(11.580827,0.2338604)(11.74875,1.534375)(11.52875,2.094375)
\psbezier[linewidth=0.03](12.96875,1.634375)(12.60875,1.214375)(12.22875,0.374375)(12.22875,-0.105625)(12.22875,-0.585625)(12.46875,-1.305625)(12.88875,-1.845625)
\psbezier[linewidth=0.03](13.70875,-0.685625)(13.28875,-0.665625)(12.911936,-0.36469206)(12.90875,-0.085625)(12.905564,0.19344206)(13.16875,0.574375)(13.66875,0.594375)
\end{pspicture} 
}
} \\
\subfigure[hang,small][A case with homoclinic separatrices]{
\scalebox{0.7}{
\begin{pspicture}(0,-2.4303124)(14.159687,2.4303124)
\pscircle[linewidth=0.04,dimen=outer](2.46875,0.004375){2.2}
\psdots[dotsize=0.12,dotstyle=square*](4.66875,0.004375)
\psdots[dotsize=0.12,dotstyle=square*](0.26875,0.004375)
\psdots[dotsize=0.12,dotstyle=square*](3.26875,2.044375)
\psdots[dotsize=0.12,dotstyle=square*](1.60875,2.024375)
\psdots[dotsize=0.12,dotstyle=square*](4.34875,1.124375)
\psdots[dotsize=0.12,dotstyle=square*](0.60875,1.104375)
\psdots[dotsize=0.12,dotstyle=square*](0.58875,-1.115625)
\psdots[dotsize=0.12,dotstyle=square*](1.58875,-1.975625)
\psdots[dotsize=0.12,dotstyle=square*](3.06875,-2.075625)
\psdots[dotsize=0.12,dotstyle=square*](4.32875,-1.115625)
\usefont{T1}{ptm}{m}{n}
\rput(4.9021874,0.0946875){\small $\delta_0$}
\usefont{T1}{ptm}{m}{n}
\rput(4.5821873,1.3146875){\small $\delta_1$}
\usefont{T1}{ptm}{m}{n}
\rput(3.4021876,2.2846874){\small $\delta_2$}
\usefont{T1}{ptm}{m}{n}
\rput(1.4421875,2.2646874){\small $\delta_3$}
\usefont{T1}{ptm}{m}{n}
\rput(0.3421875,1.2946875){\small $\delta_4$}
\usefont{T1}{ptm}{m}{n}
\rput(0.02859375,0.064375){\small $\delta_5$}
\usefont{T1}{ptm}{m}{n}
\rput(0.3221875,-1.1753125){\small $\delta_6$}
\usefont{T1}{ptm}{m}{n}
\rput(1.4221875,-2.1353126){\small $\delta_7$}
\usefont{T1}{ptm}{m}{n}
\rput(3.1521875,-2.2853125){\small $\delta_8$}
\usefont{T1}{ptm}{m}{n}
\rput(4.5121873,-1.2253125){\small $\delta_9$}
\psdots[dotsize=0.12](2.98875,1.324375)
\psdots[dotsize=0.12](2.42875,0.624375)
\psdots[dotsize=0.12](0.96875,0.004375)
\psdots[dotsize=0.12](1.64875,-0.295625)
\psdots[dotsize=0.12](3.06875,-0.695625)
\psdots[dotsize=0.12](4.26875,-0.455625)
\psbezier[linewidth=0.03,linestyle=dashed,dash=0.16cm 0.16cm,arrowsize=0.12cm 2.0,arrowlength=1.4,arrowinset=0.5]{->}(2.98875,1.324375)(3.22875,1.924375)(2.94875,1.304375)(3.26875,2.044375)
\psbezier[linewidth=0.03,linestyle=dashed,dash=0.16cm 0.16cm,arrowsize=0.12cm 2.0,arrowlength=1.4,arrowinset=0.5]{-<}(2.40875,0.644375)(2.20875,0.784375)(1.94875,1.044375)(1.82875,1.204375)(1.70875,1.364375)(1.60875,1.684375)(1.60875,2.024375)
\psbezier[linewidth=0.03,linestyle=dashed,dash=0.16cm 0.16cm,arrowsize=0.12cm 2.0,arrowlength=1.4,arrowinset=0.5]{-<}(0.96875,0.004375)(0.44875,0.004375)(0.86875,-0.015625)(0.26875,0.004375)
\psbezier[linewidth=0.03,linestyle=dashed,dash=0.16cm 0.16cm,arrowsize=0.12cm 2.0,arrowlength=1.4,arrowinset=0.5]{->}(1.62875,-0.275625)(1.56875,-0.575625)(1.46875,-0.755625)(1.36875,-0.875625)(1.26875,-0.995625)(1.12875,-1.015625)(0.58875,-1.115625)
\psbezier[linewidth=0.03,linestyle=dashed,dash=0.16cm 0.16cm,arrowsize=0.12cm 2.0,arrowlength=1.4,arrowinset=0.5]{<-<}(0.60875,1.104375)(1.32875,0.864375)(1.76875,0.544375)(1.98875,-0.075625)(2.20875,-0.695625)(1.98875,-1.575625)(1.58875,-1.975625)
\psbezier[linewidth=0.03,linestyle=dashed,dash=0.16cm 0.16cm,arrowsize=0.12cm 2.0,arrowlength=1.4,arrowinset=0.5]{-<}(2.42875,0.644375)(2.88875,0.504375)(3.248423,0.48699465)(3.44875,0.524375)(3.649077,0.56175536)(4.06875,0.824375)(4.34875,1.124375)
\psbezier[linewidth=0.03,linestyle=dashed,dash=0.16cm 0.16cm,arrowsize=0.12cm 2.0,arrowlength=1.4,arrowinset=0.5]{<-<}(4.66875,0.004375)(3.90875,0.004375)(3.8783169,-0.13259)(3.80875,-0.275625)(3.739183,-0.41866)(3.78875,-0.815625)(4.32875,-1.115625)
\psbezier[linewidth=0.03,linestyle=dashed,dash=0.16cm 0.16cm,arrowsize=0.12cm 2.0,arrowlength=1.4,arrowinset=0.5]{->}(3.04875,-0.715625)(2.82875,-1.055625)(2.8627622,-1.0783677)(2.84875,-1.275625)(2.8347378,-1.4728824)(2.82875,-1.575625)(3.06875,-2.075625)
\psbezier[linewidth=0.05,arrowsize=0.05291667cm 4.0,arrowlength=1.4,arrowinset=0.4]{->}(5.54875,-0.105625)(6.14875,0.234375)(6.84875,0.294375)(7.20875,0.294375)(7.56875,0.294375)(8.14875,0.134375)(8.76875,-0.285625)
\pscircle[linewidth=0.04,dimen=outer](11.66875,0.014375){2.2}
\psdots[dotsize=0.12,dotstyle=square*](13.86875,0.014375)
\psdots[dotsize=0.12,dotstyle=square*](9.46875,0.014375)
\psdots[dotsize=0.12,dotstyle=square*](12.46875,2.054375)
\psdots[dotsize=0.12,dotstyle=square*](10.80875,2.034375)
\psdots[dotsize=0.12,dotstyle=square*](13.54875,1.134375)
\psdots[dotsize=0.12,dotstyle=square*](9.80875,1.114375)
\psdots[dotsize=0.12,dotstyle=square*](9.78875,-1.105625)
\psdots[dotsize=0.12,dotstyle=square*](10.78875,-1.965625)
\psdots[dotsize=0.12,dotstyle=square*](12.26875,-2.065625)
\psdots[dotsize=0.12,dotstyle=square*](13.52875,-1.105625)
\usefont{T1}{ptm}{m}{n}
\rput(14.028594,-0.085625){\small $\delta_0$}
\usefont{T1}{ptm}{m}{n}
\rput(13.748593,1.234375){\small $\delta_1$}
\usefont{T1}{ptm}{m}{n}
\rput(12.628593,2.254375){\small $\delta_2$}
\usefont{T1}{ptm}{m}{n}
\rput(10.708593,2.234375){\small $\delta_3$}
\usefont{T1}{ptm}{m}{n}
\rput(9.608594,1.294375){\small $\delta_4$}
\usefont{T1}{ptm}{m}{n}
\rput(9.268594,0.074375){\small $\delta_5$}
\usefont{T1}{ptm}{m}{n}
\rput(9.628593,-1.185625){\small $\delta_6$}
\usefont{T1}{ptm}{m}{n}
\rput(10.628593,-2.105625){\small $\delta_7$}
\usefont{T1}{ptm}{m}{n}
\rput(12.328594,-2.265625){\small $\delta_8$}
\usefont{T1}{ptm}{m}{n}
\rput(13.708593,-1.245625){\small $\delta_9$}
\psdots[dotsize=0.12](10.22875,0.014375)
\psdots[dotsize=0.12](10.70875,-0.565625)
\psdots[dotsize=0.12](13.50875,-0.425625)
\psdots[dotsize=0.12](12.50875,-0.525625)
\psdots[dotsize=0.12](11.84875,0.634375)
\psdots[dotsize=0.12](12.16875,1.414375)
\psbezier[linewidth=0.03,linestyle=dashed,dash=0.16cm 0.16cm](9.84875,1.134375)(10.34875,1.114375)(10.98875,0.434375)(11.14875,-0.065625)(11.30875,-0.565625)(11.12875,-1.425625)(10.82875,-1.945625)
\psbezier[linewidth=0.03,linestyle=dashed,dash=0.16cm 0.16cm](13.52875,-1.085625)(13.12875,-0.965625)(12.96875,-0.565625)(13.06875,-0.325625)(13.16875,-0.085625)(13.30875,0.114375)(13.84875,0.034375)
\psbezier[linewidth=0.03](9.58875,0.694375)(10.12875,0.594375)(10.58875,0.414375)(10.60875,0.034375)(10.62875,-0.345625)(10.28875,-0.625625)(9.60875,-0.645625)
\psbezier[linewidth=0.03](11.60875,2.194375)(11.52875,1.794375)(11.64875,1.174375)(12.00875,1.014375)(12.36875,0.854375)(12.92875,1.154375)(13.10875,1.674375)
\psbezier[linewidth=0.03](11.52875,-2.145625)(11.40875,-1.625625)(11.72875,-0.625625)(12.12875,-0.265625)(12.52875,0.094375)(13.14875,0.394375)(13.76875,0.514375)
\end{pspicture} 
}
} \\
\subfigure[hang,small][A case with a multiple point]{
\scalebox{0.7}{
\begin{pspicture}(0,-2.5103126)(14.319688,2.5103126)
\pscircle[linewidth=0.04,dimen=outer](2.46875,0.004375){2.2}
\psdots[dotsize=0.12,dotstyle=square*](4.66875,0.004375)
\psdots[dotsize=0.12,dotstyle=square*](0.26875,0.004375)
\psdots[dotsize=0.12,dotstyle=square*](3.26875,2.044375)
\psdots[dotsize=0.12,dotstyle=square*](1.60875,2.024375)
\psdots[dotsize=0.12,dotstyle=square*](4.34875,1.124375)
\psdots[dotsize=0.12,dotstyle=square*](0.60875,1.104375)
\psdots[dotsize=0.12,dotstyle=square*](0.58875,-1.115625)
\psdots[dotsize=0.12,dotstyle=square*](1.58875,-1.975625)
\psdots[dotsize=0.12,dotstyle=square*](3.06875,-2.075625)
\psdots[dotsize=0.12,dotstyle=square*](4.32875,-1.115625)
\usefont{T1}{ptm}{m}{n}
\rput(4.9021874,0.0946875){\small $\delta_0$}
\usefont{T1}{ptm}{m}{n}
\rput(4.5821873,1.3146875){\small $\delta_1$}
\usefont{T1}{ptm}{m}{n}
\rput(3.4021876,2.2846874){\small $\delta_2$}
\usefont{T1}{ptm}{m}{n}
\rput(1.4421875,2.2646874){\small $\delta_3$}
\usefont{T1}{ptm}{m}{n}
\rput(0.3421875,1.2946875){\small $\delta_4$}
\usefont{T1}{ptm}{m}{n}
\rput(0.02859375,0.064375){\small $\delta_5$}
\usefont{T1}{ptm}{m}{n}
\rput(0.3221875,-1.1753125){\small $\delta_6$}
\usefont{T1}{ptm}{m}{n}
\rput(1.4221875,-2.1353126){\small $\delta_7$}
\usefont{T1}{ptm}{m}{n}
\rput(3.1521875,-2.2853125){\small $\delta_8$}
\usefont{T1}{ptm}{m}{n}
\rput(4.5121873,-1.2253125){\small $\delta_9$}
\psdots[dotsize=0.12](4.00875,0.024375)
\psdots[dotsize=0.12](2.80875,0.024375)
\psdots[dotsize=0.12](0.88875,0.004375)
\psdots[dotsize=0.12](1.44875,-0.315625)
\psbezier[linewidth=0.03,linestyle=dashed,dash=0.16cm 0.16cm,arrowsize=0.12cm 2.0,arrowlength=1.4,arrowinset=0.5]{->}(1.44875,-0.315625)(1.40875,-0.555625)(1.30875,-0.815625)(1.18875,-0.855625)(1.06875,-0.895625)(0.98875,-1.035625)(0.58875,-1.115625)
\psbezier[linewidth=0.03,linestyle=dashed,dash=0.16cm 0.16cm,arrowsize=0.12cm 2.0,arrowlength=1.4,arrowinset=0.5]{>-}(0.26875,0.004375)(0.72875,0.004375)(0.46875,0.004375)(0.86875,0.024375)
\psbezier[linewidth=0.03,linestyle=dashed,dash=0.16cm 0.16cm,arrowsize=0.12cm 2.0,arrowlength=1.4,arrowinset=0.5]{<-<}(0.60875,1.104375)(1.04875,0.824375)(1.809771,0.33213678)(1.92875,-0.095625)(2.047729,-0.5233868)(1.92875,-1.375625)(1.58875,-1.975625)
\psbezier[linewidth=0.03,linestyle=dashed,dash=0.16cm 0.16cm,arrowsize=0.12cm 2.0,arrowlength=1.4,arrowinset=0.5]{>-}(1.60875,2.024375)(1.62875,1.724375)(1.84875,1.004375)(1.98875,0.824375)(2.12875,0.644375)(2.50875,0.164375)(2.80875,0.004375)
\psbezier[linewidth=0.03,linestyle=dashed,dash=0.16cm 0.16cm,arrowsize=0.12cm 2.0,arrowlength=1.4,arrowinset=0.5]{->}(2.80875,0.024375)(2.70875,-0.315625)(2.559331,-0.684035)(2.62875,-1.035625)(2.698169,-1.387215)(2.78875,-1.615625)(3.06875,-2.075625)
\psbezier[linewidth=0.03,linestyle=dashed,dash=0.16cm 0.16cm,arrowsize=0.12cm 2.0,arrowlength=1.4,arrowinset=0.5]{>-}(4.32875,-1.115625)(3.96875,-1.075625)(3.72875,-0.995625)(3.48875,-0.855625)(3.24875,-0.715625)(2.96875,-0.295625)(2.76875,0.024375)
\psbezier[linewidth=0.03,linestyle=dashed,dash=0.16cm 0.16cm,arrowsize=0.12cm 2.0,arrowlength=1.4,arrowinset=0.5]{<-}(3.26875,2.044375)(3.00875,1.844375)(2.78875,1.324375)(2.76875,1.084375)(2.74875,0.844375)(2.78875,0.504375)(2.76875,0.064375)
\psbezier[linewidth=0.03,linestyle=dashed,dash=0.16cm 0.16cm,arrowsize=0.12cm 2.0,arrowlength=1.4,arrowinset=0.5]{-<}(2.76875,0.044375)(2.96875,0.344375)(3.14875,0.524375)(3.34875,0.724375)(3.54875,0.924375)(3.70875,0.984375)(4.34875,1.124375)
\psbezier[linewidth=0.03,linestyle=dashed,dash=0.16cm 0.16cm,arrowsize=0.12cm 2.0,arrowlength=1.4,arrowinset=0.5]{->}(3.96875,0.024375)(4.6275,-0.015625)(4.24875,0.004375)(4.66875,0.004375)
\pscircle[linewidth=0.04,dimen=outer](11.82875,-0.085625){2.2}
\psdots[dotsize=0.12,dotstyle=square*](14.02875,-0.085625)
\psdots[dotsize=0.12,dotstyle=square*](9.62875,-0.085625)
\psdots[dotsize=0.12,dotstyle=square*](12.62875,1.954375)
\psdots[dotsize=0.12,dotstyle=square*](10.96875,1.934375)
\psdots[dotsize=0.12,dotstyle=square*](13.70875,1.034375)
\psdots[dotsize=0.12,dotstyle=square*](9.96875,1.014375)
\psdots[dotsize=0.12,dotstyle=square*](9.94875,-1.205625)
\psdots[dotsize=0.12,dotstyle=square*](10.94875,-2.065625)
\psdots[dotsize=0.12,dotstyle=square*](12.42875,-2.165625)
\psdots[dotsize=0.12,dotstyle=square*](13.68875,-1.205625)
\usefont{T1}{ptm}{m}{n}
\rput(14.188594,-0.185625){\small $\delta_0$}
\usefont{T1}{ptm}{m}{n}
\rput(13.908594,1.134375){\small $\delta_1$}
\usefont{T1}{ptm}{m}{n}
\rput(12.788593,2.154375){\small $\delta_2$}
\usefont{T1}{ptm}{m}{n}
\rput(10.868594,2.134375){\small $\delta_3$}
\usefont{T1}{ptm}{m}{n}
\rput(9.768594,1.194375){\small $\delta_4$}
\usefont{T1}{ptm}{m}{n}
\rput(9.428594,-0.025625){\small $\delta_5$}
\usefont{T1}{ptm}{m}{n}
\rput(9.788593,-1.285625){\small $\delta_6$}
\usefont{T1}{ptm}{m}{n}
\rput(10.788593,-2.205625){\small $\delta_7$}
\usefont{T1}{ptm}{m}{n}
\rput(12.488594,-2.365625){\small $\delta_8$}
\usefont{T1}{ptm}{m}{n}
\rput(13.868594,-1.345625){\small $\delta_9$}
\psdots[dotsize=0.12](13.36875,-0.065625)
\psdots[dotsize=0.12](12.16875,-0.065625)
\psdots[dotsize=0.12](10.24875,-0.085625)
\psdots[dotsize=0.12](10.80875,-0.405625)
\psbezier[linewidth=0.03,linestyle=dashed,dash=0.16cm 0.16cm](9.98875,1.014375)(10.40875,0.734375)(11.169771,0.24213678)(11.28875,-0.185625)(11.407729,-0.61338675)(11.28875,-1.465625)(10.96875,-2.005625)
\psbezier[linewidth=0.05,arrowsize=0.05291667cm 4.0,arrowlength=1.4,arrowinset=0.4]{->}(5.54875,-0.005625)(6.14875,0.334375)(6.84875,0.394375)(7.20875,0.394375)(7.56875,0.394375)(8.14875,0.234375)(8.76875,-0.185625)
\psbezier[linewidth=0.03](9.76875,0.574375)(10.26875,0.414375)(10.533863,0.3528676)(10.56875,0.014375)(10.603637,-0.3241176)(10.42875,-0.605625)(9.74875,-0.705625)
\psbezier[linewidth=0.03](13.88875,0.574375)(13.42875,0.534375)(12.82875,0.414375)(12.84875,-0.025625)(12.86875,-0.465625)(13.28875,-0.725625)(13.90875,-0.745625)
\end{pspicture}
}
}
\caption{Examples of separatrix graphs and these respective transversal graphs}
\label{exsepgraph}
\end{figure}

\vspace{2mm}

One of the remarkable properties due to the construction of these graphs is that each connected component of the disk bounded by $\Sigma_P$ contains one and only one equilibrium point of the polynomial vector field $\xi_P$. So it is easy to reconstruct the separatrix graph $\Gamma_P$ from the associated transversal graph $\Sigma_P$. In other words, we have a bijection between the set of separatrix graphs and the set of transversal graphs.

Moreover, the reader can see that for $P,Q \in \mathcal{P}_d$, $P \thicksim Q$ if and only if they have the same transversal graph, i.e. $\Sigma_P$ and $\Sigma_Q$ are constructed by pairing the same boards. Similarly, $P \thicksim_{top} Q$ if and only if there exists $l \in \{0, \ldots , 2d-3\}$ such that $\Sigma_Q$ is obtained by a rotation of $\Sigma_P$ of order $l/(2d-2)$ (see Figure \ref{extopbrarep} for an example)

\vspace{2mm}

In summary, in order to count the number of equivalent phase portraits of the polynomial differential equations $\dot{z} = P(z)$, we have to enumerate the number of transversal graphs, i.e. the number of different ways to pair the $2d-2$ boards of $\overline{\mathbb{D}}$. A first enumeration of this problem was given by A.Douady, F.Estrada and P.Sentenac who classified the structurally stable (or generic) case, i.e. the vector fields such that there are neither homoclinic separatrix nor multiple equilibrium point. Later this result was completed, thanks to the introduction of a new combinatorial, by the classification of the global structures of complex polynomial vector fields in $\mathbb{C}$. This generalization is due to a work of B.Branner and K.Dias \cite{Dias1}.

\subsection{From transversal graphs to valid bracketings.}

First let me explain the idea developed by K.Dias in \cite{Dias2}. This idea is to translate the combinatorial data set defined by the transversal graphs into a simpler bracketing problem in order to facilitate the enumeration of general polynomial vector fields. The bracketing problem proposed here is not exactly the same as the one proposed by K.Dias, but it's an equivalent problem. Consider a string of $2d-2$ elements including the integers from $0$ to $2d-3$, and use the following rules:
\begin{enumerate}
\item[1.] for each homoclinic separatrix $s_{i,j}$ of the transversal graph, replace the integers $i$ and $j$ by left and right round parentheses respectively, i.e. $i-1$ $($ $i+1 \ldots j-1$ $)$ $j+1$. These pair of parentheses are called associated.
\item[2.] for each transversal $T_{k,l}$ of the transversal graph, replace the integers $k$ and $l$ by square parentheses, i.e. $k-1$ $[$ $k+1 \ldots l-1$ $]$ $l+1$. These pair of parentheses are called associated.
\item[3.] for each board $\tilde{E}_i$ which is not connected with another one, replace the integer $i$ by a dot, i.e. $i-1$ $\bullet$ $i+1$.
\end{enumerate}
In this way, a transversal graph induces a unique bracketing representation. For example, the respective bracketing representations of the examples $(a)$, $(b)$ and $(c)$ of Figure \ref{exsepgraph} are respectively $[ \, ] \, [ \, [ \, [ \, ] \, [ \, ] \, ] \, ]$, $( \, [ \, [ \, ] \, ( \, [ \, ] \, ) \, ] \, )$ and $[ \, ] \, \bullet \, \bullet \, ( \, [ \, ] \, ) \, \bullet \, \bullet$. Figure \ref{extopbrarep} gives also an example of translation from transversal graphs to valid bracketings. Remark that we can also write a valid bracketing with the elements of the string (as K.Dias in her paper). For example, we can write $[01][23]4(56)7$ instead of $[ \, ] \, [ \, ] \, \bullet \, ( \, ) \, \bullet $.

\vspace{3mm}

Conversely, the bracketing representation must satisfy some properties so that they are in accordance with what can happen for a given polynomial vector field (and so for the associated transversal graph). For that, we need to impose some rules on how to place the elements of a bracketing representation. So consider a bracketing representation contained $2d-2$ elements, it is called a \textbf{valid bracketing} if:

\begin{enumerate}
\item[1.] there are an equal number of right and left round and square parentheses.
\item[2.] the number of left square (resp. round) parentheses must be greater than or equal to the number of right, reading from left to right.
\item[3.] there must be an even number of parentheses and dots between a pair of round (resp. square) associated parentheses.
\item[4.] there must be an equal number of right and left square (resp. round) parentheses between a pair of (round or square) associated parentheses.
\end{enumerate}

So, we create a bijection between the set of transversal graphs and the set of valid bracketings. Notice that the structurally stable case of vector fields (containing neither homoclinic separatrix nor sepal zone) is equivalent by this transformation to the classical bracketing problem whose the enumeration is given by the Catalan numbers \cite{catalan}. 

\vspace{2mm}

Thanks to this combinatorial given first in the work of K.Dias and B.Branner \cite{Dias1}, we can enumerate the equivalent classes defined by the relation $\thicksim$. This enumeration is given in lemma \ref{lemvalbr}, and is proved first by K.Dias in \cite{Dias2}. The main problem of this point of view is the following : this combinatorics depends on the enumeration of the points $\delta_l$ on the unit circle. In fact, Figure \ref{extopbrarep} shows an example of two topologically equivalent transversal graphs (one is obtained by a rotation of the other) having different respective bracketing representations. In other words, the problem here is that it's difficult to identify topologically equivalent polynomials using this combinatorics. In order to solve this difficulty, we will give in the next subsection another equivalent combinatorial data set. Before that, we give a generalization of the concept of valid bracketing defined for an odd number of elements in the string.

\begin{figure}[ht]
\begin{center}
\scalebox{0.75} 
{
\begin{pspicture}(0,-3.775)(13.2109375,3.775)
\rput{-270.0}(12.38125,-9.640625){\pscircle[linewidth=0.04,dimen=outer](11.010938,1.3703125){2.2}}
\psdots[dotsize=0.12,dotangle=-250.0,dotstyle=square*](11.010938,3.5703125)
\psdots[dotsize=0.12,dotangle=-250.0,dotstyle=square*](11.010938,-0.8296875)
\psdots[dotsize=0.12,dotangle=-250.0,dotstyle=square*](8.970938,2.1703124)
\psdots[dotsize=0.12,dotangle=-250.0,dotstyle=square*](8.990937,0.5103125)
\psdots[dotsize=0.12,dotangle=-250.0,dotstyle=square*](9.890938,3.2503126)
\psdots[dotsize=0.12,dotangle=-250.0,dotstyle=square*](9.910937,-0.4896875)
\psdots[dotsize=0.12,dotangle=-250.0,dotstyle=square*](12.130938,-0.5096875)
\psdots[dotsize=0.12,dotangle=-250.0,dotstyle=square*](12.990937,0.4903125)
\psdots[dotsize=0.12,dotangle=-250.0,dotstyle=square*](13.090938,1.9703125)
\psdots[dotsize=0.12,dotangle=-250.0,dotstyle=square*](12.130938,3.2303126)
\usefont{T1}{ptm}{m}{n}
\rput(11.010938,3.8203125){\small $\delta_3$}
\usefont{T1}{ptm}{m}{n}
\rput(11.010938,-1.1296875){\small $\delta_8$}
\usefont{T1}{ptm}{m}{n}
\rput(8.670938,2.3703124){\small $\delta_5$}
\usefont{T1}{ptm}{m}{n}
\rput(8.640937,0.5603125){\small $\delta_6$}
\usefont{T1}{ptm}{m}{n}
\rput(9.690938,3.5503126){\small $\delta_4$}
\usefont{T1}{ptm}{m}{n}
\rput(9.710937,-0.6896875){\small $\delta_7$}
\usefont{T1}{ptm}{m}{n}
\rput(12.400938,-0.7096875){\small $\delta_9$}
\usefont{T1}{ptm}{m}{n}
\rput(13.290937,0.4903125){\small $\delta_0$}
\usefont{T1}{ptm}{m}{n}
\rput(13.390938,2.1203125){\small $\delta_1$}
\usefont{T1}{ptm}{m}{n}
\rput(12.330938,3.4803126){\small $\delta_2$}
\psdots[dotsize=0.12,dotangle=-250.0](11.010938,-0.0696875)
\psdots[dotsize=0.12,dotangle=-250.0](11.590938,0.4103125)
\psdots[dotsize=0.12,dotangle=-250.0](11.450937,3.2103126)
\psdots[dotsize=0.12,dotangle=-250.0](11.550938,2.2103126)
\psdots[dotsize=0.12,dotangle=-250.0](10.390938,1.5503125)
\psdots[dotsize=0.12,dotangle=-250.0](9.610937,1.8703125)
\psbezier[linewidth=0.03,linestyle=dashed,dash=0.16cm 0.16cm](9.890938,-0.4496875)(9.910937,0.0503125)(10.590938,0.6903125)(11.090938,0.8503125)(11.590938,1.0103126)(12.450937,0.8303125)(12.970938,0.5303125)
\psbezier[linewidth=0.03,linestyle=dashed,dash=0.16cm 0.16cm](12.110937,3.2303126)(11.990937,2.8303125)(11.590938,2.6703124)(11.350938,2.7703125)(11.110937,2.8703125)(10.910937,3.0103126)(10.990937,3.5503125)
\psbezier[linewidth=0.03](10.330937,-0.7096875)(10.430938,-0.1696875)(10.610937,0.2903125)(10.990937,0.3103125)(11.370937,0.3303125)(11.650937,-0.0096875)(11.670938,-0.6896875)
\psbezier[linewidth=0.03](8.830937,1.3103125)(9.230938,1.2303125)(9.850938,1.3503125)(10.010938,1.7103125)(10.170938,2.0703125)(9.870937,2.6303124)(9.350938,2.8103125)
\psbezier[linewidth=0.03](13.170938,1.2303125)(12.650937,1.1103125)(11.650937,1.4303125)(11.290937,1.8303125)(10.930938,2.2303126)(10.630938,2.8503125)(10.510938,3.4703126)
\pscircle[linewidth=0.04,dimen=outer](2.46875,1.3590626){2.2}
\psdots[dotsize=0.12,dotstyle=square*](4.66875,1.3590626)
\psdots[dotsize=0.12,dotstyle=square*](0.26875,1.3590626)
\psdots[dotsize=0.12,dotstyle=square*](3.26875,3.3990624)
\psdots[dotsize=0.12,dotstyle=square*](1.60875,3.3790624)
\psdots[dotsize=0.12,dotstyle=square*](4.34875,2.4790626)
\psdots[dotsize=0.12,dotstyle=square*](0.60875,2.4590626)
\psdots[dotsize=0.12,dotstyle=square*](0.58875,0.2390625)
\psdots[dotsize=0.12,dotstyle=square*](1.58875,-0.6209375)
\psdots[dotsize=0.12,dotstyle=square*](3.06875,-0.7209375)
\psdots[dotsize=0.12,dotstyle=square*](4.32875,0.2390625)
\usefont{T1}{ptm}{m}{n}
\rput(4.8785937,1.2590625){\small $\delta_0$}
\usefont{T1}{ptm}{m}{n}
\rput(4.5485935,2.5790625){\small $\delta_1$}
\usefont{T1}{ptm}{m}{n}
\rput(3.4285936,3.5990624){\small $\delta_2$}
\usefont{T1}{ptm}{m}{n}
\rput(1.5085938,3.5790625){\small $\delta_3$}
\usefont{T1}{ptm}{m}{n}
\rput(0.40859374,2.6390624){\small $\delta_4$}
\usefont{T1}{ptm}{m}{n}
\rput(0.06859375,1.4190625){\small $\delta_5$}
\usefont{T1}{ptm}{m}{n}
\rput(0.42859375,0.1590625){\small $\delta_6$}
\usefont{T1}{ptm}{m}{n}
\rput(1.4285938,-0.7609375){\small $\delta_7$}
\usefont{T1}{ptm}{m}{n}
\rput(3.1285937,-0.9209375){\small $\delta_8$}
\usefont{T1}{ptm}{m}{n}
\rput(4.5085936,0.0990625){\small $\delta_9$}
\psdots[dotsize=0.12](1.02875,1.3590626)
\psdots[dotsize=0.12](1.50875,0.7790625)
\psdots[dotsize=0.12](4.30875,0.9190625)
\psdots[dotsize=0.12](3.30875,0.8190625)
\psdots[dotsize=0.12](2.64875,1.9790626)
\psdots[dotsize=0.12](2.96875,2.7590625)
\psbezier[linewidth=0.03,linestyle=dashed,dash=0.16cm 0.16cm](0.64875,2.4790626)(1.14875,2.4590626)(1.78875,1.7790625)(1.94875,1.2790625)(2.10875,0.7790625)(1.92875,-0.0809375)(1.62875,-0.6009375)
\psbezier[linewidth=0.03,linestyle=dashed,dash=0.16cm 0.16cm](4.32875,0.2590625)(3.92875,0.3790625)(3.76875,0.7790625)(3.86875,1.0190625)(3.96875,1.2590625)(4.10875,1.4590625)(4.64875,1.3790625)
\psbezier[linewidth=0.03](0.38875,2.0390625)(0.92875,1.9390625)(1.38875,1.7590625)(1.40875,1.3790625)(1.42875,0.9990625)(1.08875,0.7190625)(0.40875,0.6990625)
\psbezier[linewidth=0.03](2.40875,3.5390625)(2.32875,3.1390624)(2.44875,2.5190625)(2.80875,2.3590624)(3.16875,2.1990626)(3.72875,2.4990625)(3.90875,3.0190625)
\psbezier[linewidth=0.03](2.32875,-0.8009375)(2.20875,-0.2809375)(2.52875,0.7190625)(2.92875,1.0790625)(3.32875,1.4390625)(3.94875,1.7390625)(4.56875,1.8590626)
\usefont{T1}{ptm}{m}{n}
\rput(6.827344,1.3790625){\LARGE $\sim_{top}$}
\usefont{T1}{ptm}{m}{n}
\rput(2.5139062,-3.5059376){\large $( \, [ \, [ \, ] \, ( \, [ \, ] \, ) \, ] \, )$}
\usefont{T1}{ptm}{m}{n}
\rput(6.8890624,-3.4809375){\LARGE $\neq$}
\usefont{T1}{ptm}{m}{n}
\rput(11.055,-3.3659375){\large $( \, [ \, ( \, ) \, ] \, [ \, ] \, ) \, [ \, ]$}
\usefont{T1}{ptm}{m}{n}
\rput(2.410625,-1.9859375){\LARGE $\Updownarrow$}
\usefont{T1}{ptm}{m}{n}
\rput(11.090625,-1.9059376){\LARGE $\Updownarrow$}
\end{pspicture}
}
\end{center}
\caption{Examples of two topological equivalent vector fields with different bracketing representations}
\label{extopbrarep}
\end{figure}
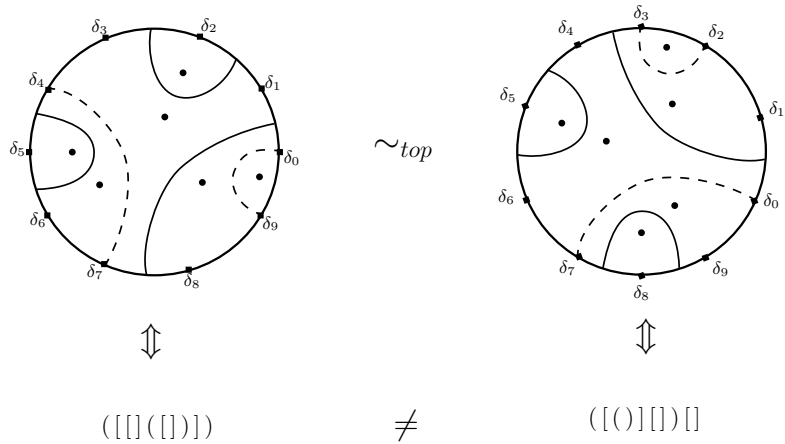

\vspace{2mm}

Consider a bracketing representation of $n$ elements. It is called a valid bracketing if:
\begin{enumerate}
\item[1.] there are an equal number of right and left square (resp. round) parentheses.
\item[2.] the number of left square (resp. round) parentheses must be greater than or equal to the number of right, reading from left to right.
\item[3.] there must be an even number of parentheses and dots between a pair of round (resp. square) associated parentheses, except maybe one.
\item[4.] there must be an equal number of right and left square (resp. round) parentheses between a pair of (round or square) associated parentheses.
\end{enumerate}
This notion of valid bracketing contained an odd number of elements has no relation with polynomial vector fields, but will be useful later.

\subsection{A step towards trees.}

The purpose of this subsection is to show a link between all combinatorial forms (transversal graphs, valid bracketings) presented so far in this article and a more general concept of trees. How we present this link will be very close to what can be found in the article of K.M.Pilgrim \cite{Pilg} making himself the connection between transversal graphs and combinatorial objects called Grothendieck dessins d'enfants. This presentation will be an opportunity to propose a simple encoding of combinatorics we work with so far. In the following, we treat separately the situations with multiple points, and the situations without multiple points. But before that, we will return soon on some basic notions about planar maps.

\vspace{2mm}

A \textbf{planar graph} is a graph that can be drawn in $\mathbb{R}^2$ without edge-crossing. A \textbf{planar map} is a proper embedding of a connected graph into the two-dimensional sphere considered up to orientation preserving homeomorphisms of the sphere. In other words, a planar map is a connected planar graph drawn in the sphere considered up to continuous deformation. 
So the difference between a map and a graph is that : a map has vertices and edges (like graph), and also faces. In the topological point of view, a planar graph is a CW complex of dimension $1$ whereas a planar map is a CW complex of dimension $2$.

Notice also that for a graph, to be able to be drawn in $\mathbb{R}^2$ is equivalent to be able to be drawn in $S^2$. This property is not true for a map. We will finish this introduction by the definition of a rooted planar map. Figure \ref{figexrplmap} shows an example of a rooted map.

\begin{defrmap}
\label{defrmap}
A planar map is called \textbf{rooted} by distinguishing an oriented edge. The original vertex of this distinguished oriented edge is called the root vertex of the map.
\end{defrmap}

\begin{figure}[ht]
\centering
\scalebox{0.5} 
{
\begin{pspicture}(0,-2.697699)(10.797812,2.697699)
\psdots[dotsize=0.12](1.9578125,2.062301)
\psdots[dotsize=0.12](1.9578125,-1.317699)
\psdots[dotsize=0.12](8.557813,2.1023011)
\psdots[dotsize=0.12](8.557813,-1.297699)
\psdots[dotsize=0.12](5.4978123,0.66230106)
\psdots[dotsize=0.12](10.717813,-2.617699)
\psbezier[linewidth=0.04](1.9578125,2.042301)(1.8178124,1.7623011)(1.6978126,0.82230103)(1.7178125,0.24230103)(1.7378125,-0.33769897)(1.7978125,-0.89769894)(1.9378124,-1.317699)
\psbezier[linewidth=0.04](1.9578125,2.0823011)(2.3978126,2.462301)(4.6979313,2.646903)(5.2378125,2.662301)(5.7776937,2.6776989)(8.077812,2.542301)(8.557813,2.142301)
\psbezier[linewidth=0.04](1.9378124,2.062301)(2.1178124,1.742301)(2.9018123,1.1034787)(3.3578124,0.92230105)(3.8138127,0.7411234)(4.8978124,0.68230104)(5.4778123,0.68230104)
\psbezier[linewidth=0.04](5.4978123,0.68230104)(6.1978126,0.642301)(6.7378125,0.902301)(7.1578126,1.1423011)(7.5778127,1.3823011)(8.117812,1.6623011)(8.577812,2.1023011)
\psbezier[linewidth=0.04](8.577812,2.1023011)(8.817813,1.582301)(8.837812,0.66230106)(8.817813,0.38230103)(8.797812,0.10230104)(8.817813,-0.83769894)(8.537812,-1.277699)
\psbezier[linewidth=0.04](8.517813,-1.297699)(8.317813,-1.6176989)(5.637806,-1.7940456)(5.2978125,-1.777699)(4.957819,-1.7613524)(2.2178125,-1.7176989)(1.9778125,-1.3376989)
\psbezier[linewidth=0.04](8.557813,-1.297699)(8.677813,-1.557699)(9.117812,-1.917699)(9.437813,-2.097699)(9.7578125,-2.277699)(10.317813,-2.597699)(10.697812,-2.637699)
\psbezier[linewidth=0.04](5.4978123,0.66230106)(5.5778127,0.30230105)(6.3388343,-0.4135336)(6.8778124,-0.69769895)(7.4167905,-0.98186433)(8.077812,-1.177699)(8.557813,-1.297699)
\psline[linewidth=0.04cm](1.7178125,0.10230104)(1.4578125,0.32230103)
\psline[linewidth=0.04cm](1.7178125,0.10230104)(1.9778125,0.34230104)
\usefont{T1}{ptm}{m}{n}
\rput(0.65,0.25230104){\large root edge}
\usefont{T1}{ptm}{m}{n}
\rput(0.9,2.192301){\large root vertex}
\end{pspicture} 
}
\caption{Example of a rooted planar map}
\label{figexrplmap}
\end{figure}
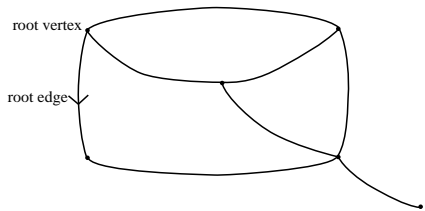

\subsubsection{Case without multiple points.}

Now, back to our combinatorial description. First consider a transversal graph $\Sigma$ associated with a polynomial vector field of degree $d$ without multiple points, i.e. for every index $i \in \{ 0, \dots , 2d-2 \}$ there exists an index $j \in \{ 0, \dots , 2d-2 \}$, $j \neq i$, so that the boards $\tilde{E}_i$ and $\tilde{E}_j$ are connected by a transversal $T_{i,j}$ or by a homoclinic separatrix $s_{i,j}$. From the results reported in the second section of the article, we deduce that the graph $\Sigma$ splits the closed unit disk $\overline{\mathbb{D}}$ into $d$ areas whose boundaries consist of transversals and homoclinic separatrices. Subsequently, we say that two areas are \textbf{adjacent} if their respective boundaries share a transversal or a homoclinic separatrix. The construction of an associated planar map $\mathcal{T}$ is then natural. First, in each area is arranged a vertex and then two vertices are connected to each other if their respective areas are adjacent. More precisely, two vertices are connected by a continuous edge (resp. dotted edge) if the boundaries of their respective areas share a common transversal (resp. homoclinic separatrix). Notice that, the intersection between the planar map $\mathcal{T}$ and the separatrix graph associated to $\Sigma$ is equal to the vertices of $\mathcal{T}$. See Figure \ref{figexeqtgvbgt1} for an example.

\vspace{2mm}

It's easy to see that the planar map $\mathcal{T}$ obtained is a map with only one face seeing him as the dual map of the transversal graph $\Sigma$ on the sphere. These particular maps are called planar trees (or just trees). Notice that our trees have the particularity to be defined by distinguishing two types of edges. Now, we need to translate the numeration of boards $\tilde{E}_i$. For that, we root our tree. Consider the oriented transversal $T_{0,j}$ (resp. homoclinic separatrix $s_{0,j}$), i.e. the transversal (resp. homoclinic separatrix) comes from $\tilde{E}_0$ to $\tilde{E}_j$, and denote by $x_0$ the vertex of $\mathcal{T}$ to the left of this transversal (resp. homoclinic separatrix). This vertex will be the root vertex of the tree.

Similarly, consider the oriented transversal $T_{1,k}$ (resp. homoclinic separatrix $s_{1,k}$), and denote by $x_1$ the vertex of $\mathcal{T}$ to the left of this transversal (resp. homoclinic separatrix). Then we define the root of the tree as the edge from $x_0$ to $x_1$.

\vspace{2mm}

So, we get an application that transforms a given transversal graph into a rooted tree. Conversely, consider a rooted tree $\mathcal{T}$ as defined above, we will show by another construction that this tree is linked to a valid bracketing (the same valid bracketing describing the transversal graph $\Sigma$ which comes from the tree). 

From the root vertex, and following the direction given by the root, we will follow the contour of the tree in the counterclockwise direction, stopping at each vertex encountered. In doing so, we will go through each edge exactly twice. So, following the contour defined later, we will write in a string of elements a left square parenthesis $[$ whenever a continuous edge is passed for the first time, and a right one "$]$" whenever a continuous edge is passed for the second time. Similarly, we will write in the same string a left round parenthesis "$($" whenever a dotted edge is passed for the first time, and a right one "$)$" whenever a dotted edge is passed for the second time. This is a classical encoding for trees. See \cite{tree} for more details on this encoding. \newline
It's easy to see that this transformation construct a valid bracketing. It's left to the reader to verify that the valid bracketing obtained thanks to the rooted tree is the same as that obtained directly through the transversal graph $\Sigma$. An example is given at Figure \ref{figexeqtgvbgt1}.

\begin{figure}[ht]
\centering
\scalebox{0.7} 
{
\begin{pspicture}(0,-4.5781555)(16.295,4.5931554)
\psbezier[linewidth=0.05,arrowsize=0.05291667cm 4.0,arrowlength=1.4,arrowinset=0.4]{->}(5.32875,2.237218)(5.92875,2.577218)(6.62875,2.637218)(6.98875,2.637218)(7.34875,2.637218)(7.92875,2.477218)(8.54875,2.057218)
\pscircle[linewidth=0.04,dimen=outer](2.46875,2.177218){2.2}
\psdots[dotsize=0.12,dotstyle=square*](4.66875,2.177218)
\psdots[dotsize=0.12,dotstyle=square*](0.26875,2.177218)
\psdots[dotsize=0.12,dotstyle=square*](3.26875,4.217218)
\psdots[dotsize=0.12,dotstyle=square*](1.60875,4.197218)
\psdots[dotsize=0.12,dotstyle=square*](4.34875,3.297218)
\psdots[dotsize=0.12,dotstyle=square*](0.60875,3.2772179)
\psdots[dotsize=0.12,dotstyle=square*](0.58875,1.057218)
\psdots[dotsize=0.12,dotstyle=square*](1.58875,0.19721797)
\psdots[dotsize=0.12,dotstyle=square*](3.06875,0.09721797)
\psdots[dotsize=0.12,dotstyle=square*](4.32875,1.057218)
\usefont{T1}{ptm}{m}{n}
\rput(4.8285937,2.077218){\small $\delta_0$}
\usefont{T1}{ptm}{m}{n}
\rput(4.5485935,3.397218){\small $\delta_1$}
\usefont{T1}{ptm}{m}{n}
\rput(3.4285936,4.417218){\small $\delta_2$}
\usefont{T1}{ptm}{m}{n}
\rput(1.5085938,4.3972178){\small $\delta_3$}
\usefont{T1}{ptm}{m}{n}
\rput(0.40859374,3.457218){\small $\delta_4$}
\usefont{T1}{ptm}{m}{n}
\rput(0.06859375,2.237218){\small $\delta_5$}
\usefont{T1}{ptm}{m}{n}
\rput(0.42859375,0.977218){\small $\delta_6$}
\usefont{T1}{ptm}{m}{n}
\rput(1.4285938,0.05721797){\small $\delta_7$}
\usefont{T1}{ptm}{m}{n}
\rput(3.1285937,-0.10278203){\small $\delta_8$}
\usefont{T1}{ptm}{m}{n}
\rput(4.5085936,0.91721797){\small $\delta_9$}
\psdots[dotsize=0.12](1.02875,2.177218)
\psdots[dotsize=0.12](1.50875,1.5972179)
\psdots[dotsize=0.12](4.30875,1.737218)
\psdots[dotsize=0.12](3.30875,1.637218)
\psdots[dotsize=0.12](2.64875,2.797218)
\psdots[dotsize=0.12](2.96875,3.577218)
\psbezier[linewidth=0.03,linestyle=dashed,dash=0.16cm 0.16cm](0.64875,3.297218)(1.14875,3.2772179)(1.78875,2.597218)(1.94875,2.097218)(2.10875,1.5972179)(1.92875,0.73721796)(1.62875,0.21721797)
\psbezier[linewidth=0.03,linestyle=dashed,dash=0.16cm 0.16cm](4.32875,1.0772179)(3.92875,1.197218)(3.76875,1.5972179)(3.86875,1.8372179)(3.96875,2.077218)(4.10875,2.2772179)(4.64875,2.197218)
\psbezier[linewidth=0.03](0.38875,2.857218)(0.92875,2.757218)(1.38875,2.577218)(1.40875,2.197218)(1.42875,1.817218)(1.08875,1.537218)(0.40875,1.517218)
\psbezier[linewidth=0.03](2.40875,4.357218)(2.32875,3.957218)(2.44875,3.337218)(2.80875,3.177218)(3.16875,3.0172179)(3.72875,3.317218)(3.90875,3.837218)
\psbezier[linewidth=0.03](2.32875,0.01721797)(2.20875,0.537218)(2.52875,1.537218)(2.92875,1.897218)(3.32875,2.257218)(3.94875,2.557218)(4.56875,2.677218)
\psdots[dotsize=0.12](9.94875,2.137218)
\psdots[dotsize=0.12](10.42875,1.557218)
\psdots[dotsize=0.12](13.22875,1.697218)
\psdots[dotsize=0.12](12.22875,1.5972179)
\psdots[dotsize=0.12](11.56875,2.757218)
\psdots[dotsize=0.12](11.88875,3.5372179)
\psline[linewidth=0.04cm](9.96875,2.117218)(10.38875,1.5972179)
\psline[linewidth=0.04cm,linestyle=dashed,dash=0.16cm 0.16cm](10.42875,1.5772179)(11.54875,2.757218)
\psline[linewidth=0.04cm](11.58875,2.7772179)(11.86875,3.5372179)
\psline[linewidth=0.04cm](11.58875,2.737218)(12.20875,1.617218)
\psline[linewidth=0.04cm,linestyle=dashed,dash=0.16cm 0.16cm](12.24875,1.5972179)(13.20875,1.697218)
\psline[linewidth=0.04cm](12.62875,1.617218)(12.78875,1.877218)
\psline[linewidth=0.04cm](12.64875,1.617218)(12.84875,1.4572179)

\usefont{T1}{ptm}{m}{n}
\rput(7.9959373,-3.3359375){\large $( \, [ \, [ \, ] \, ( \, [ \, ] \, ) \, ] \, )$}
\psbezier[linewidth=0.05,arrowsize=0.12cm 2.0,arrowlength=1.4,arrowinset=0.4]{->}(6.76875,-3.3109374)(5.94875,-3.2109375)(5.54875,-2.9309375)(5.18875,-2.7109375)(4.82875,-2.4909375)(4.46875,-2.1909375)(4.08875,-1.4909375)
\psbezier[linewidth=0.05,arrowsize=0.12cm 2.0,arrowlength=1.4,arrowinset=0.4]{->}(11.92875,-0.4309375)(11.84875,-1.0909375)(11.64875,-1.9109375)(11.50875,-2.1509376)(11.36875,-2.3909376)(10.76875,-3.2909374)(9.86875,-3.2709374)
\end{pspicture} 
}
\caption{Example of equivalence between a transversal graph without multiple points and its respective valid bracketing and generalized tree.}
\label{figexeqtgvbgt1}
\end{figure}
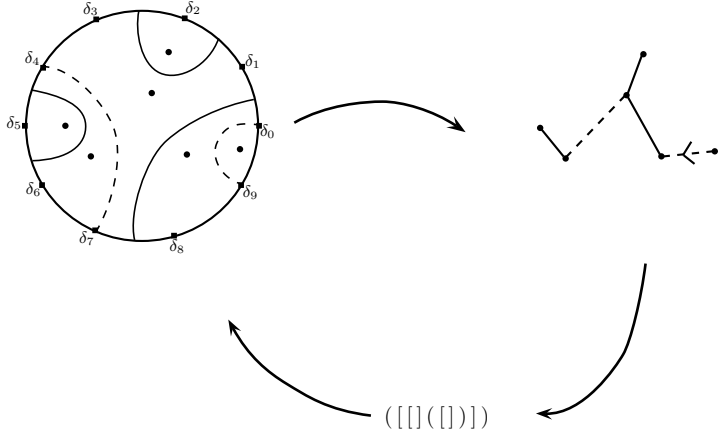

\subsubsection{Case with multiple points.}

Now, consider a transversal graph $\Sigma'$ associated with a polynomial vector field of degree $d$ with multiple point(s), i.e. there is at least one index $i \in \{ 0, \dots , 2d-2 \}$ such that the board $\tilde{E}_i$ is not connected to another one. In this case, the graph $\Sigma'$ splits the closed unit disk into $k<d$ areas whose boundaries consist of transversals and homoclinic separatrices. In first, we perform the same construction as in the previous case : in each area is arranged a vertex and two vertices are connected by a continuous (resp. dotted) edge to each other if their respective areas share a common transversal (resp. homoclinic separatrix). Then, we must take in consideration the boards which are not connected with other boards. It's important to notice that each of these boards is entirely contained in an area. So, we can draw a (continuous) edge from a point of a board which is not connected with another one to the vertex associated to the area containing this board. In this case, our map $\mathcal{T}'$ has the particularity to have some edges connected by only one vertex. These edges are called \textbf{half-edges}. 

\vspace{2mm}

So, from the transversal graph $\Sigma'$, we construct a planar map $\mathcal{T}'$, similar to the previous one, which contain some half-edges. In the following, these trees will be called \textbf{generalized trees}. As above, now we will root the map $\mathcal{T}'$ as follows: denote by $x_0$ the vertex defined as the vertex to the left of the oriented transversal $T_{0,j}$ (resp. homoclinic separatrix $s_{0,j}$) or as the vertex associated to the area containing entirely the end $\tilde{E}_0$ if $\tilde{E}_0$ is not connected with another end. This vertex will be the root vertex of the tree $\mathcal{T}'$. 
Similarly, note $x_1$ the vertex defined as the vertex to the left of the oriented transversal $T_{1,k}$ (resp. homoclinic separatrix $s_{1,k}$) or as the vertex associated to the area containing entirely the end $\tilde{E}_1$. Then we define the root of $\mathcal{T}'$ as the edge from $x_0$ to $x_1$.

\vspace{3mm}

Conversely, consider a rooted generalized tree $\mathcal{T}'$. We will use the same encoding as before with only one difference. Because of the existence of half-edges, it's possible, when we follow the contour of the map, to go from a vertex to itself. In this case, we will simply add a unique dot $\bullet$ in the string of elements. In other words, we will suppose that half-edges are passed only once (and not two as for edges). Through the same transformation as above, it's easy to see that this encoding is equivalent to a valid bracketing and that the valid bracketing obtained thanks to $\mathcal{T}'$ is the same as that obtained directly through the transversal graph $\Sigma'$. An example is given at Figure \ref{figexeqtgvbgt2}.

\begin{figure}[ht]
\centering
\scalebox{0.7} 
{
\begin{pspicture}(0,-4.5781555)(15.733125,4.5931554)
\psbezier[linewidth=0.05,arrowsize=0.05291667cm 4.0,arrowlength=1.4,arrowinset=0.4]{->}(5.74875,2.237218)(6.34875,2.577218)(7.04875,2.637218)(7.40875,2.637218)(7.76875,2.637218)(8.34875,2.477218)(8.96875,2.057218)
\pscircle[linewidth=0.04,dimen=outer](2.46875,2.177218){2.2}
\psdots[dotsize=0.12,dotstyle=square*](4.66875,2.177218)
\psdots[dotsize=0.12,dotstyle=square*](0.26875,2.177218)
\psdots[dotsize=0.12,dotstyle=square*](3.26875,4.217218)
\psdots[dotsize=0.12,dotstyle=square*](1.60875,4.197218)
\psdots[dotsize=0.12,dotstyle=square*](4.34875,3.297218)
\psdots[dotsize=0.12,dotstyle=square*](0.60875,3.2772179)
\psdots[dotsize=0.12,dotstyle=square*](0.58875,1.057218)
\psdots[dotsize=0.12,dotstyle=square*](1.58875,0.19721797)
\psdots[dotsize=0.12,dotstyle=square*](3.06875,0.09721797)
\psdots[dotsize=0.12,dotstyle=square*](4.32875,1.057218)
\usefont{T1}{ptm}{m}{n}
\rput(4.8285937,2.077218){\small $\delta_0$}
\usefont{T1}{ptm}{m}{n}
\rput(4.5485935,3.397218){\small $\delta_1$}
\usefont{T1}{ptm}{m}{n}
\rput(3.4285936,4.417218){\small $\delta_2$}
\usefont{T1}{ptm}{m}{n}
\rput(1.5085938,4.3972178){\small $\delta_3$}
\usefont{T1}{ptm}{m}{n}
\rput(0.40859374,3.457218){\small $\delta_4$}
\usefont{T1}{ptm}{m}{n}
\rput(0.06859375,2.237218){\small $\delta_5$}
\usefont{T1}{ptm}{m}{n}
\rput(0.42859375,0.977218){\small $\delta_6$}
\usefont{T1}{ptm}{m}{n}
\rput(1.4285938,0.05721797){\small $\delta_7$}
\usefont{T1}{ptm}{m}{n}
\rput(3.1285937,-0.10278203){\small $\delta_8$}
\usefont{T1}{ptm}{m}{n}
\rput(4.5085936,0.91721797){\small $\delta_9$}
\psdots[dotsize=0.12](1.02875,2.177218)
\psdots[dotsize=0.12](1.50875,1.5972179)
\psdots[dotsize=0.12](4.14875,2.177218)
\psdots[dotsize=0.12](2.84875,2.2772179)
\psbezier[linewidth=0.03,linestyle=dashed,dash=0.16cm 0.16cm](0.64875,3.297218)(1.14875,3.2772179)(1.78875,2.597218)(1.94875,2.097218)(2.10875,1.5972179)(1.92875,0.73721796)(1.62875,0.21721797)
\psbezier[linewidth=0.03](4.52875,1.417218)(4.02875,1.617218)(3.76875,1.937218)(3.82875,2.237218)(3.88875,2.5372179)(4.02875,2.7772179)(4.52875,2.837218)
\psbezier[linewidth=0.03](0.38875,2.857218)(0.92875,2.757218)(1.38875,2.577218)(1.40875,2.197218)(1.42875,1.817218)(1.08875,1.537218)(0.40875,1.517218)
\psdots[dotsize=0.12](10.02875,2.097218)
\psdots[dotsize=0.12](10.50875,1.517218)
\psdots[dotsize=0.12](13.14875,2.097218)
\psdots[dotsize=0.12](11.84875,2.197218)
\psline[linewidth=0.04cm](10.02875,2.117218)(10.48875,1.537218)
\psline[linewidth=0.04cm,linestyle=dashed,dash=0.16cm 0.16cm](10.50875,1.537218)(11.82875,2.197218)
\psline[linewidth=0.04cm](11.86875,2.217218)(13.12875,2.097218)
\psline[linewidth=0.04cm](11.86875,2.237218)(12.38875,3.157218)
\psline[linewidth=0.04cm](11.84875,2.237218)(11.28875,3.177218)
\psline[linewidth=0.04cm](11.84875,2.217218)(11.44875,1.177218)
\psline[linewidth=0.04cm](11.84875,2.197218)(12.52875,1.277218)
\psline[linewidth=0.04cm](12.60875,2.337218)(12.74875,2.137218)
\psline[linewidth=0.04cm](12.74875,2.117218)(12.56875,1.9572179)

\usefont{T1}{ptm}{m}{n}
\rput(8.2959373,-3.3359375){\large $[ \, ] \, \bullet \, \bullet \, ( \, [ \, ] \, ) \, \bullet \, \bullet$}
\psbezier[linewidth=0.05,arrowsize=0.12cm 2.0,arrowlength=1.4,arrowinset=0.4]{->}(6.76875,-3.3109374)(5.94875,-3.2109375)(5.54875,-2.9309375)(5.18875,-2.7109375)(4.82875,-2.4909375)(4.46875,-2.1909375)(4.08875,-1.4909375)
\psbezier[linewidth=0.05,arrowsize=0.12cm 2.0,arrowlength=1.4,arrowinset=0.4]{->}(11.92875,-0.4309375)(11.84875,-1.0909375)(11.64875,-1.9109375)(11.50875,-2.1509376)(11.36875,-2.3909376)(10.76875,-3.2909374)(9.86875,-3.2709374)
\end{pspicture} 
}
\caption{Example of equivalence between a transversal graph with a multiple point and its respective valid bracketing and generalized tree.}
\label{figexeqtgvbgt2}
\end{figure}
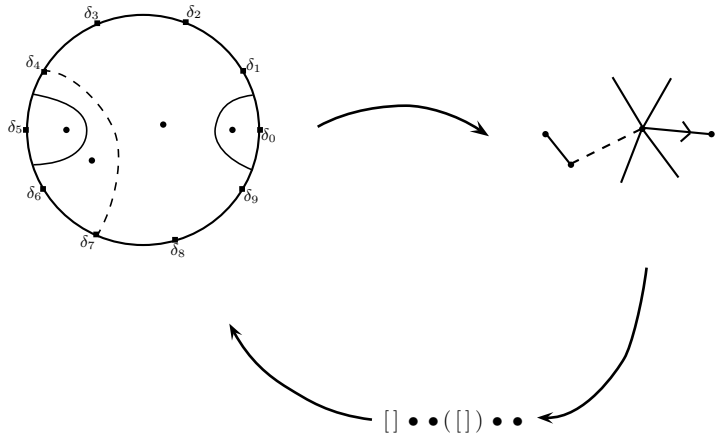

In conclusion, we have established a bijection between the set of transversal graphs and the set of rooted generalized trees. By taking quotients we also obtain a bijection between the set of transversal graphs quotiented by the group generated by a rotation of the unit circle of order $1/(2d-2)$ and the set of unrooted generalized tree.

\section{Enumeration}

\subsection{A method to enumerate unrooted maps.}

Results of the last section imply that counting the number of polynomial vector fields to topological equivalence amounts to count the number of unrooted generalized trees. In general, enumerate unrooted maps is a difficult problem due to the presence of various symmetries. Here, we use a method developed by V.A.Liskovets to overcome this difficulty. Only main arguments of this method will be given, for more details see \cite{Lisken} and \cite{Lisk}. The interested reader can also see \cite{Lisken} for an application and \cite{Liskexen} for some results obtained thanks to that method. Notice that results we use here are extensions of results proved by V.A.Liskovets because we work with a generalized concept of maps, due to the presence of half-edges, but demonstrations of these results are the same. In the following, we only work with a concept of generalized maps.

\vspace{3mm}

For the problem under consideration, let $\mathcal{M}$ be a certain set of maps described in a given surface, and $\mathcal{M}(n) \subset \mathcal{M}$ be the subset of $\mathcal{M}$ containing all maps of $\mathcal{M}$ with exactly $n$ edges ($n \geq 2$). Moreover, we suppose that the surface where we construct our maps is orientable and has a given orientation. Then we can calculate the number of unrooted maps knowing the number of $l$-rooted maps (see Definition \ref{deflrmap}). More precisely, we have the following result:

\begin{thenmap}
\label{thenmap}
\begin{displaymath}
M^+(n) = \frac{1}{2n} \sum_{l \geq 1 \atop l | 2n}{\varphi(l).M^{(+,l)}(n)}, \quad n \geq 2
\end{displaymath}
where $M^+(n)$ (respectively $M^{(+,l)}(n)$) is the number of non-isomorphic (respectively $l$-rooted) maps in $\mathcal{M}(n)$ considered up to orientation-preserving homeomorphisms and $\varphi$ is the Euler totient function.
\end{thenmap}

Notice that $M^{+,1}(n)$ is simply the number of rooted maps in $\mathcal{M}(n)$, this set is generally noted $M'(n)$. We will use this notation in the following.

This theorem, obtained thanks to some algebraic results on the group action theory, is true in all orientable surfaces, but have a more interesting interpretation in the case of the sphere $S^2$. In fact, according to a result of Mani \cite{Mani}, any planar map may be represented on the (geometrical) sphere in such a way that all its automorphisms are induced by symmetries of the sphere. On the other hand, as we only consider up to orientation-preserving automorphisms, this result implies that all symmetries of a given planar map, described on the sphere with a good representation, can be deduced by rotations of the sphere. This observation will induce later the concept of quotient maps that will be very useful to enumerate the $l$-rooted maps which is at first sight a difficult problem. But first we need to define the notion of $l$-rooted map.

\begin{deflrmap}
\label{deflrmap}
A $l$-rooted map is a map with $l$ rooted edges that exhibits a symmetry of order $l$ (with respect to the roots).
\end{deflrmap}

The definition of $l$-rooted maps is a natural generalization of the concept of rooted maps introduced in Definition \ref{defrmap}. The reader can find an example of a $3$-rooted map in Figure \ref{figex3rplmap} and an example of a $2$-rooted map in Figure \ref{figex2rmquo}.

\begin{figure}[ht]
\centering
\scalebox{0.7} 
{
\begin{pspicture}(0,-2.16)(4.12,2.16)
\psdots[dotsize=0.12](2.06,2.08)
\psdots[dotsize=0.12](2.06,0.1)
\psdots[dotsize=0.12](4.04,-1.48)
\psdots[dotsize=0.12](0.06,-1.5)
\psbezier[linewidth=0.04](2.06,2.1)(2.6,1.82)(2.9160924,1.2658379)(3.16,0.9)(3.4039075,0.5341621)(4.04,-0.62)(4.06,-1.48)
\psbezier[linewidth=0.04](4.04,-1.48)(3.74,-1.88)(2.54,-1.92)(2.08,-1.92)(1.62,-1.92)(0.4,-1.96)(0.04,-1.52)
\psbezier[linewidth=0.04](0.04,-1.52)(0.08,-0.9)(0.58832383,0.32922187)(0.84,0.78)(1.0916761,1.2307781)(1.52,1.8)(2.04,2.1)
\psbezier[linewidth=0.04](2.04,2.08)(2.04,1.28)(2.04,0.76)(2.06,0.12)
\psbezier[linewidth=0.04](2.06,0.12)(2.38,-0.36)(2.7973547,-0.7575322)(3.02,-0.92)(3.2426453,-1.0824678)(3.5,-1.24)(4.02,-1.44)
\psbezier[linewidth=0.04](2.04,0.1)(1.84,-0.28)(1.44,-0.7)(1.28,-0.84)(1.12,-0.98)(0.52,-1.32)(0.06,-1.48)
\psline[linewidth=0.04cm](0.8,0.68)(0.9,0.36)
\psline[linewidth=0.04cm](0.8,0.7)(0.46,0.54)
\psline[linewidth=0.04cm](3.56,0.68)(3.48,0.36)
\psline[linewidth=0.04cm](3.48,0.36)(3.14,0.48)
\psline[linewidth=0.04cm](1.8,-1.94)(2.0,-1.68)
\psline[linewidth=0.04cm](1.8,-1.92)(2.04,-2.14)
\end{pspicture} 
}
\caption{Example of a $3$-rooted planar map}
\label{figex3rplmap}
\end{figure}
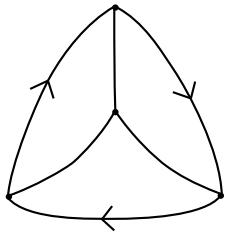

Now back to the concept of quotient maps mentioned earlier. We will define this notion in a geometrical point of view. Let $A$ be a planar map admitting a symmetry of order $l$, and consider a representation of this map on the sphere such that the symmetry of $A$ is induced by a rotation $\rho$ of the sphere of order $l$. Then, we construct the quotient map $B$ of $A$, with respect to the rotation $\rho$, by cutting out a spherical sector with angle $2\pi /l$ bearing upon the poles (i.e. the intersection between the rotation axis and the sphere) and then by glueing its boundary half-circles to form a sphere so that the quotient map is also a planar map. Figure \ref{exquograph} shows an example of the construction of a quotient map. For the sake of clarity, other examples of quotient map will be represented  on the plane.

\vspace{2mm}

Notice that the quotient map is independent of the choice of the sector because of the symmetry of the map $A$. Moreover, if $A$ is rooted, then we will choose a sector that contains the root, so that the quotient map $B$ is also rooted. Similarly, if $A$ is a $l$-rooted map, its quotient map, with respect to the symmetry of $A$ of order $l$, is a rooted map (due to the definition of a $l$-rooted map which implies that each sector contains one and only one root of the map $A$). In particular, it is for this reason that the concept of quotient map will be a great help later.

Notice also that the poles of the map $A$ are still contained on the quotient map $B$. The vertices, edges or faces (called cells) of the map $A$ (resp. the vertices, half-edges or faces of the quotient map $B$) which contain these poles are called \textbf{axial}.

\begin{figure}[ht]
\begin{center}
\scalebox{0.7} 
{
\begin{pspicture}(0,-5.59)(15.5075,5.59)
\pscircle[linewidth=0.04,dimen=outer](2.4,2.73){2.4}
\psbezier[linewidth=0.04](0.04,2.73)(0.4,2.29)(1.44,2.13)(2.44,2.13)(3.44,2.13)(4.16,2.27)(4.8,2.73)
\psbezier[linewidth=0.04,linestyle=dotted,dotsep=0.16cm](4.8,2.73)(4.44,3.17)(3.4,3.33)(2.4,3.33)(1.4,3.33)(0.68,3.19)(0.04,2.73)
\psdots[dotsize=0.2,dotstyle=x](3.42,3.67)
\psdots[dotsize=0.2,dotstyle=x](1.48,1.53)
\psline[linewidth=0.04cm,linestyle=dotted,dotsep=0.16cm](3.3,3.57)(0.88,0.85)
\psline[linewidth=0.04cm,linestyle=dashed,dash=0.16cm 0.16cm](3.52,3.77)(5.16,5.57)
\psline[linewidth=0.04cm,linestyle=dashed,dash=0.16cm 0.16cm](0.92,0.89)(0.44,0.29)
\psdots[dotsize=0.12](2.6,4.17)
\psdots[dotsize=0.12](3.82,4.19)
\psdots[dotsize=0.12](3.8,2.93)
\psdots[dotsize=0.12](2.62,2.91)
\psdots[dotsize=0.12](4.36,3.57)
\psdots[dotsize=0.12](2.0,3.57)
\psdots[dotsize=0.12](3.14,4.83)
\psdots[dotsize=0.12](3.34,2.29)
\psbezier[linewidth=0.04](2.6,4.17)(2.82,4.21)(3.0430474,4.2319884)(3.2,4.23)(3.3569524,4.2280116)(3.62,4.21)(3.82,4.19)
\psbezier[linewidth=0.04](3.82,4.21)(3.9,3.97)(3.96,3.71)(3.96,3.57)(3.96,3.43)(3.94,3.15)(3.82,2.95)
\psbezier[linewidth=0.04](2.62,2.93)(2.78,2.87)(3.0800092,2.8257174)(3.24,2.83)(3.3999908,2.8342826)(3.68,2.85)(3.8,2.93)
\psbezier[linewidth=0.04](2.6,2.93)(2.7,3.15)(2.7,3.37)(2.7,3.55)(2.7,3.73)(2.68,4.03)(2.58,4.17)
\psbezier[linewidth=0.04](3.14,4.83)(3.3,4.75)(3.42,4.69)(3.52,4.59)(3.62,4.49)(3.72,4.39)(3.82,4.21)
\psbezier[linewidth=0.04](2.6,4.19)(2.68,4.35)(2.74,4.51)(2.84,4.61)(2.94,4.71)(3.02,4.77)(3.12,4.85)
\psbezier[linewidth=0.04](3.82,4.21)(3.98,4.11)(4.06,4.05)(4.18,3.93)(4.3,3.81)(4.3,3.75)(4.36,3.59)
\psbezier[linewidth=0.04](3.8,2.95)(3.98,2.95)(4.14,3.11)(4.22,3.21)(4.3,3.31)(4.38,3.41)(4.36,3.57)
\psbezier[linewidth=0.04](3.8,2.93)(3.78,2.75)(3.72,2.67)(3.62,2.55)(3.52,2.43)(3.46,2.41)(3.32,2.29)
\psbezier[linewidth=0.04](3.32,2.31)(3.22,2.426129)(3.0604649,2.556581)(2.98,2.6390324)(2.8995352,2.7214835)(2.7,2.8325806)(2.6,2.91)
\psbezier[linewidth=0.04](2.6,2.91)(2.46,3.09)(2.4,3.17)(2.32,3.25)(2.24,3.33)(2.1,3.47)(1.98,3.57)
\psbezier[linewidth=0.04](1.98,3.57)(2.04,3.71)(2.1494231,3.8481598)(2.24,3.93)(2.330577,4.0118403)(2.44,4.09)(2.58,4.19)
\usefont{T1}{ptm}{m}{n}
\rput(5.842656,5.26){rotation axis}
\psbezier[linewidth=0.05,arrowsize=0.12cm 2.0,arrowlength=1.4,arrowinset=0.4]{->}(6.62,2.71)(6.94,3.07)(8.141643,3.1872988)(8.74,3.17)(9.338357,3.1527011)(10.14,2.99)(10.56,2.55)
\usefont{T1}{ptm}{m}{n}
\rput(8.480625,2.7){cutting}
\usefont{T1}{ptm}{m}{n}
\rput(8.454687,2.34){a sector}
\psdots[dotsize=0.2,dotstyle=x](12.56,4.67)
\psdots[dotsize=0.2,dotstyle=x](12.62,0.07)
\psbezier[linewidth=0.04,linestyle=dashed,dash=0.16cm 0.16cm](12.56,0.07)(12.12,0.33)(11.317588,1.1701547)(11.32,2.31)(11.3224125,3.4498453)(12.14,4.41)(12.52,4.67)
\psbezier[linewidth=0.04,linestyle=dashed,dash=0.16cm 0.16cm](12.622413,4.67)(13.08,4.41)(13.86,3.49)(13.862412,2.43)(13.864825,1.37)(13.08,0.33)(12.662413,0.07)
\psbezier[linewidth=0.04,linestyle=dashed,dash=0.16cm 0.16cm](11.36,2.41)(11.576508,2.11)(12.324515,2.0579998)(12.659048,2.05)(12.99358,2.0420003)(13.466032,2.11)(13.84,2.41)
\psline[linewidth=0.04cm,linestyle=dotted,dotsep=0.16cm](11.34,2.41)(12.6,2.73)
\psline[linewidth=0.04cm,linestyle=dotted,dotsep=0.16cm](12.6,2.73)(13.82,2.41)
\psdots[dotsize=0.12](12.58,3.07)
\psdots[dotsize=0.12](11.4,2.91)
\psdots[dotsize=0.12](13.84,2.77)
\psbezier[linewidth=0.04](11.42,2.91)(11.6,2.89)(11.810802,2.7965872)(12.02,2.79)(12.2291975,2.783413)(12.52,2.89)(12.58,3.09)
\psbezier[linewidth=0.04](12.58,3.07)(12.7,2.81)(12.94,2.73)(13.1,2.73)(13.26,2.73)(13.62,2.71)(13.82,2.77)
\psbezier[linewidth=0.04](11.96,4.09)(12.12,4.07)(12.21037,4.030251)(12.32,3.89)(12.429629,3.7497492)(12.58,3.33)(12.58,3.07)
\psbezier[linewidth=0.04](12.58,3.07)(12.58,3.39)(12.695726,3.708794)(12.8,3.83)(12.904273,3.951206)(12.98,4.03)(13.22,4.07)
\usefont{T1}{ptm}{m}{n}
\rput(12.95125,4.85){\small pole}
\usefont{T1}{ptm}{m}{n}
\rput(12.23125,-0.05){\small pole}
\psbezier[linewidth=0.05,arrowsize=0.12cm 2.0,arrowlength=1.4,arrowinset=0.4]{->}(12.78,-0.43)(13.06,-0.87)(13.081014,-1.0102208)(13.1,-1.59)(13.118986,-2.169779)(13.06,-2.19)(12.82,-2.85)
\usefont{T1}{ptm}{m}{n}
\rput(14.534219,-1.72){bearing up to}
\usefont{T1}{ptm}{m}{n}
\rput(14.492031,-2.12){the poles}
\psdots[dotsize=0.16](12.82,-5.49)
\psdots[dotsize=0.16](12.8,-4.13)
\psbezier[linewidth=0.04](12.78,-4.11)(12.28,-4.11)(12.3,-3.25)(12.8,-3.25)(13.3,-3.25)(13.28,-4.11)(12.8,-4.11)
\psbezier[linewidth=0.04](12.78,-4.13)(12.44,-4.13)(12.38,-4.65)(12.38,-4.85)(12.38,-5.05)(12.44,-5.49)(12.8,-5.49)
\psbezier[linewidth=0.04](12.82,-5.49)(13.18,-5.49)(13.28,-5.09)(13.28,-4.83)(13.28,-4.57)(13.2,-4.13)(12.8,-4.13)
\psdots[dotsize=0.2,dotstyle=x](12.82,-4.87)
\psdots[dotsize=0.2,dotstyle=x](12.78,-3.71)
\psdots[dotsize=0.12](1.58,-3.09)
\psdots[dotsize=0.12](3.2,-3.09)
\psdots[dotsize=0.12](3.22,-4.49)
\psdots[dotsize=0.12](1.58,-4.51)
\psline[linewidth=0.04cm](1.58,-3.11)(1.58,-4.49)
\psline[linewidth=0.04cm](1.56,-4.49)(3.18,-4.47)
\psline[linewidth=0.04cm](3.22,-4.47)(3.2,-3.09)
\psline[linewidth=0.04cm](3.18,-3.09)(1.6,-3.07)
\psdots[dotsize=0.12](2.42,-2.11)
\psdots[dotsize=0.12](4.22,-3.81)
\psdots[dotsize=0.12](0.64,-3.77)
\psdots[dotsize=0.12](2.4,-5.39)
\psline[linewidth=0.04cm](1.58,-3.05)(2.38,-2.13)
\psline[linewidth=0.04cm](2.42,-2.09)(3.2,-3.09)
\psline[linewidth=0.04cm](3.2,-3.09)(4.2,-3.79)
\psline[linewidth=0.04cm](4.22,-3.81)(3.22,-4.49)
\psline[linewidth=0.04cm](3.22,-4.49)(2.42,-5.37)
\psline[linewidth=0.04cm](2.38,-5.37)(1.58,-4.51)
\psline[linewidth=0.04cm](1.56,-4.51)(0.64,-3.79)
\psline[linewidth=0.04cm](0.64,-3.75)(1.56,-3.07)
\psbezier[linewidth=0.05,arrowsize=0.12cm 2.0,arrowlength=1.4,arrowinset=0.4]{->}(2.56,-1.63)(2.7275982,-1.41)(2.9197788,-1.1893519)(2.9398894,-0.87)(2.96,-0.55064815)(2.82,-0.27)(2.56,0.01)
\psline[linewidth=0.04cm,linestyle=dotted,dotsep=0.16cm](12.56,4.63)(12.62,0.15)
\end{pspicture} 
}
\end{center}
\caption{Example of a construction of a quotient map with two axial faces.}
\label{exquograph}
\end{figure}
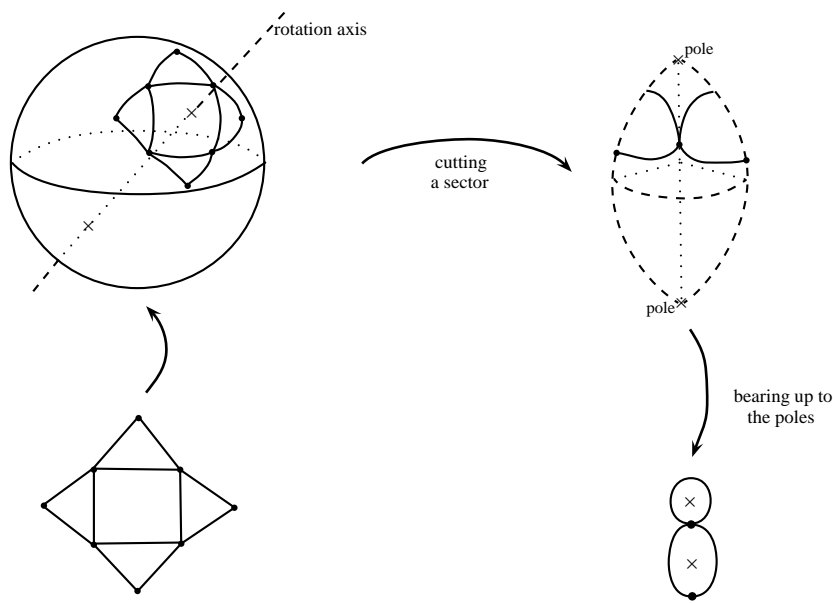

Conversely, we want to reconstruct the initial planar map $A$ knowing its quotient map $B$. This construction is a simple one from the moment when we know the axial cells of $B$ and the order $l$ of the rotation used to create the quotient map. In fact, by knowing these informations, we can represent on the sphere the quotient map $B$ such that the two axial cells of $B$ are placed on the north and south poles of the sphere. Then by cutting the sphere relative to a half-circle from the north pole to the south pole, we can open the sphere until obtaining a sphere sector with angle $2 \pi / l$. Finally, by glueing $l$ copies of this sphere sector, we construct a planar map which is exactly the map $A$. Figure \ref{figexcpmfromq} shows an example of this construction.

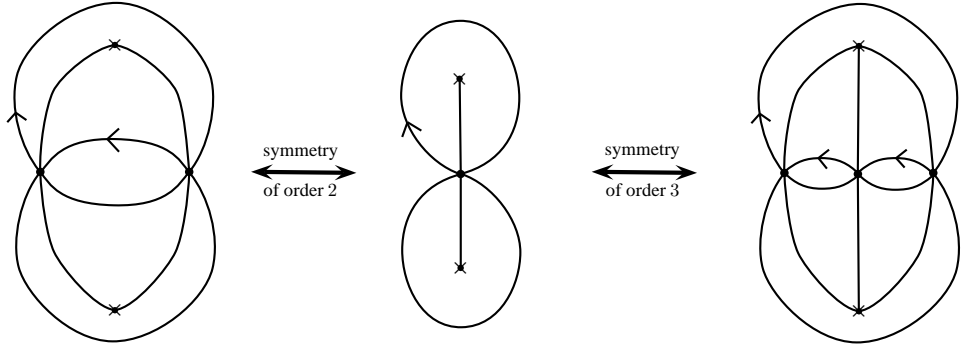
\begin{figure}[ht]
\centering
\scalebox{0.7} 
{
\begin{pspicture}(0,-3.24409)(18.001717,3.2440898)
\psdots[dotsize=0.24,dotstyle=x](8.5,1.77591)
\psdots[dotsize=0.12](8.5,1.77591)
\psdots[dotsize=0.24,dotstyle=x](8.52,-1.80409)
\psdots[dotsize=0.12](8.52,-1.80409)
\psdots[dotsize=0.16](8.52,-0.02408999)
\psline[linewidth=0.04cm](8.5,1.75591)(8.52,-0.00408999)
\psline[linewidth=0.04cm](8.52,-0.00408999)(8.52,-1.78409)
\psbezier[linewidth=0.04](8.5,-0.02408999)(7.92,0.17591001)(7.4,0.85591)(7.4,1.45591)(7.4,2.05591)(7.760635,2.8602848)(8.5,2.87591)(9.239366,2.8915353)(9.62,2.07591)(9.62,1.41591)(9.62,0.75591004)(9.28,0.17591001)(8.52,-0.02408999)
\psbezier[linewidth=0.04](8.54,-0.028464675)(9.12,-0.22846468)(9.64,-0.90846467)(9.64,-1.5084647)(9.64,-2.1084647)(9.279366,-2.9128394)(8.54,-2.9284647)(7.800635,-2.94409)(7.42,-2.1284647)(7.42,-1.4684647)(7.42,-0.80846465)(7.76,-0.22846468)(8.52,-0.028464675)
\psline[linewidth=0.08cm,arrowsize=0.133cm 2.0,arrowlength=1.4,arrowinset=0.4]{<->}(10.98,0.015910009)(12.96,-0.00408999)
\psdots[dotsize=0.24,dotstyle=x](16.0,2.39591)
\psdots[dotsize=0.12](16.0,2.39591)
\psdots[dotsize=0.24,dotstyle=x](16.0,-2.62409)
\psdots[dotsize=0.12](16.0,-2.62409)
\psdots[dotsize=0.16](15.98,-0.02408999)
\psdots[dotsize=0.16](17.4,-0.00408999)
\psdots[dotsize=0.16](14.6,-0.00408999)
\psline[linewidth=0.04cm](16.0,2.39591)(15.98,-0.04408999)
\psline[linewidth=0.04cm](15.98,0.015910009)(16.0,-2.64409)
\psbezier[linewidth=0.04](14.6,0.015910009)(14.84,0.21591)(15.0,0.27591002)(15.3,0.27591002)(15.6,0.27591002)(15.8,0.17591001)(16.0,-0.02408999)
\psbezier[linewidth=0.04](16.0,0.015910009)(16.24,0.21591)(16.4,0.27591002)(16.7,0.27591002)(17.0,0.27591002)(17.2,0.17591001)(17.4,-0.02408999)
\psbezier[linewidth=0.04](16.0,-0.04408999)(15.76,-0.24408999)(15.6,-0.30409)(15.3,-0.30409)(15.0,-0.30409)(14.8,-0.20408998)(14.6,-0.00408999)
\psbezier[linewidth=0.04](17.38,-0.06408999)(17.14,-0.26409)(16.98,-0.32409)(16.68,-0.32409)(16.38,-0.32409)(16.18,-0.22409)(15.98,-0.02408999)
\psbezier[linewidth=0.04](14.58,0.015910009)(14.24,0.43591002)(14.014061,1.2038589)(14.16,1.81591)(14.305939,2.427961)(15.222327,3.2040899)(16.04,3.19591)(16.857674,3.18773)(17.658283,2.3157127)(17.8,1.79591)(17.941717,1.2761073)(17.76,0.43591002)(17.4,-0.00408999)
\psbezier[linewidth=0.04](17.41578,-0.03591001)(17.75578,-0.45591)(17.981718,-1.223859)(17.83578,-1.83591)(17.68984,-2.447961)(16.77345,-3.22409)(15.955778,-3.21591)(15.138105,-3.20773)(14.337496,-2.3357127)(14.195779,-1.81591)(14.054061,-1.2961073)(14.235779,-0.45591)(14.595778,-0.015910009)
\psline[linewidth=0.04cm](7.42,0.67591)(7.52,0.95591)
\psline[linewidth=0.04cm](7.52,0.95591)(7.78,0.79591)
\psline[linewidth=0.04cm](15.24,0.27591002)(15.4,0.07591001)
\psline[linewidth=0.04cm](15.24,0.27591002)(15.44,0.43591002)
\psline[linewidth=0.04cm](16.86,0.09591001)(16.7,0.27591002)
\psline[linewidth=0.04cm](16.7,0.27591002)(16.9,0.43591002)
\psline[linewidth=0.04cm](13.98,0.93591)(14.14,1.11591)
\psline[linewidth=0.04cm](14.14,1.11591)(14.34,0.97591)
\psbezier[linewidth=0.04](16.0,2.39591)(15.62,2.35591)(14.98,1.81591)(14.84,1.51591)(14.7,1.21591)(14.58,0.47591)(14.6,-0.00408999)
\psbezier[linewidth=0.04](16.0,-2.62409)(16.38,-2.58409)(17.0,-1.82409)(17.14,-1.5240899)(17.28,-1.22409)(17.38,-0.52409)(17.4,-0.00408999)
\psbezier[linewidth=0.04](16.0,2.39591)(16.38,2.35591)(17.02,1.81591)(17.16,1.51591)(17.3,1.21591)(17.36,0.47591)(17.4,-0.02408999)
\psbezier[linewidth=0.04](16.0,-2.62409)(15.62,-2.58409)(15.0,-1.82409)(14.86,-1.5240899)(14.72,-1.22409)(14.62,-0.52409)(14.6,-0.00408999)
\psline[linewidth=0.08cm,arrowsize=0.133cm 2.0,arrowlength=1.4,arrowinset=0.4]{<->}(4.58,0.015910009)(6.56,-0.00408999)
\psdots[dotsize=0.24,dotstyle=x](2.02,2.41591)
\psdots[dotsize=0.12](2.02,2.41591)
\psdots[dotsize=0.24,dotstyle=x](2.02,-2.60409)
\psdots[dotsize=0.12](2.02,-2.60409)
\psdots[dotsize=0.16](3.42,0.015910009)
\psdots[dotsize=0.16](0.62,0.015910009)
\psbezier[linewidth=0.04](0.6,0.03591001)(0.26,0.45591)(0.034060992,1.223859)(0.18,1.83591)(0.325939,2.447961)(1.242327,3.22409)(2.06,3.21591)(2.8776731,3.20773)(3.6782825,2.3357127)(3.82,1.81591)(3.9617174,1.2961073)(3.78,0.45591)(3.42,0.015910009)
\psbezier[linewidth=0.04](3.4357784,-0.015910009)(3.7757785,-0.43591002)(4.0017176,-1.2038589)(3.8557785,-1.81591)(3.7098393,-2.427961)(2.7934515,-3.2040899)(1.9757785,-3.19591)(1.1581054,-3.18773)(0.35749587,-2.3157127)(0.21577844,-1.79591)(0.07406099,-1.2761073)(0.25577843,-0.43591002)(0.61577845,0.00408999)
\psline[linewidth=0.04cm](0.0,0.95591)(0.16,1.13591)
\psline[linewidth=0.04cm](0.16,1.13591)(0.36,0.99591)
\psbezier[linewidth=0.04](2.02,2.41591)(1.64,2.37591)(1.0,1.83591)(0.86,1.53591)(0.72,1.23591)(0.6,0.49591002)(0.62,0.015910009)
\psbezier[linewidth=0.04](2.02,-2.60409)(2.4,-2.56409)(3.02,-1.80409)(3.16,-1.50409)(3.3,-1.20409)(3.4,-0.50409)(3.42,0.015910009)
\psbezier[linewidth=0.04](2.02,2.41591)(2.4,2.37591)(3.04,1.83591)(3.18,1.53591)(3.32,1.23591)(3.38,0.49591002)(3.42,-0.00408999)
\psbezier[linewidth=0.04](2.02,-2.60409)(1.64,-2.56409)(1.02,-1.80409)(0.88,-1.50409)(0.74,-1.20409)(0.64,-0.50409)(0.62,0.015910009)
\psbezier[linewidth=0.04](0.62,-0.00408999)(0.82,0.25591)(1.0403296,0.66157997)(2.04,0.63591003)(3.0396705,0.61024004)(3.24,0.25591)(3.42,-0.00408999)
\psbezier[linewidth=0.04](3.42,0.021579992)(3.22,-0.23842001)(2.9996705,-0.64409)(2.0,-0.61842)(1.0003295,-0.59275)(0.8,-0.23842001)(0.62,0.021579992)
\psline[linewidth=0.04cm](2.1,0.39591002)(1.88,0.63591003)
\psline[linewidth=0.04cm](1.88,0.63591003)(2.1,0.83591)
\usefont{T1}{ptm}{m}{n}
\rput(5.510156,0.36591002){symmetry}
\usefont{T1}{ptm}{m}{n}
\rput(5.5064063,-0.33409){of order $2$}
\usefont{T1}{ptm}{m}{n}
\rput(11.930157,0.38591){symmetry}
\usefont{T1}{ptm}{m}{n}
\rput(11.938907,-0.33409){of order $3$}
\end{pspicture}  
}
\caption{Example of a quotient map with two axial vertices and its associated $2$-rooted map (left) and $3$-rooted map (right).}
\label{figexcpmfromq}
\end{figure}

The only exception of this construction is when the quotient map has an axial half-edge. In this case, the order $l$ of the symmetry is fixed at $2$. See Figure \ref{figex2rmquo} for an example. It's for this reason that for our enumeration we will distinguished two case : the quotient map having an axial half-edge, or the quotient map without an axial half-edge. The first case is often more difficult to enumerate.

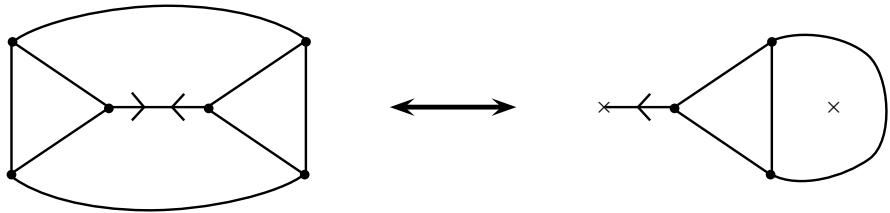
\begin{figure}[ht]
\centering
\scalebox{0.8} 
{
\begin{pspicture}(0,-1.7250745)(14.62,1.7250744)
\psdots[dotsize=0.16](0.1,1.0949255)
\psdots[dotsize=0.16](0.08,-1.1050745)
\psdots[dotsize=0.16](1.68,-0.0050745266)
\psline[linewidth=0.04cm](0.08,1.0949255)(0.08,-1.0850745)
\psline[linewidth=0.04cm](0.1,1.0949255)(1.68,0.014925473)
\psline[linewidth=0.04cm](0.08,-1.0850745)(1.7,0.014925473)
\psdots[dotsize=0.16,dotangle=-180.0](4.9,-1.1050745)
\psdots[dotsize=0.16,dotangle=-180.0](4.92,1.0949255)
\psdots[dotsize=0.16,dotangle=-180.0](3.32,-0.0050745266)
\psline[linewidth=0.04cm](4.92,-1.1050745)(4.92,1.0749254)
\psline[linewidth=0.04cm](4.9,-1.1050745)(3.32,-0.025074527)
\psline[linewidth=0.04cm](4.92,1.1149255)(3.3,0.014925473)
\psline[linewidth=0.04cm](1.68,0.014925473)(3.3,0.014925473)
\psbezier[linewidth=0.04](0.1,1.0949255)(0.5,1.4149255)(1.7000515,1.7050745)(2.7,1.6949254)(3.6999485,1.6847764)(4.5,1.4549254)(4.92,1.1349255)
\psbezier[linewidth=0.04](4.9,-1.0949255)(4.5,-1.4149255)(3.2999485,-1.7050745)(2.3,-1.6949254)(1.3000515,-1.6847764)(0.5,-1.4549254)(0.08,-1.1349255)
\psline[linewidth=0.04cm](2.06,-0.20507453)(2.26,0.014925473)
\psline[linewidth=0.04cm](2.26,0.014925473)(2.06,0.25492546)
\psline[linewidth=0.04cm](2.92,-0.20507453)(2.72,0.014925473)
\psline[linewidth=0.04cm](2.72,0.014925473)(2.92,0.23492548)
\psline[linewidth=0.08cm,arrowsize=0.133cm 2.0,arrowlength=1.4,arrowinset=0.4]{<->}(6.3,0.014925473)(8.38,0.014925473)
\psdots[dotsize=0.16,dotangle=-180.0](12.56,-1.1050745)
\psdots[dotsize=0.16,dotangle=-180.0](12.58,1.0949255)
\psdots[dotsize=0.16,dotangle=-180.0](10.98,-0.0050745266)
\psline[linewidth=0.04cm](12.58,-1.1050745)(12.58,1.0749254)
\psline[linewidth=0.04cm](12.56,-1.1050745)(10.98,-0.025074527)
\psline[linewidth=0.04cm](12.6,1.1149255)(10.98,0.014925473)
\psline[linewidth=0.04cm](9.84,0.014925473)(10.96,0.014925473)
\psline[linewidth=0.04cm](10.58,-0.20507453)(10.38,0.014925473)
\psline[linewidth=0.04cm](10.38,0.014925473)(10.58,0.23492548)
\psdots[dotsize=0.2,dotstyle=x](9.82,0.014925473)
\psdots[dotsize=0.2,dotstyle=x](13.6,0.014925473)
\psbezier[linewidth=0.04](12.58,1.1149255)(13.06,1.3149254)(13.8,1.1949254)(14.16,0.8749255)(14.52,0.5549255)(14.6,-0.5450745)(14.18,-0.84507453)(13.76,-1.1450745)(13.04,-1.3450745)(12.58,-1.1050745)
\end{pspicture} 
}
\caption{Example of a $2$-rooted map and its quotient map.}
\label{figex2rmquo}
\end{figure}

\vspace{2mm}

\textbf{Rk:} The notion of $l$-rooted map can be defined for a bracketing representation. In fact, we say that a bracketing representation of $n$ elements is $l$-rooted if it is constructed by copying the same string of $n/l$ elements $l$ times. The quotient associated to this bracketing is simply the string of $n/l$ elements.

\vspace{2mm}

Now, thanks to this concept of quotient maps, we can rewrite Theorem \ref{thenmap} as follows:

\begin{thenmapnr}
\label{thenmapnr}
(Liskovets)
\begin{displaymath}
M^+(n) = \frac{1}{2n} \left[ M'(n) + \sum_{l \geq 2 \atop l | 2n}{\varphi(l).M'_{l}(n)} + M'_{e}(n) \right], \quad n \geq 2
\end{displaymath}
where $M'(n)$ is the number of rooted maps of $\mathcal{M}(n)$, $M'_{e}(n)$ is the number of rooted quotient maps from $2$-rooted maps of $\mathcal{M}(n)$ with an axial half-edge, and $M'_{l}(n)$ is the number of rooted quotient maps from $l$-rooted maps without axial half-edge.
\end{thenmapnr}

\subsection{A first application: the generic case.}

In this subsection, we will use the example of structurally stable polynomial vector fields described by A.Douady, F.Estrada and P.Sentenac in \cite{DES} to explain Theorem \ref{thenmapnr} in a simple case. In fact, in the case the bracketing problem is equivalent to the enumeration of rooted tree, also the number $\sigma_n$ of topologically equivalent phase portrait of generic complex polynomial differential equation (of degree $d=n+1$) is equal to the number $T^{+}(n)$ of unrooted tree (with $n$ edges). So, in order to prove Theorem \ref{mainth1}, we have to prove first that

\begin{displaymath}
\sigma_n = T^{+}(n) = \frac{1}{2n} \left[ \frac{1}{n+1} \left( \! \begin{array}{c} 2n \\ n \end{array} \! \right) + \sum_{l \geq 2 \atop l|n}{\varphi(l) \left( \! \begin{array}{c} 2n/l \\ n/l \end{array} \! \right)} + \left\{ \begin{array}{cc}
\left( \! \begin{array}{c} n \\ \frac{n-1}{2} \end{array} \! \right) & \text{if } n \text{ is odd} \\
& \\
0 & \text{if } n \text{ is even}
\end{array} \right. \right] .
\end{displaymath}

This result can also be found in \cite{Liskexen}. Notice also that we obtain the same result in the case of complex polynomial vector fields where all equilibrium points are centers. the interested readers can refer to the article \cite{AGP} for more details about this particular case.

\vspace{3mm}

Now let us solve the enumeration problem of unrooted trees using Theorem \ref{thenmapnr}. This theorem applied to our situation gives that:

\begin{displaymath}
T^+(n) = \frac{1}{2n} \left[ T'(n) + \sum_{l \geq 2 \atop l | 2n}{\varphi(l).T'_{l}(n)} + T'_{e}(n) \right], \quad n \geq 2,
\end{displaymath}
where $T'(n)$ is the number of rooted trees, $T'_{e}(n)$ is the number of rooted quotient trees with an axial half-edge, and $T'_{l}(n)$ is the number of rooted quotient trees without axial half-edge obtained thanks to a symmetry of order $l$.

The enumeration of rooted trees are obtained a long time ago, and the solution of this problem is given by the numbers of Catalan (see \cite{catalan} for more details). So, 

\begin{displaymath}
T'(n) = \frac{1}{n+1} \left( \! \begin{array}{c} 2n \\ n \end{array} \! \right).
\end{displaymath}

It remains to determinate the values $T'_{l}(n)$ and $T'_{e}(n)$. For that, we need to understand the quotient maps obtain from trees, and more particularly the set of pair of axial cells involved in this quotient action. In a general study, there are six different pairs:

\begin{itemize}
\item the two axial cells are vertices (see Figure \ref{figexcpmfromq}).
\item the two axial cells are edges.
\item the two axial cells are faces (see Figure \ref{exquograph}).
\item the two axial cells are a vertex and an edge.
\item the two axial cells are a vertex and a face (see Figure \ref{figex3rplmap}).
\item the two axial cells are an edge and a face (see Figure \ref{figex2rmquo}).
\end{itemize}

Our situation is easier. In fact, suppose that the two axial cells are vertices, then by connectedness of a tree, there is a path between these two vertices. But by symmetry of the map, if such a path exists, it must exist in as many copies as the order of the symmetry (therefore at least twice). This implies that the map has at least two faces which is in contradiction with the definition of a tree. In conclusion, the two axial cells can not be both vertices.

By a similar argument, we prove that the four first situations described earlier are impossible. So, the only cases we have to consider here are the pairs vertex-face and edge-face. Now study each case separately.

\begin{enumerate}
\item[1.] The two axial cells are a vertex and a face.

In this case, it's simple to see that the quotient map of a $l$-rooted tree with $n$ edges is a rooted tree with $n/l$ edges and a distinguished face and vertex (the axial ones). In conclusion,

\begin{displaymath}
T'_{l}(n) = \frac{1}{n/l +1} \left( \! \begin{array}{c} 2n/l \\ n/l \end{array} \! \right) \times (n/l +1) = \left( \! \begin{array}{c} 2n/l \\ n/l \end{array} \! \right).
\end{displaymath}

\item[2.] The two axial cells are an edge and a face.

In this case, we must consider a $2$-rooted tree with $n$ edges. The quotient action will divide an edge and share the other $n-1$ edges into two equal sets. So, if $n$ is even, this quotient is impossible, and $T'_{e}(n) = 0$. Now, we suppose that $n$ is odd, then two situations can appear:
\begin{itemize}
\item the axial edge is not the root edge. In this case, the quotient map is simply a rooted tree with $(n-1)/2$ edges with an added half-edge (the axial one).
\item the axial edge is the root edge. In this case, the quotient map is a tree with $(n-1)/2$ edges and a rooted half-edge. By rooted the first edge to the left of the rooted half-edge, preserving the root vertex, we construct a bijection between these quotient maps and the set of rooted trees with $(n-1)/2$ edges (see Figure \ref{figexbtree} for an example).
\end{itemize}

So, if $n$ is odd,

\begin{eqnarray*}
T'_{e}(n) & = & \frac{1}{(n-1)/2 + 1} \left( \! \begin{array}{c} n-1 \\ (n-1)/2 \end{array} \! \right) \times (n-1) + \frac{1}{(n-1)/2 + 1} \left( \! \begin{array}{c} n-1 \\ (n-1)/2 \end{array} \! \right) \\
& = & \left( \! \begin{array}{c} n \\ (n-1)/2 \end{array} \! \right).
\end{eqnarray*}

\end{enumerate}

This completes the enumeration of unrooted tree with $n$ edges.

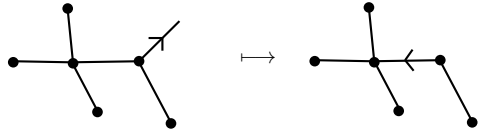
\begin{figure}[ht]
\centering
\scalebox{0.7} 
{
\begin{pspicture}(0,-1.18)(8.76,1.18)
\psdots[dotsize=0.2](1.08,1.08)
\psdots[dotsize=0.2](0.06,0.06)
\psdots[dotsize=0.2](1.18,0.04)
\psdots[dotsize=0.2](2.42,0.08)
\psdots[dotsize=0.2](1.64,-0.88)
\psdots[dotsize=0.2](3.02,-1.1)
\psline[linewidth=0.04cm](0.08,0.08)(1.14,0.06)
\psline[linewidth=0.04cm](1.14,0.06)(2.42,0.08)
\psline[linewidth=0.04cm](1.18,0.06)(1.08,1.06)
\psline[linewidth=0.04cm](1.18,0.02)(1.64,-0.86)
\psline[linewidth=0.04cm](2.4,0.1)(3.02,-1.08)
\psline[linewidth=0.04cm](2.42,0.08)(3.18,0.84)
\psline[linewidth=0.04cm](2.86,0.52)(2.86,0.28)
\psline[linewidth=0.04cm](2.86,0.52)(2.6,0.52)
\psdots[dotsize=0.2](6.74,1.1)
\psdots[dotsize=0.2](5.72,0.08)
\psdots[dotsize=0.2](6.84,0.06)
\psdots[dotsize=0.2](8.08,0.1)
\psdots[dotsize=0.2](7.3,-0.86)
\psdots[dotsize=0.2](8.68,-1.08)
\psline[linewidth=0.04cm](5.74,0.1)(6.8,0.08)
\psline[linewidth=0.04cm](6.8,0.08)(8.08,0.1)
\psline[linewidth=0.04cm](6.84,0.08)(6.74,1.08)
\psline[linewidth=0.04cm](6.86,0.04)(7.32,-0.84)
\psline[linewidth=0.04cm](8.06,0.12)(8.68,-1.06)
\psline[linewidth=0.04cm](7.42,0.1)(7.56,0.3)
\psline[linewidth=0.04cm](7.42,0.1)(7.58,-0.1)
\usefont{T1}{ptm}{m}{n}
\rput(4.657344,0.115){\large $\longmapsto$}
\end{pspicture} 
}
\caption{Example of equivalence between a tree with a rooted half-edge and a rooted tree.}
\label{figexbtree}
\end{figure}

Now, to complete the proof of Theorem \ref{thenmapnr}, we just need to prove the following result:

\begin{cortree}
\begin{displaymath}
\lim_{n \rightarrow +\infty}{(\sigma_{n})^{1/n}} = 4.
\end{displaymath}
\end{cortree}

\begin{proof}
This corollary is a direct consequence of the last result. In fact, from the formula of $\sigma_n$, it's easy to prove that
\begin{displaymath}
\frac{1}{2n} \frac{1}{n+1}  \left( \! \begin{array}{c} 2n \\ n \end{array} \! \right) \leq \sigma_{n} \leq \frac{1}{n+1}  \left( \! \begin{array}{c} 2n \\ n \end{array} \! \right),
\end{displaymath}
and so 
\begin{displaymath}
\lim_{n \rightarrow +\infty}{(\sigma_{n})^{1/n}} = \lim_{n \rightarrow +\infty}{\left( \frac{1}{n+1}  \left( \! \begin{array}{c} 2n \\ n \end{array} \! \right) \right)^{1/n}}.
\end{displaymath}
Then, by using the Stirling formula, we can show that
\begin{displaymath}
\frac{1}{n+1}  \left( \! \begin{array}{c} 2n \\ n \end{array} \! \right) \sim \frac{1}{\sqrt{\pi}} \frac{4^n}{n^{3/2}},
\end{displaymath}
in the sense that the quotient of the two terms tends to $1$ as $n$ tends to $\infty$. From this equivalence, we deduce the required result.
\end{proof}

\subsection{Enumeration of rooted generalized trees.}

As in the previous example, we will use Theorem \ref{thenmapnr} to give an enumeration in the general case. But to use Theorem \ref{thenmapnr}, we need to enumerate the set of rooted generalized trees (or equivalently the set of valid bracketing). We need to use the Lagrange-B\"{u}rmann inversion theorem (\ref{thlbinv}), to demonstrate the following lemma:

\begin{lemvalbr}
\label{lemvalbr}
Denote by $b_n$ the number of valid bracketings in a string of $n$ elements. Then,
\begin{eqnarray*}
b_{2n} \quad = \quad p_n & = & \sum_{k \geq 0}{\frac{(-2n)_k (-n)_k}{(2)_k} \frac{2^k}{k!}}, \quad n \geq 0 \\
b_{2n-1} \quad = \quad q_n & = & \frac{1}{n} \sum_{k=0}^{n-1}{\left( \begin{array}{c} 2n \\ k \end{array} \right) \left( \begin{array}{c} n \\ n-1-k \end{array} \right) 2^k}, \quad n\geq 1,
\end{eqnarray*}
with $(x)_n = x.(x+1) \ldots (x+n-1)$, and $(x)_0 = 1$.
\end{lemvalbr}

\begin{proof}
Consider a valid bracketing in a string of $n$ elements, and read it, from the left to the right. Two different situations appear in the beginning:

\begin{enumerate}
\item[Case 1.] The first element in the string is a dot, and the bracketing representation composed by the others $n-1$ elements is a valid one.
\item[Case 2.] The first element in the string is a left square (resp. round) parenthesis, and so it is paired with a $j$-th element of the string (with $j$ even). Then, the bracketings representation construct thanks to the elements from $2$ to $j-1$ and the elements from $j+1$ to $n$ are valid bracketings too.
\end{enumerate}

So we deduce for these observations that $b_n$ satisfies the recursion equation
\begin{displaymath}
b_n = b_{n-1} + 2 \sum_{2 \alpha + \beta +2 = n}{b_{2\alpha} b_{\beta}}, \quad n\geq 1 \text{ and } \alpha , \beta \geq 0,
\end{displaymath}
where the first term of the sum is deducted from Case 1 and the second from Case 2. By convention, we set $b_0 = p_0 = 1$ and $q_0 = 0$. Now we consider the generating functions
\begin{eqnarray*}
p(z) & = & \sum_{n \geq 0}{p_n z^n}, \\
q(z) & = & \sum_{n \geq 0}{q_n z^n}.
\end{eqnarray*}
Then by the recursive formula, we deduce for $n \geq 1$
\begin{eqnarray*}
p_n \quad := \quad b_{2n} & = & b_{2n-1} + 2 \sum_{2 \alpha + \beta = 2n-2}{b_{2\alpha} b_{\beta}} \\
& = & b_{2n-1} + 2 \sum_{k=0}^{n-1}{b_{2k} b_{2(n-1-k)}} \\
& = & q_n + 2 \sum_{k=0}^{n-1}{p_k p_{n-1-k}}
\end{eqnarray*}

and

\begin{eqnarray*}
q_n := b_{2n-1} & = & b_{2n-2} + 2 \sum_{2 \alpha + \beta = 2n-3}{b_{2\alpha} b_{\beta}} \\
& = & b_{2(n-1)} + 2 \sum_{k=0}^{n-1}{b_{2k} b_{2(n-1-k)-1}} \\
& = & p_{n-1} + 2 \sum_{k=0}^{n-1}{p_k q_{n-1-k}}.
\end{eqnarray*}

So from these two equations, one can deduce that
\begin{eqnarray*}
p & = & 1 + q + 2zp^2, \\
q & = & zp + 2zpq.
\end{eqnarray*}

From the second equation, we obtain
\begin{displaymath}
p = \frac{q}{z(1+2q)} ,
\end{displaymath}

and so substituting this in the first equation, and after some simplifications, we get
\begin{displaymath}
q(1+2q) = z(1+q)(1+2q)^2 + 2q^2 ,
\end{displaymath}

or
\begin{displaymath}
q = z(1+q)(1+2q)^2.
\end{displaymath}

Then we use the following theorem:

\begin{thlbinv}
\label{thlbinv}
(Lagrange-B\"{u}rmann inversion theorem) Let $\phi(u)$ be a formal power series with $\phi_0 \neq 0$ and let $Y(z)$ be the unique formal power series solution of the equation $Y = z.\phi(Y)$. Then the coefficient of $Y(z)$ of order $n$, noted $[z^n]Y(z)$, is given by
\begin{displaymath}
[z^n]Y(z) = \frac{1}{n} [u^{n-1}] \phi(u)^n.
\end{displaymath}
\end{thlbinv}

By using this theorem in our situation, setting $Y(z) = q(z)$ and $\phi(Y) = (1+Y)(1+2Y)^2$, we get

\begin{eqnarray*}
q_n = [z^n]q(z) & = & \frac{1}{n} [u^{n-1}] \left( (1+u)(1+2u)^2 \right)^n \\
& = & \frac{1}{n} \sum_{k=0}^{n-1}{\left( \begin{array}{c} 2n \\ k \end{array} \right) 2^k \left( \begin{array}{c} n \\ n-1-k \end{array} \right)}.
\end{eqnarray*}

Thanks to the Lagrange-B\"{u}rmann inversion theorem and equations verified by the generating functions $p$ and $q$, we find also that

\begin{displaymath}
p_n = \sum_{k \geq 0}{\frac{(-2n)_k (-n)_k}{(2)_k} \frac{2^k}{k!}}.
\end{displaymath}

The formula of $p_n$ has been proved in the article of K.Dias \cite{Dias2}. 
\end{proof}

\subsection{Enumeration of unrooted generalized trees.}

In the following, in order to simplify the notation, we set
\begin{displaymath}
p_{n,k} = \frac{(-2n)_k (-n)_k}{(2)_k} \frac{2^k}{k!}.
\end{displaymath}
Notice that $p_{n,k}$ is equal to the number of rooted generalized tree with $k$ edges, $k+1$ vertices and $2(n-k)$ half-edges.

Now, we can prove the main result of this section:

\begin{thenpoly}
\label{thenpoly}
Denote by $p^+_n$ the number of topologically different complex polynomial vector fields of degree $d=n+1$, then 
\begin{displaymath}
p^+_n = \frac{1}{2n} \left[ b_{2n} + \sum_{l \geq 0 \atop l|2n} \varphi(l). \left\{ \begin{array}{cc}
\sum\limits_{k \geq 0}{p_{n/l,k}.(k+1)} & \text{ if } 2n/l \text{ even} \\
c_{2n/l} & \text{ if } 2n/l \text{ odd}
\end{array} \right. + \left\{ \begin{array}{cc}
0 & \text{ if } n \text{ even} \\
2n. b_{n-1} & \text{ if } n \text{ odd}
\end{array} \right. \right], \quad n \geq 1,
\end{displaymath}
with $c_m$ satisfying the recursive formula
\begin{eqnarray*}
c_m & = & c_{m-1} + 2\sum_{2\alpha + \beta + 2 = m}{b_{2\alpha} c_{\beta}} + 2\sum_{2\alpha + \beta + 3 = m}{c_{2\alpha +1} b_{\beta}}, \quad m\geq 1 \\
c_0 & = & 1.
\end{eqnarray*}
\end{thenpoly}

We will give in the next subsection an explicit formula for the coefficient $c_m$. The first values obtained by this formula are given in Figure \ref{figtabval}, and modelisations of the polynomial vector fields of degree $2$ and $3$ are given in Figures \ref{figmodpoldeg2} and \ref{figmodpoldeg3}.

\begin{figure}[ht]
\centering
\begin{minipage}[b]{.46\linewidth}
\centering
\begin{tabular}{|c||c|}
\hline
$deg=n+1$ & $p^+_n$ \\ \hline
2 & 3 \\
3 & 6 \\
4 & 26 \\
5 & 123 \\
6 & 801 \\
7 & 5686 \\
8 & 43846 \\
9 & 353987 \\
10 & 2968801 \\
11 & 25605445 \\
\hline
\end{tabular}
\caption{The first $10$ values of $p_n^+$.}
\label{figtabval}
\end{minipage} \hfill
\begin{minipage}[b]{.46\linewidth}
\centering
\scalebox{0.8}{
\begin{pspicture}(0,-3.28)(7.6925,3.28)
\pscircle[linewidth=0.04,dimen=outer](1.63625,1.68){1.6}
\psdots[dotsize=0.12,dotstyle=square*](3.21625,1.7)
\psdots[dotsize=0.12,dotstyle=square*](0.05625,1.7)
\pscircle[linewidth=0.04,dimen=outer](6.03625,1.68){1.6}
\psdots[dotsize=0.12,dotstyle=square*](7.61625,1.7)
\psdots[dotsize=0.12,dotstyle=square*](4.45625,1.7)
\pscircle[linewidth=0.04,dimen=outer](3.83625,-1.68){1.6}
\psdots[dotsize=0.12,dotstyle=square*](5.41625,-1.66)
\psdots[dotsize=0.12,dotstyle=square*](2.25625,-1.66)
\psdots[dotsize=0.16](1.61625,1.68)
\psdots[dotsize=0.16](6.03625,0.86)
\psdots[dotsize=0.16](6.01625,2.5)
\psdots[dotsize=0.16](3.21625,-1.68)
\psdots[dotsize=0.16](4.49625,-1.68)
\psline[linewidth=0.04cm,linestyle=dashed,dash=0.16cm 0.16cm,arrowsize=0.12cm 2.0,arrowlength=1.4,arrowinset=0.5]{>-}(0.07625,1.72)(1.55625,1.7)
\psline[linewidth=0.04cm,linestyle=dashed,dash=0.16cm 0.16cm,arrowsize=0.12cm 2.0,arrowlength=1.4,arrowinset=0.5]{->}(1.67625,1.7)(3.19625,1.7)
\psline[linewidth=0.04cm,linestyle=dashed,dash=0.16cm 0.16cm,arrowsize=0.12cm 2.0,arrowlength=1.4,arrowinset=0.5]{>->}(4.45625,1.72)(7.59625,1.72)
\psline[linewidth=0.04cm,linestyle=dashed,dash=0.16cm 0.16cm,arrowsize=0.12cm 2.0,arrowlength=1.4,arrowinset=0.5]{>-}(2.23625,-1.66)(3.17625,-1.66)
\psline[linewidth=0.04cm,linestyle=dashed,dash=0.16cm 0.16cm,arrowsize=0.12cm 2.0,arrowlength=1.4,arrowinset=0.5]{->}(4.51625,-1.66)(5.37625,-1.66)
\end{pspicture} 
} 
\caption{Modelisation of the $3$ polynomial vector fields of degree $2$.}
\label{figmodpoldeg2}
\end{minipage}
\end{figure}

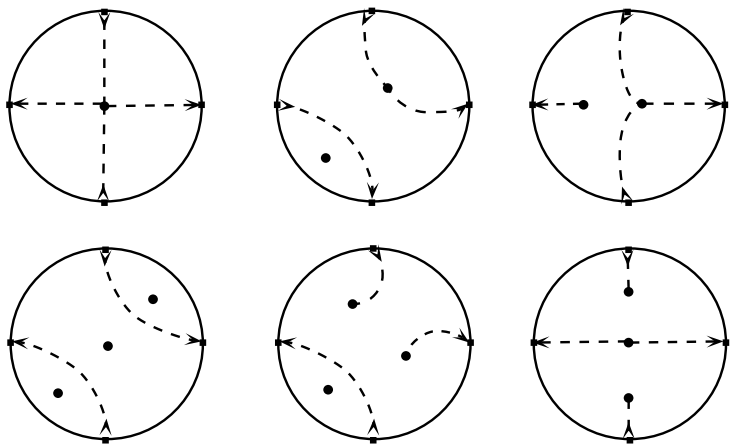
\begin{figure}[ht]
\centering
\scalebox{0.8} 
{
\begin{pspicture}(0,-3.63625)(11.9125,3.63625)
\pscircle[linewidth=0.04,dimen=outer](1.63625,1.98){1.6}
\psdots[dotsize=0.12,dotstyle=square*](3.21625,2.0)
\psdots[dotsize=0.12,dotstyle=square*](0.05625,2.0)
\pscircle[linewidth=0.04,dimen=outer](6.03625,1.98){1.6}
\psdots[dotsize=0.12,dotstyle=square*](7.61625,2.0)
\psdots[dotsize=0.12,dotstyle=square*](4.45625,2.0)
\pscircle[linewidth=0.04,dimen=outer](10.23625,1.98){1.6}
\psdots[dotsize=0.12,dotstyle=square*](11.81625,2.0)
\psdots[dotsize=0.12,dotstyle=square*](8.65625,2.0)
\psdots[dotsize=0.12,dotstyle=square*](1.61625,3.54)
\psdots[dotsize=0.12,dotstyle=square*](1.61625,0.38)
\psdots[dotsize=0.12,dotstyle=square*](6.01625,0.38)
\psdots[dotsize=0.12,dotstyle=square*](6.01625,3.56)
\psdots[dotsize=0.12,dotstyle=square*](10.21625,3.54)
\psdots[dotsize=0.12,dotstyle=square*](10.23625,0.38)
\pscircle[linewidth=0.04,dimen=outer](1.65625,-1.96){1.6}
\psdots[dotsize=0.12,dotstyle=square*](3.23625,-1.94)
\psdots[dotsize=0.12,dotstyle=square*](0.07625,-1.94)
\pscircle[linewidth=0.04,dimen=outer](6.05625,-1.96){1.6}
\psdots[dotsize=0.12,dotstyle=square*](7.63625,-1.94)
\psdots[dotsize=0.12,dotstyle=square*](4.47625,-1.94)
\pscircle[linewidth=0.04,dimen=outer](10.25625,-1.96){1.6}
\psdots[dotsize=0.12,dotstyle=square*](11.83625,-1.94)
\psdots[dotsize=0.12,dotstyle=square*](8.67625,-1.94)
\psdots[dotsize=0.12,dotstyle=square*](1.63625,-0.4)
\psdots[dotsize=0.12,dotstyle=square*](1.63625,-3.56)
\psdots[dotsize=0.12,dotstyle=square*](6.03625,-3.56)
\psdots[dotsize=0.12,dotstyle=square*](6.03625,-0.38)
\psdots[dotsize=0.12,dotstyle=square*](10.23625,-0.4)
\psdots[dotsize=0.12,dotstyle=square*](10.25625,-3.56)
\psdots[dotsize=0.16](1.61625,1.98)
\psdots[dotsize=0.16](6.27625,2.28)
\psdots[dotsize=0.16](5.25625,1.12)
\psdots[dotsize=0.16](9.49625,2.0)
\psdots[dotsize=0.16](10.45625,2.02)
\psdots[dotsize=0.16](2.41625,-1.22)
\psdots[dotsize=0.16](0.85625,-2.78)
\psdots[dotsize=0.16](1.67625,-2.0)
\psdots[dotsize=0.16](5.29625,-2.72)
\psdots[dotsize=0.16](6.57625,-2.16)
\psdots[dotsize=0.16](5.69625,-1.3)
\psdots[dotsize=0.16](10.23625,-1.94)
\psdots[dotsize=0.16](10.23625,-1.1)
\psdots[dotsize=0.16](10.23625,-2.86)
\psline[linewidth=0.04cm,linestyle=dashed,dash=0.16cm 0.16cm,arrowsize=0.12cm 2.0,arrowlength=1.4,arrowinset=0.5]{<-}(0.07625,2.02)(1.57625,2.02)
\psline[linewidth=0.04cm,linestyle=dashed,dash=0.16cm 0.16cm,arrowsize=0.12cm 2.0,arrowlength=1.4,arrowinset=0.5]{->}(1.65625,1.98)(3.21625,2.0)
\psline[linewidth=0.04cm,linestyle=dashed,dash=0.16cm 0.16cm,arrowsize=0.12cm 2.0,arrowlength=1.4,arrowinset=0.5]{-<}(1.61625,2.0)(1.59625,0.4)
\psline[linewidth=0.04cm,linestyle=dashed,dash=0.16cm 0.16cm,arrowsize=0.12cm 2.0,arrowlength=1.4,arrowinset=0.5]{-<}(1.61625,1.98)(1.61625,3.54)
\psline[linewidth=0.04cm,linestyle=dashed,dash=0.16cm 0.16cm,arrowsize=0.12cm 2.0,arrowlength=1.4,arrowinset=0.5]{<-}(8.63625,2.0)(9.45625,2.02)
\psline[linewidth=0.04cm,linestyle=dashed,dash=0.16cm 0.16cm,arrowsize=0.12cm 2.0,arrowlength=1.4,arrowinset=0.5]{-<}(10.23625,-1.08)(10.21625,-0.4)
\psline[linewidth=0.04cm,linestyle=dashed,dash=0.16cm 0.16cm,arrowsize=0.12cm 2.0,arrowlength=1.4,arrowinset=0.5]{>-}(10.23625,-3.54)(10.23625,-2.86)
\psline[linewidth=0.04cm,linestyle=dashed,dash=0.16cm 0.16cm,arrowsize=0.12cm 2.0,arrowlength=1.4,arrowinset=0.5]{->}(10.25625,-1.94)(11.79625,-1.92)
\psline[linewidth=0.04cm,linestyle=dashed,dash=0.16cm 0.16cm,arrowsize=0.12cm 2.0,arrowlength=1.4,arrowinset=0.5]{->}(10.21625,-1.92)(8.65625,-1.94)
\psline[linewidth=0.04cm,linestyle=dashed,dash=0.16cm 0.16cm,arrowsize=0.12cm 2.0,arrowlength=1.4,arrowinset=0.5]{->}(10.47625,2.02)(11.79625,2.02)
\psbezier[linewidth=0.04,linestyle=dashed,dash=0.16cm 0.16cm,arrowsize=0.12cm 2.0,arrowlength=1.4,arrowinset=0.5]{>->}(4.47625,2.0)(4.81625,1.96)(5.283834,1.7417572)(5.53625,1.54)(5.7886662,1.3382428)(6.03625,0.82)(6.01625,0.44)
\psbezier[linewidth=0.04,linestyle=dashed,dash=0.16cm 0.16cm,arrowsize=0.12cm 2.0,arrowlength=1.4,arrowinset=0.5]{>-}(5.97625,3.56)(5.87625,3.36)(5.8960204,3.135448)(5.91625,2.9)(5.93648,2.664552)(6.13625,2.44)(6.25625,2.3)
\psbezier[linewidth=0.04,linestyle=dashed,dash=0.16cm 0.16cm,arrowsize=0.12cm 2.0,arrowlength=1.4,arrowinset=0.5]{->}(6.27625,2.28)(6.41625,2.02)(6.69625,1.88)(6.89625,1.88)(7.09625,1.88)(7.41625,1.88)(7.59625,2.02)
\psbezier[linewidth=0.04,linestyle=dashed,dash=0.16cm 0.16cm,arrowsize=0.12cm 2.0,arrowlength=1.4,arrowinset=0.5]{-<}(10.43625,2.02)(10.21625,2.02)(10.11625,2.54)(10.09625,2.78)(10.07625,3.02)(10.09625,3.32)(10.23625,3.56)
\psbezier[linewidth=0.04,linestyle=dashed,dash=0.16cm 0.16cm,arrowsize=0.12cm 2.0,arrowlength=1.4,arrowinset=0.5]{-<}(10.45625,2.04)(10.21625,1.92)(10.116455,1.6185666)(10.09625,1.36)(10.076045,1.1014334)(10.09625,0.74)(10.212369,0.38)
\psbezier[linewidth=0.04,linestyle=dashed,dash=0.16cm 0.16cm,arrowsize=0.12cm 2.0,arrowlength=1.4,arrowinset=0.5]{<-<}(0.09625,-1.92)(0.43625,-1.96)(0.9038338,-2.1782427)(1.15625,-2.38)(1.4086661,-2.5817573)(1.65625,-3.1)(1.63625,-3.48)
\psbezier[linewidth=0.04,linestyle=dashed,dash=0.16cm 0.16cm,arrowsize=0.12cm 2.0,arrowlength=1.4,arrowinset=0.5]{<-<}(3.21625,-1.92)(2.91625,-1.88)(2.3686662,-1.7017572)(2.11625,-1.5)(1.8638338,-1.2982428)(1.61625,-0.78)(1.63625,-0.4)
\psbezier[linewidth=0.04,linestyle=dashed,dash=0.16cm 0.16cm,arrowsize=0.12cm 2.0,arrowlength=1.4,arrowinset=0.5]{<-<}(4.49625,-1.92)(4.83625,-1.96)(5.303834,-2.1782427)(5.55625,-2.38)(5.808666,-2.5817573)(6.05625,-3.1)(6.03625,-3.48)
\psbezier[linewidth=0.04,linestyle=dashed,dash=0.16cm 0.16cm,arrowsize=0.12cm 2.0,arrowlength=1.4,arrowinset=0.5]{-<}(5.71625,-1.3)(5.93625,-1.3)(6.08198,-1.1673199)(6.15625,-1.0)(6.23052,-0.8326801)(6.17625,-0.6)(6.03625,-0.36)
\psbezier[linewidth=0.04,linestyle=dashed,dash=0.16cm 0.16cm,arrowsize=0.12cm 2.0,arrowlength=1.4,arrowinset=0.5]{->}(6.57625,-2.16)(6.59625,-2.0)(6.7981963,-1.8065884)(6.99625,-1.76)(7.194304,-1.7134116)(7.45625,-1.82)(7.61625,-1.94)
\end{pspicture} 
}
\caption{Modelisation of the $6$ polynomial vector fields of degree $3$.}
\label{figmodpoldeg3}
\end{figure}

\begin{proof}
By adapting Theorem \ref{thenmapnr} in our situation, we can write:

\begin{displaymath}
p^+(n) = \frac{1}{2n} \left[ p'(n) + p'_{e}(n) + \sum_{l \geq 2 \atop l | 2n}{\varphi(l).p'_{l}(n)} \right], \quad n \geq 2.
\end{displaymath}

Moreover, by Lemma \ref{lemvalbr}, we still have $p'(n) = p_{n}$, so we just need to determine the values of $p'_{e}(n)$ and $p'_{l}(n)$.

For that, we need to study the quotient maps and their respective pair of axial cells. As in the example of trees, and for the same argument, there are only two different pairs of axial cells: the pair vertex-face and the pair edge-face. Now we study each case separately.

\begin{enumerate}
\item[1.] the two axial cells are an edge and a face. \newline
In this case, the symmetry we consider is necessarily of order $2$, so we need to consider $2$-rooted generalized trees. Moreover, to be in this situation it is necessary that the number of edges is odd. In fact, by symmetry, the number of vertices is even, so as the number of isolated half-edge in each vertex is even, we deduce that the parity of the number of edges is equal to the parity of $n$. So that the symmetry axis is the axis edge-face, it requires that the number of edges is odd. In conclusion,

\begin{displaymath}
p'_{e}(n) = 0 \quad \text{if } n \text{ is even}.
\end{displaymath}

Suppose now that $n$ is odd, we need to distinguish two possible situations:

\begin{itemize}
\item The axial edge does not contain the roots of the map. In this case, the quotient map is a rooted generalized tree with an added, distinguished, isolated half-edge. Moreover, the axial edge may be of two types (continuous or dotted), so the number of map in this situation is equal to
\begin{displaymath}
2 \times 2\left( \frac{n-1}{2} \right) \cdot p_{(n-1)/2}.
\end{displaymath}

\item The axial edge contains the two roots of the map. In this case, the quotient map is an unrooted generalized tree to which we must be added to a vertex a distinguished isolated half-edge. Moreover, this distinguished half-edge becomes the root of the quotient map. So, by transcribing the information into a parenthesis representation, we obtain a valid bracketing in a string of $n$ elements containing the integers from $0$ to $n-1$. Moreover the element $0$ is not paired with another number, because the root of the map is an isolated half-edge, so we can erase the element $0$ to obtain a valid bracketing in the string $\{1, \ldots , n-1\}$. \newline
In conclusion, the set of quotient maps obtained in this case from rooted generalized tree having $n$ half-edges is in bijection with the set of rooted generalized tree having $(n-1)/2$ half-edges. As the axial edge may be of two types (continuous or dotted), we deduce that the number of map in this situation is equal to

\begin{displaymath}
2 \cdot p_{(n-1)/2}.
\end{displaymath}
\end{itemize}

Finally if $n$ is odd,

\begin{displaymath}
p'_{e}(n) = 2(n-1) \cdot p_{(n-1)/2} + 2 \cdot p_{(n-1)/2} = 2n \cdot p_{(n-1)/2}. 
\end{displaymath}

\item[2.] the two axial cells are a vertex and a face. \newline
In this case, the order of the symmetry is equal to $l \geq 2$, and the maps considered are $l$-rooted generalized tree. As in the previous case, it will be necessary to distinguish two situations.

\begin{itemize}
\item If $2n/l$ is even, then the quotient map is a rooted generalized tree with $2n/l$ half-edge and a distinguished vertex. By Lemma \ref{lemvalbr}, we deduce that
\begin{displaymath}
p'_{l}(n) = \sum_{k \geq 0}{p_{n/l,k} . (k+1)}.
\end{displaymath}

\item If $2n/l$ is odd, then the situation becomes more difficult in the sense that the quotient map is not necessarily a generalized tree. To achieve an enumeration in this situation, we will return to a bracketing problem by interpreting this one as follows: a bracketing representation of this problem comes from a valid bracketing with an added dot anywhere in the string. \newline
For example, from the valid bracketing $[ \, [ \, ] \, ]$ we can construct $[ \, [ \, \bullet \, ] \, ]$ as a bracketing representation of our problem. These bracketing representations are called \textbf{quasi-valid bracketing}, and the number of quasi-valid bracketing in a string of $m$ elements is denoted by $c_m$. 

Now, consider a quasi-valid bracketing in a string of $m$ elements, three situations can occur:
\begin{enumerate}
\item[-] The first element is a dot. In this case, the bracketing representation obtained by deleting the first element of the initial string is a quasi-valid bracketing.
\item[-] The first element is a left (square or round) parenthesis and its associated right parenthesis is the $j$-th element of the string, with $j$ even. In this case, the string containing the elements from $2$ to $j-1$ is a valid bracketing, and the string containing the integers from $j+1$ to $m$ is a quasi-valid bracketing.
\item[-] The first element is a left (square or round) parenthesis and its associated right parenthesis is the $k$-th element of the string, with $k$ odd. In this case, the string containing the integers from $2$ to $k-1$ is a quasi-valid bracketing, and the bracketing representation in the string containing the integers from $k+1$ to $m$ is a valid bracketing.
\end{enumerate}

In conclusion, $c_m$ satisfies the following recurrence relation
\begin{equation}
\label{eqrecurrence}
c_m = c_{m-1} + 2\sum_{2\alpha + \beta + 2 = m}{b_{2\alpha} c_{\beta}} + 2\sum_{2\alpha + \beta + 3 = m}{c_{2\alpha +1} b_{\beta}}.
\end{equation}
\end{itemize}
This completes the proof of the theorem.
\end{enumerate}

\end{proof}

\begin{cortreeg}
\begin{displaymath}
\lim_{n \rightarrow +\infty}{(p^{+}_{n})^{1/n}} = \frac{2}{5\sqrt{5} - 11} \approx 11,09.
\end{displaymath}
\end{cortreeg}

\begin{proof}
Thanks to Theorem \ref{thenpoly}, it's easy to see that
\begin{displaymath}
\frac{1}{2n} p_n \leq p^+_n \leq p_n.
\end{displaymath}
So, by using the result proved by K.Dias in \cite{Dias2}, we deduce that:
\begin{displaymath}
\lim_{n \rightarrow +\infty}{(p^+_n)^{1/n}} = \lim_{n \rightarrow +\infty}{(p_n)^{1/n}} = \frac{2}{5\sqrt{5} - 11}.
\end{displaymath}
\end{proof}

\subsection{A closed form of $c_m$.}
\label{detqm}

In order to complete the proof of Theorem \ref{mainth}, it only remains to give a closed formula for the coefficients $c_m$.  For that, we try to use the same method as before (cf Lemma \ref{lemvalbr}) by manipulating somewhat generating functions. To start, set
\begin{eqnarray*}
r_m & := & c_{2m} \\
s_m & := & c_{2m+1}
\end{eqnarray*}

Then, using the recurrence relation (\ref{eqrecurrence}) satisfied by $c_m$, we deduce

\begin{eqnarray*}
r_m \quad := \quad c_{2m} & = & c_{2m-1} + 2 \sum_{j=0}^{m-1} b_{2j} c_{2m-2j-2} + 2 \sum_{j=0}^{m-2} b_{2m-2j-3} c_{2j+1}, \\
& = & s_{m-1} + 2 \sum_{j=0}^{m-1} p_j r_{m-j-1} + 2 \sum_{j=0}^{m-2} q_{m-j-1} s_j
\end{eqnarray*}

and

\begin{eqnarray*}
s_m \quad := \quad c_{2m+1} & = & c_{2m} + 2 \sum_{j=0}^{m-1} b_{2j} c_{2m-2j-1} + 2 \sum_{j=0}^{m-1} b_{2m-2j-2} c_{2j+1}, \\
& = & c_{2m} + 4 \sum_{j=0}^{m-1} b_{2m-2j-2} c_{2j+1}, \\
& = & r_m + 4 \sum_{j=0}^{m-1} p_{m-j-1} s_j.
\end{eqnarray*}

So, let us consider the generating functions

\begin{displaymath}
p(z) = \sum_{n \geq 0} p_n z^n, \quad q(z) = \sum_{n \geq 0} q_n z^n, \quad \text{as before}
\end{displaymath}
and \begin{displaymath}
r(z) = \sum_{n \geq 0} r_n z^n, \quad s(z) = \sum_{n \geq 0} s_n z^n.
\end{displaymath}

These four functions satisfy the following equations:

\begin{equation}
\label{eqdef}
p(z) = 1 + q(z) + 2zp(z)^2,
\end{equation}
\begin{equation}
\label{eqdeg}
q(z) = zp(z) + 2zp(z)q(z),
\end{equation}
\begin{equation}
\label{eqdeq}
r(z) = 1 + zs(z) + 2zp(z)r(z) + 2zq(z)s(z),
\end{equation}
\begin{equation}
\label{eqdep}
s(z) = r(z) + 4zp(z)s(z).
\end{equation}

Notice already that the last equation can be used to determine the coefficients $r_m$ of the function $r$ in terms of those of $s$ and $p$. More precisely,
\begin{equation}
\label{eqformr}
r_m = s_m - 4 \sum_{j=0}^{m-1}{p_{m-j-1} s_j}.
\end{equation}
Thus, we only need to determine the coefficients $s_m$ of the function $s$ to complete our enumeration. For that, we note that by multiplying the equation (\ref{eqdef}) by $r$ , the equation (\ref{eqdeq}) by $p$ and by identifying the results, we obtain:

\begin{displaymath}
p + zsp + 2zqsp + 2zrp^2 = r + qr + 2zrp^2.
\end{displaymath}

So, after a simplification, and noticing that $sq = zsp + 2zsqp$, we deduce that

\begin{displaymath}
p + sq = r + rq.
\end{displaymath}

Then, as $sq = rq + 4zpsq$, we obtain

\begin{displaymath}
p + 4zpsq = r.
\end{displaymath}

Thanks to the equation (\ref{eqdeg}), and the equation (\ref{eqdep}), we can conclude that
\begin{displaymath}
s(1-2zp - 2q) = p,
\end{displaymath}

that we can still write

\begin{equation}
\label{eqcalp}
s(4\tilde{p}^2 - 6\tilde{p} + 1) = p(1-2zp), \quad \text{with } \tilde{p}=zp.
\end{equation}

So, for now, we need to determine the function $(4\tilde{p}^2 - 6\tilde{p} +1)^{-1}$ in order to give explicit formula for the coefficients of $s$. Note that using the equations (\ref{eqdef}) and (\ref{eqdeg}), one easily checks that

\begin{displaymath}
4z^2p^3 - 4zp^2 + (z+1)p -1 = 0,
\end{displaymath}

so

\begin{displaymath}
4\tilde{p}^3 = 4\tilde{p}^2 - (z+1)\tilde{p} + z.
\end{displaymath}

We can hope from this relation to determinate some (rational) coefficients $u,v,w$ such that
\begin{displaymath}
(u\tilde{p}^2 + v\tilde{p} + w)(4\tilde{p}^2 - 6\tilde{p} +1) = 1.
\end{displaymath}
After some calculations, we obtain a positive result given by the following relations:

\begin{eqnarray*}
u & = & \frac{-4(z+3)}{z^2 + 11z -1}, \\
v & = & \frac{10}{z^2 + 11z -1}, \\
w & = & \frac{-(z^2 + 5z +1)}{z^2 + 11z -1}.
\end{eqnarray*}

From this result, and the equation (\ref{eqcalp}), we obtain

\begin{displaymath}
(z^2 +11z - 1)s = p(1-2zp) \left[ -4(z+3)z^2 p^2 + 10zp - (z^2 + 5z +1) \right]
\end{displaymath}

and so after simplification
\begin{equation}
\label{eqinvp}
(z^2 +11z - 1)s = 2z(3z-1)p^2 + (2z+1)p + (z-2).
\end{equation}

However, the function $2z(3z-1)p(z)^2 + (2z+1)p(z) + (z-2)$ can be written as a formal series with coefficients in $\mathbb{R}$, so there exists coefficients $a_n \in \mathbb{R}$ such that
\begin{displaymath}
2z(3z-1)p(z)^2 + (2z+1)p(z) + (z-2) = \sum_{n \geq 0} a_n z^n.
\end{displaymath}

Then we have $a_0 = -1$, $a_1 = 4$ and for $n \geq 2$
\begin{eqnarray*}
a_n & = & 6 \sum_{k=0}^{n-2}p_k p_{n-k-2} - 2 \sum_{k=0}^{n-1} p_k p_{n-k-1} + 2p_{n-1} + p_n \\
& = & 6 \sum_{k=0}^{n-2}p_k p_{n-k-2} - 2 \sum_{k=0}^{n-2} p_k p_{n-k-1} + p_n.
\end{eqnarray*}

Finally,
\begin{displaymath}
(z^2 +11z-1)^{-1} = \sum_{n \geq 0} \left( \sum_{k=0}^{n} - \alpha^k \beta^{n-k} \right) z^n,
\end{displaymath}
where $\alpha$ and $\beta$ are the roots of the polynomial $X^2 + 11X -1$. 

We conclude from the equation (\ref{eqinvp}) and the last results that
\begin{displaymath}
c_{2m+1} := s_m = -\sum_{k=0}^{m} a_k \left( \sum_{i=0}^{m-k} \alpha^{i} \beta^{m-k-i} \right),
\end{displaymath}
and from the equation (\ref{eqformr}) that
\begin{eqnarray*}
c_{2m} := r_m & = & s_m - 4 \sum_{j=0}^{m-1}{p_{m-j-1}s_j} \\
& = & c_{2m+1} - 4 \sum_{j=0}^{m-1} c_{2j+1} p_{m-j-1}.
\end{eqnarray*}

This completes the proof of the main result of this article.

\bibliographystyle{alpha}
\bibliography{bibenumpoly} 

\end{document}